\documentclass[a4paper,11pt]{amsart}
\usepackage{amssymb,mathrsfs}
\usepackage{amsthm}
\usepackage{amsmath}
\usepackage{latexsym}
\usepackage{fancyhdr}
\usepackage{ascmac}
\usepackage{color}
\usepackage[all]{xy}
\usepackage{amscd}

\theoremstyle{plain}
\newtheorem{thm}{Theorem}[section]
\newtheorem{prop}[thm]{Proposition}

\newtheorem{lem}[thm]{Lemma}
\newtheorem{dfn}[thm]{Definition}
\newtheorem{conj}[thm]{Conjecture}
\newtheorem{rmk}[thm]{Remark}
\makeatletter

\@addtoreset{equation}{section}
\makeatother

\newcommand{\bQ}{\overline{\mathbb{Q}}}

\newcommand{\bZ}{\overline{\mathbb{Z}}}

\newcommand{\bF}{\overline{\mathbb{F}}}

\newcommand{\C}{\mathbb{C}}
\newcommand{\R}{\mathbb{R}}
\newcommand{\Q}{\mathbb{Q}}

\newcommand{\Z}{\mathbb{Z}}

\newcommand{\F}{\mathbb{F}}

\newcommand{\lra}{\longrightarrow}

\newcommand{\A}{\mathbb{A}}

\renewcommand{\O}{\mathcal{O}}

\newcommand{\br}{\overline{\rho}}

\newcommand{\diag}{{\rm diag}}
\newcommand{\ds}{\displaystyle}

\newcommand{\G}{\Gamma}

\newcommand{\uk}{\underline{k}}

\newcommand{\e}{\varepsilon}
\newcommand{\rr}{\overline{r}}
\newcommand{\ve}{\overline{\varepsilon}}
\newcommand{\bs}{\backslash}
\newcommand{\gs}{{\rm GSp}}
\newcommand{\T}{\overline{T}}
\newcommand{\op}{\overline{\psi}}
\newcommand{\oa}{\overline{a}}
\newcommand{\ob}{\overline{b}}
\newcommand{\ot}{\overline{\tau}}
\newcommand{\oc}{\overline{\chi}}
\newcommand{\wc}{\widetilde{\chi}}

\title[Serre weights for $GSp_4$ over totally real fields]
{Serre weights for $GSp_4$ over totally real fields}
\author{Takuya Yamauchi}
\thanks{The author
is partially supported by JSPS KAKENHI Grant Number (B) No.19H01778}

\address{Takuya Yamauchi \\ 
Mathematical Inst. Tohoku Univ.\\
 6-3,Aoba, Aramaki, Aoba-Ku, Sendai 980-8578, JAPAN}
\email{takuya.yamauchi.c3@tohoku.ac.jp}

\begin{document}

\maketitle

\begin{abstract}
We prove the existence of a potentially diagonalizable lift of a given automorphic mod $p$ Galois representation 
$\br:{\rm Gal}(\overline{F}/F)\lra {\rm GSp}_4(\bF_p)$ for any totally real field $F$ 
and any rational prime $p>2$ under the adequacy condition by using automorphic lifting techniques developed by Barnet-Lamb, Gee, Geraghty, and Taylor. 
As an application, when $p$ is split completely in $F$,  
we prove a variant of Serre's weight conjecture for $\br$. 
The formulation of our Serre conjecture is done by following Toby Gee's philosophy.  
Applying these results to the case when $F=\Q$ with a detailed study of 
potentially diagonalizable, crystalline lifts with some prescribed properties, we also define 
classical (naive) Serre's weights. This weight would be the minimal weight among 
possible classical weights in some sense which occur in candidates of holomorphic 
Siegel Hecke eigen cusp forms of degree 2 with levels prime to $p$. 
The main task is to construct a potentially ordinary automorphic lift for $\br$ by 
assuming only the adequacy condition. The main theorems in this paper also extend many results  
obtained by Barnet-Lamb, Gee and Geraghty \cite{BGG1} for potentially ordinary lifts and 
Gee and Geraghty for companion forms \cite{gg}. 
\end{abstract}

\tableofcontents

\section{Introduction}
Let $\mathcal{G}$ be a connected reductive group over a number field $K$ and $p$ be a prime number. 
Fix an embedding $\overline{K}\hookrightarrow \C$ and an isomorphism $\bQ_p\simeq \C$. 
Let $\iota=\iota_p:\overline{K}\lra\bQ_p$ be am embedding which is compatible with $\overline{K}\hookrightarrow \C$ and 
$\bQ_p\simeq \C$ fixed right before. 
The number theorists have expected that any geometric $p$-adic representation 
$\rho_{p,\iota}:G_K:={\rm Gal}(\overline{K}/K)\lra {}^L \mathcal{G}(\bQ_p)\rtimes G_K$ such that 
the composition with the second projection becomes the identity map ${\rm Id}_{G_K}$ 
corresponds to a $C$-algebraic automorphic representation of $G(\A_K)$ where ${}^L \mathcal{G}$ is the Langlands $L$-group of $\mathcal{G}$ 
(cf. \cite{BG}). The problem asking the modularity (or the automorphy in other words) of $\rho_{p,\iota}$ 
would be nowadays divided into two major tasks which are firstly various modular lifting theorems symbolizing as ``$R=T$" type 
theorems and 
secondary the modularity (or the automorphy) of mod $p$ Galois representations which is called Serre conjecture for $\mathcal{G}$. 
In this decade, the modular (automorphy) lifting theorems have been developed for the groups $\mathcal{G}$ so that the Taylor-Wiles-Kisin system 
can be applicable. For example, such $\mathcal{G}$'s are symplectic groups, orthogonal groups, or unitary groups but ultimately 
everything known so far is reduced to the inner form $\mathcal{G}_n$ (introduced in \cite{CHT}) of ${\rm GL}_n\times {\rm GL}_1$ for some $n\ge 1$ by passing to well-known techniques including base changes and so on. 
The readers should consult with the currently definitive article \cite{BGGT}. 
To address Serre conjecture for $\mathcal{G}$, the emergency task would be to specify all information, as much as possible, of candidates of  
automorphic forms or automorphic representations which give rise to a given mod $p$ Galois representation $\br$ as above. 
In particular, to understand all possible weights conjecturally corresponding to $\br$ is now called (generalized) Serre weight conjecture and 
it is now understood and generalized in fully general settings  in the context of several philosophies 
compatible with mod $p$ local Langlands conjecture \cite{herzig}, \cite{bdj} and \cite{GHS} (see also many works referred therein) 
and also Gee's observation so that what kind of lifts we can lift to a crystalline lift \cite{gee}. 

Let us recall Serre conjecture for $GL_2/\Q$ which is the prototype due to J-P. Serre \cite{serre}. 
Let $f$ be an elliptic Hecke eigen cusp form of level $N$ and weight $k$ with character $\e$. 
It is well-known that for each prime $p$, one can associate $f$ with a mod $p$ Galois representation $\br_{f,p}:
G_\Q:={\rm Gal}(\bQ/\Q)\lra {\rm GL}_2(\bF_p)$ by Deligne et al. 
On the other hand, for any odd, irreducible Galois representation $\br:G_\Q\lra {\rm GL}_2(\bF_p)$, in a celebrated paper \cite{serre},  
J-P. Serre defined the three invariants consisting of  
weight $k(\br)$, level $N(\br)$, and character $\e(\br)$, respectively. 
In those definitions, the most elaborated one is the weight $k(\br)$ and others are easily defined from $\br$. 
We say such $\br$ is modular if there exists an elliptic Hecke eigen cusp form $f$ such that $\br\simeq \br_{f,p}$. 
There are a number of Hecke eigen cusp forms with the same property. So it is important to specify all possible weights and
 a minimal choice among the candidates. Serre conjectured if $\br$ is modular, then one can find 
an elliptic Hecke eigen cusp form $f$ with the weight $k(\br)$  
such that $\br\simeq \br_{f,p}$. 
Serre's modularity conjecture for any odd irreducible Galois representation $\br:{\rm Gal}(\bQ/\Q)\lra {\rm GL}_2(\bF_p)$ was proved by Khare-Wintenberger \cite{kw1},\cite{kw2}. In their proof, 
the weight reduction of Edixhoven's Theorem 4.5 of \cite{Edix} played an important role. Further, 
with Theorem 3.17 of \cite{bdj}, we have known all candidates of (regular) weights of 
modular forms give rise to a given mod $p$ Galois representation $\br$ and it is revealed that $k(\br)$ is 
the ``minimal" weight in some sense.  

In this paper, we will address Serre weight conjecture for $GSp_4$ over a totally real field $F$ 
formulated along Toby Gee's philosophy. The main ingredient is to give an improvement of automorphy lifting theorems.   
We will prepare some notation to explain our main results. 
 Let $GSp_4=GSp_J$ be the symplectic similitude group in $GL_4$ associated to $J=\begin{pmatrix} 0_2& s\\-s &0_2\end{pmatrix},\ 
s=\begin{pmatrix} 0& 1\\1 &0\end{pmatrix}$  with the similitude character $\nu:GSp_4\lra GL_1$. 
Let $\pi$ be an automorphic cuspidal representation of ${\rm GSp}_4(\A_F)$ whose infinite components are all 
discrete series representations or limits of discrete series representations. We say $\pi$ is regular if 
all infinite components are discrete series representations. 

Thanks to the works \cite{taylor-thesis},\cite{taylor-low},\cite{laumon},\cite{wei1},\cite{wei} 
for $F=\Q$ and \cite{Sor}, \cite{Mok} for general $F$, for each prime $p$ and 
$\iota=\iota_p:\bQ_p\stackrel{\sim}{\lra} \C$,  
there exists a semisimple $p$-adic Galois representation 
$\rho_{\pi,\iota_p}:G_F\lra {\rm GSp}_4(\bQ_p)$ which satisfies all expected properties which will be 
explained in Section \ref{ARGSp}. The anticipated Arthur's classification of automorphic representations for $GSp_4$ 
(over any number field) 
is now available due to Gee-Ta\"ibi \cite{GeeT} and it follows from this that 
the construction of $\rho_{\pi,\iota_p}$ is now 
unconditional under the assumption explained in the second paragraph of p.472 in \cite{GeeT}.  
By choosing a suitable lattice, we have a mod $p$ Galois representation 
$\br_{\pi,\iota_p}:G_F\lra {\rm GSp}_4(\bF_p)$ and for a given mod $p$ Galois representation 
$\br:G_F\lra {\rm GSp}_4(\bF_p)$ we may ask the existence of such  a $\pi$ so that $\br\simeq \br_{\pi,\iota_p}$ 
as a representation to ${\rm GL}_4(\bF_p)$. 
We call $\br$ is modular or automorphic if such a $\pi$ exists. 

The first task is to study a suitable lift so that we can apply various automorphy lifting theorems for 
$\br$ under a mild condition. 
Then our first main result is the following:       
\begin{thm}\label{main-thm1}Let $F$ be a totally real field and $p$ be an odd prime number.    
Let $\br:G_F\lra {\rm GSp}_4(\bF_p)$ be an irreducible continuous mod $p$ Galois representation. 
Assume the followings: 
\begin{enumerate}
\item $\br|_{G_{F(\zeta_p)}}$ is irreducible and $\br(G_{F(\zeta_p)})$ is adequate;  
\item $\br$ comes from a regular cuspidal automorphic representation of 
${\rm GSp}_4(\A_F)$. 
\end{enumerate}
Then there exists a $p$-adic Galois representation $\rho:G_F\lra {\rm GSp}_4(\bQ_p)$ such that  
\begin{enumerate}
\item $\rho|_{G_{F,v}}$ is potentially diagonalizable for each finite place $v$ above $p$;
\item $\rho$ comes from a regular cuspidal automorphic representation of ${\rm GSp}_4(\A_F)$ of level potentially prime to $p$.  Further if 
 $\br$ comes from a regular cuspidal automorphic representation of 
${\rm GSp}_4(\A_F)$ of level prime to $p$, then so is $\rho$.  
\end{enumerate}
\end{thm}
This claim is known if we further assume that $\br$ comes from a potentially diagonalizable regular cuspidal automorphic 
representation. Therefore, the main novelty is to show this expected condition is guaranteed under the above condition.     
Once this kind of lifts is guaranteed to exist, we can apply Proposition 4.1.1 of \cite{BGGT} to show 
the existence of automorphic lifts with the prescribed local properties as expected in \cite{gee} and in Section 1 
(see the last few lines 
right before Section 2) of \cite{gg}. 
\begin{thm}\label{main-thm2} Let $F$ be a totally real field and $p$ be an odd prime number.    
Let $\br:G_F\lra {\rm GSp}_4(\bF_p)$ be an irreducible continuous mod $p$ Galois representation. 
Assume the followings: 
\begin{enumerate}
\item $\br|_{G_{F(\zeta_p)}}$ is irreducible and $\br(G_{F(\zeta_p)})$ is adequate;  
\item $\br$ comes from a regular cuspidal automorphic representation of 
${\rm GSp}_4(\A_F)$ of level $($resp. potentially$)$ prime to $p$.  
\end{enumerate}
Assume that for each finite place $v$ of $F$ lying over $p$, a potentially diagonalizable, crystalline lift 
$\rho_v$ of $\br|_{G_{F,v}}$ with regular Hodge-Tate weights for any embedding $F_v\hookrightarrow \bQ_p$ 
is given. 
Then there exists a regular cuspidal automorphic representation $\pi$ of ${\rm GSp}_4(\A_F)$ of level $($resp. potentially$)$ prime to $p$  
such that 
\begin{enumerate}
\item 
$\br_{\pi,\iota_p}\simeq \br$; 
\item for each finite place $v$ of $F$ lying over $p$, $\rho_{\pi,\iota_p}|_{G_{F,v}}$ 
is crystalline with the same Hodge-Tate weights as $\rho_v$. 
Further,  $\rho_{\pi,\iota_p}|_{G_{F,v}}$ connects to  
$\rho_v$ in the sense of the notation given in p.530, lines 15-22 of \cite{BGGT} or Definition 3.1.4 of \cite{BGG1}; 
\item if $\br|_{G_{F,v}}$ is ordinary for each finite place $v$ of $F$ lying over $p$, then $\pi$ can be 
ordinary at all $v$ above $p$. 
\end{enumerate}
Further, we can freely choose the types of $\pi$ among 
L-packets at infinite places. In particular, $\pi_\infty=\otimes\pi_{v|\infty}\pi_v$ can be 
the holomorphic discrete series representation.     
\end{thm}
This generalizes main results of \cite{gg},  
\cite{til&her}, and \cite{tilouine1} for companion forms. 
Our advantage is that we can also consider companion forms in non-ordinary cases (see Section \ref{companion}).  

As another application, we can study a kind of Serre's weight conjecture for 
$GSp_4$ over a totally real field in some cases 
which is formulated in Conjecture \ref{conj} according to Toby Gee's philosophy. 
To simplify the situation we consider the base change version of Serre's weight 
conjecture. As defined later 
(see Conjecture \ref{conj} and Definition \ref{Pot-Serre-weights-FL-range} for the notation), let us introduce a collection $W_{\mathcal{I}}(\br)$ of Serre weights for $\br|_{G_L}$ and a totally real field $L$ of even degree. 
Here $L$ runs over some finite set $\mathcal{I}$ with respect to $\br$. 
By definition, it never loses much information. We will also define  
a collection $W^{C_0\cup C_1}_{{\rm pd-cris},\mathcal{I}}(\br)$ of Serre weights $F(\underline{\lambda})$ such that $\underline{\lambda})$ is in the range $C_0\cup C_1$ inside $p$-restricted weights and such that t $\br|_{G_L}$ 
has a potentially diagonalizable, crystalline lift whose Hodge-Tate weights are related to 
the Serre weight $F(\underline{\lambda})$. The former set is 
defined by using automorphic forms. In contrast, the latter set is defined by a 
combinatorial way in both of (modular) representation theory  and $p$-adic Hodge theory.   
\begin{thm}\label{main-thm3} Let $F$ be a totally real field and $p$ be an odd prime number which is split completely 
in $F$.    
Let $\br:G_F\lra {\rm GSp}_4(\bF_p)$ be an irreducible continuous mod $p$ Galois representation. 
Assume the followings: 
\begin{enumerate}
\item $\br|_{G_{F(\zeta_p)}}$ is irreducible and $\br(G_{F(\zeta_p)})$ is adequate;  
\item $\br$ comes from a regular cuspidal automorphic representation of 
${\rm GSp}_4(\A_F)$ of level prime to $p$;  
\item for each finite place $v$ above $p$, $\br|_{G_{F,v}}$ is semisimple, but not irreducible and generic in the sense of 
Definition \ref{g-genericity}. 
\end{enumerate}
Then it holds that 
$$W_{ \mathcal{I}}(\br)\supset W^{C_0\cup C_1}_{{\rm pd-cris},\mathcal{I}}(\br)$$
for any non-empty set $ \mathcal{I}$ of $ \mathcal{I}^{{\rm BC}}(\br)$ 
$($see Conjecture \ref{conj} and Definition \ref{Pot-Serre-weights-FL-range} for the notation$)$.  
\end{thm}
\begin{rmk}The set $W^{C_0\cup C_1}_{{\rm pd-cris},\mathcal{I}}(\br)$ in the 
above claim is incomprehensive and different from the conjectural description of Serre weights due to 
Herzig-Tilouine \cite{til&her} $($cf. Corollaire 4.15 of \cite{til&her}$)$ where 
their description is taken over the compatibility of mod $p$ local Langlands 
correspondence for $GSp_4$.  
\end{rmk}

Let us consider the case when $F=\Q$. 
In Section \ref{Serre-weights}, we will define a triple of three integers $(k_1(\br),k_2(\br),w(\br))$ satisfying $k_1(\br)\ge k_2(\br)\ge 3$ and $w(\br)\in \Z_{\ge 0}$  
for any given mod $p$ irreducible continuous Galois representation $\br:G_\Q\lra {\rm GSp}_4(\bF_p)$. In fact, it is 
actually depending only on $\br|_{G_{\Q_p}}$ by definition. 
Then as an application of our main results, we have an analogue of a weight reduction theorem which generalizes 
Edixhoven's result (Theorem 4.5 of \cite{Edix}) to the case of $GSp_4/\Q$;
\begin{thm}\label{main-thm4} 
Let $\br:G_\Q\lra {\rm GSp}_4(\bF_p)$ be an irreducible continuous mod $p$ Galois representation. 
Assume the followings: 
\begin{enumerate}
\item $p>2$;
\item $\br|_{G_{\Q(\zeta_p)}}$ is irreducible and $\br(G_{\Q(\zeta_p)})$ is adequate;  
\item $\br$ comes from a regular cuspidal automorphic representation of ${\rm GSp}_4(\A_\Q)$ of level prime to $p$.  
\end{enumerate}
Then there exists a Hecke eigen Siegel cusp form $F$ of weight $(k_1(\br),k_2(\br))$ with the level prime to $p$ such that 
$$\br\simeq \br_{F,p}\otimes \op\ve^{w(\br)}$$
where $\op:G_{\Q}\lra \F^\times_p$ is a character of conductor prime to $p$ and $\ve$ is 
the mod $p$ cyclotomic character.  
\end{thm}
The weight $(k_1(\br),k_2(\br))$ is given by an explicit procedure explained in 
Section \ref{prescribed} and Section \ref{Serre-weights}. 
As explained in Section \ref{examples}, if 
$$\br|_{G_{\Q_p}}=
\begin{pmatrix}
\op_1\ve^{a+b} & \ot_0 & \ast & \ast \\
0 & \op_2\ve^a & \ast & \ast \\
0 & 0 & \op_1\op^{-1}_2\ve^b & -(\op_1\op_2\ve^{a+b})^{-1}\ot_0 \\
0 & 0 & 0 & 1
\end{pmatrix}
$$ where $\op_1,\op_2$ are unramified character and 
the integers are enjoying the condition $p-1\ge a>b>0$, 
then we have 
$$(k_1(\br),k_2(\br))=\left\{\begin{array}{cl}
(a+1,b+2) & \mbox{if $\tau_0$ is peu ramifi\'ee} \\
(a+1,1+2)+(p-1,p-1) & \mbox{if $\tau_0$ is tr\`es ramifi\'ee ($b=1$ by definition)}
\end{array}\right.
$$

Let us explain an application to the companion form theorem.  
Let $\br$ be in Theorem \ref{main-thm2}.
In the case of elliptic modular forms, the companion forms had been understood as 
the existence of another form which gives rise to the same mod $p$ Galois representation up to the twist by a 
power of the mod $p$ cyclotomic character. 
There also exists a beautiful explanation of such a form in terms of BGG complex (cf. \cite{Gee-com0},\cite{tilouine} and see also \cite{til&her}).  
This is now understood well as a part of the context of a prescribed modular (automorphy) lifting theorem 
(cf. \cite{Gee-com},\cite{gg}). Thanks to Theorem \ref{va-edix}, 
we don't need to assume $\br$ to be ordinary and even when it is ordinary, 
the condition (3) of Theorem 7.6.9 of \cite{gg} is unnecessary. 
Hence Theorem \ref{va-edix} includes the (extensive) companion form theorem. A detailed formulation will be discussed somewhere else for non-ordinary cases. 
Let us here state the companion form theorem only for the ordinary case:
\begin{thm}\label{companion}
Let $\br:G_\Q\lra {\rm GSp}_4(\bF_p)$ be an irreducible continuous mod $p$ Galois representation. 
Assume the followings: 
\begin{enumerate}
\item $p>2$;
\item $\br|_{G_{\Q(\zeta_p)}}$ is irreducible and $\br(G_{\Q(\zeta_p)})$ is adequate;  
\item $\br$ comes from a holomorphic Siegel cusp form $h$ on ${\rm GSp}_4(\A_\Q)$ 
of a regular weight;  
\item $\br|_{G_{\Q_p}}=\op_3 \ve^c
\begin{pmatrix}
\op_1\ve^{a+b} & \ast & \ast & \ast \\
0 & \op_2\ve^a & \ast & \ast \\
0 & 0 & \op_1\op^{-1}_2\ve^b & \ast \\
0 & 0 & 0 & 1
\end{pmatrix} $ where $\op_i,\ i=1,2,3$ are unramified characters,  
$0<b<a< p-1,\ a+b<p-1$, and $c\in \Z$.  
\end{enumerate}
Then there exists a companion form for $h$ expected in  
Section 5 of \cite{til&her}. 
\end{thm}
\begin{proof}
Only concerning is when the tr\`es ramifi\'ee occurs. 
However, the conjectural weight due to Section 5 of \cite{til&her} has been formulated by raising 
the weights by $p-1$ so that the the tr\`es ramifi\'ee class  in question is 
liftable to an ordinary crystalline class. Then the claim immediately follows from 
Theorem \ref{va-edix}.   
\end{proof}

\begin{rmk}\label{working-assumption}
Even if the case when $F=\Q$, in the course of the proofs of Theorem \ref{main-thm1} through Theorem \ref{companion}, 
we apply the main results in \cite{GeeT} and \cite{T} which have been built upon 
several assumptions in the trace formula $($see the second paragraph in p.472 of \cite{GeeT}$)$. 
Therefore, our results are as unconditional as the results in \cite{GeeT} and \cite{T}.    
\end{rmk}

To prove Theorem \ref{main-thm1} we follow the techniques developed in \cite{BGG1} as they expected 
(see ``Note added in proof",p.1578 in loc.cit.). The key ingredients are as follows; 
\begin{enumerate}
\item construction of a potentially automorphic lift of weight zero; 
\item construction of a non-ordinary automorphic lift after suitable solvable base change;  
\item well-known Harris's tensor product trick to switch 
a non-ordinary (semi-ordinary) automorphic lift to a potentially ordinary automorphic lift. 
\end{enumerate}
The first ingredient is an important step and it has been well-known for $GL_2/F$ but not for higher rank cases. 
To carry out this, we use algebraic modular forms by passing to Jacquet-Langlands correspondence after 
a suitable base change. 
For the second ingredient, we take a base change lift for a constructed automorphic representation $\pi$ of weight zero and move on the setting in \cite{Sor1} to apply 
Jacquet-Langlands correspondence. Then applying well-known analysis on algebraic modular forms 
on a reductive group which is compact modulo the center at all infinity places, we can change  
the possible ramification type of $\pi$ at all places dividing $p$ to be that  
the corresponding Weil-Deligne representations have no monodromy. 
Then Harris's tensor trick employed in \cite{BGG1} yields a desired lift. 
To restrict ourselves to consider the case of $GSp_4$ seems restrictive nowadays for some experts in comparison with 
the current developments in this field. However as Base change, Jacquet-Langlands correspondence, and 
arithmetic geometry around $GSp_4/F$,  
we fully apply these powerful tools which are still conditional depending on the expected completion of 
Arthur's classification (see \cite{GeeT},\cite{T}) and it would include some advantages rather than to consider general cases. 

%The conditions on $\br$ in the prime $p$ and the image $\br(G_{F(\zeta_p)})$ may be inevitable as far as we apply Proposition 4.4.1 of \cite{BGGT} and the adequacy of $\br(G_{F(\zeta_p)})$ (cf. \cite{GHT}). However as is done in \cite{BGG2} for $GL_2$ over a totally real field if we use the classification of 
%finite subgroups in ${\rm GSp}_4(\bF_p)$ (which is available, for example, in Theorem 3.2 of \cite{DZ}) we may 
%be able to relax those conditions. 

Regarding Theorem \ref{main-thm4}, the definition of $(k_1(\br),k_2(\br),w(\br))$ follows from 
Toby Gee's philosophy. In this context, under a good condition so called ``peu ramif\'ee" considered in 
\cite{Muller},\cite{GHLS},  
the authors there constructed crystalline lifts with certain prescribed types for a given mod $p$ local 
Galois representation $\rr:G_K\lra {\rm GL}_n(\bF_p)$ for each finite extension $K/\Q_p$. In general, constructing crystalline lifts 
becomes complicated as $n$ is large, since the extension classes show up in 
$\rr$ are intervened each other. 
This may be one reason to introduce the notion ``peu ramif\'ee" in \cite{Muller},\cite{GHLS} to carry out systematic study 
for crystalline lifts.  
To explain an idea to define $(k_1(\br),k_2(\br),w(\br))$, let us consider the original case of $GL_2/\Q$ due to Serre. 
If $\br|_{G_{\Q_p}}$ is not ``peu ramif\'ee", hence is tr\`es ramifi\'ee, then 
it is of form $\br|_{G_{\Q_p}}\simeq 
\begin{pmatrix} \ve & \ast \\ 0 &1 \end{pmatrix}$ up to a twist and the 
extension class $\ast$ comes from a non-unit element in the Galois cohomology $H^1(\Q_p,\F_p(\ve))\simeq \Q^\times_p/(\Q^\times_p)^p$ via Kummer theory. 
The shape of $\br|_{G_{\Q_p}}$ looks like the local type at $p$ of a mod $p$ Galois representation comes from 
an elliptic cusp form  $f$ of weight 2 but it can not be possible since $\ast$ is not finite flat by definition  
and therefore, $\rho_{f,p}$ is not crystalline at $p$. 
Observe that any element of $H^1(\Q_p,\F_p(\ve))$ can be liftable to a non-torsion element of  
$H^1_f(\Q_p,\O(\psi\e))$ or $H^1_f(\Q_p,\O(\psi \e^{p}))$ for the integer ring $\O=\O_E$ of 
some finite extension $E/\Q_p$  and an unramified character $\psi:G_{\Q_p}\lra \O^\times$ whose reduction modulo 
the maximal ideal of $\O$ is trivial. Here $\e$ is the $p$-adic cyclotomic character. The above extension class 
$\ast$ is liftable to $H^1_f(\Q_p,\O(\psi \e^{p}))$ for some $\psi$. 
As is well-known (cf. Proposition 3.5 of \cite{KW}), there exists an elliptic cusp form $g$ of 
weight $2+p-1=p+1$ such that $\br_{g,p}\simeq \br$. Hence $\rho_{g,p}$ is a lift of $\br$ which is crystalline at $p$.  
As defined in \cite{serre}, this is a reason why we need to raise the weight 2,
 which is naively observed,   
to $k(\br)=2+(p-1)=p+1$. 

In the case of $GSp_4$, for simplicity, we visit the following form again:
$$\br|_{G_{\Q_p}}\simeq \ve^c\op_0 \otimes
\begin{pmatrix}
\br_1 & B \\
0_2 & \op_1\ve^{a+b}\br^\ast_1 
\end{pmatrix}\subset {\rm GSp}_4(\bF_p).$$
Here  
$\br_1=\begin{pmatrix}
 \op_1\ve^{a+b} & \tau_0 \\
0 & \op_2\ve^{a}  
\end{pmatrix}:G_{\Q_p}\lra {\rm GL}_2(\bF_p)$ and 
for each $0\le i\le 2$, $\psi_i:G_{\Q_p}\lra \bF^\times_p$ is an unramified character and $0\le a,b,c \le p-2$ are integers. 
The class $B$ can be regarded as a map from $G_{\Q_p}$ to the unipotent radical $N$ of Siegel parabolic subgroup in 
$GSp_4$. By definition, $N\simeq {\rm Sym}^2{\rm St}_2$. Therefore, 
 the class of $B$ belongs to $H^1(\Q_p,(\op_1\ve^{a+b})^{-1}{\rm Sym}^2(\br_1))$. 

If the extensions $\tau_0$ and $B$ satisfy a good condition (for example ``peu ramifiee" in the sense of \cite{GHLS}) and $0<b<a$, then we can produce   
a Siegel cusp form $F$ with the level prime to $p$ of weight $(a+1,b+2)$ such that $\br_{F,\iota_p}\simeq \br$. 
The general case becomes more complicated.
According to the ramification of each extension class, as observed in 
the original case,   
we have to raise $a$ or $b$ with a multiple of $p-1$ to obtain a crystalline lift as minimal as possible 
among regular Hodge-Tate weights. Thus, $(k_1(\br),k_2(\br),w(\br))$ is defined to be the minimal element among  
all triples of non-negative integers $k_1,k_2,w$ such that $k_1-1\equiv a$ mod $p-1$, $k_2-2\equiv b$ mod $p-1$, and 
there exists a potentially diagonalizable, crystalline lift of $\br|_{G_{\Q_p}}$ with the 
regular Hodge Tate weights $\{w,w+k_2-2,w+k_1-1,w+k_1+k_2-3\}$ for some $w\in \Z$. Here the regularity condition 
means $k_1\ge k_2\ge 3$. The details are given in Section \ref{Serre-weights}.

This paper will be organized as follows. 
In Section 2 through Section 3, we recall basics of cohomological cuspidal automorphic 
representations of ${\rm GSp}_4(\A_F)$. Section 3 contains several important ingredients which would play an important role in switching the ramification types.  
Section 4,5, and 6 are devoted to proving several facts which are standard, but necessary to modify the automorphy lifting theorem in question. Section 7 is the most heavy part and a bulk of this paper so that 
the existence of a non-ordinary lift and switching it to a potentially ordinary lift are both proved. 
In the course of proving them, we borrow several ideas from \cite{BGG1},\cite{BGG2}.  
Section 8 is devoted to proving first two main theorems.   
A detailed study of potentially diagonalizable crystalline lift of mod $p$ Galois representations is 
given in Section 9. Using the results there, we will define the classical (naive) Serre weights for ${\rm GSp}_4$ in Section 10 and some examples are given there. 
In Section 11, we revisit Serre weights again and give a proof of Theorem \ref{main-thm3}.  
In Section 12, we briefly discuss a sufficient condition on adequacy. Finally, we give some comments on a generalization of our results.

\textbf{Acknowledgment.}The author would like to thank H. Atobe, W-T. Gan, T. Gee, D-R. Gullota, F. Herzig, T. Ito, Kai-Wen. Lan, 
Henry H.Kim, N. Tsuzuki, Y. Ozeki, and C. Sorensen for answering many questions. In particular, professor Tetsushi Ito 
read some parts of the article carefully and pointed out some mistakes in 
an earlier version and then it helps to remove the unnecessary condition on $p$ and to develop the contents substantially. 
Professor F. Herzig pointed out some mistakes on Serre's weight conjecture and kindly guided the author 
consulting the important article \cite{BGG-unitary}. 
A part of this work was done during the author's visiting to Max Plank Institute f\"ur mathematics.  
The author would like to thank all staffs there for kindness and incredible hospitality.

\section{Automorphic representations for ${\rm GSp}_4(\A_F)$}\label{ARGSp} 
In this section we recall basic facts of automorphic cuspidal representations for ${\rm GSp}_4(\A_F)$ and corresponding Galois representations.  
Except for Jacquet-Langlands correspondence all contents may be known for most readers and they may skip this section. 

\subsection{Automorphic representations and Galois representations}\label{AutoGal}
Let $F$ be a totally real field. For each place $v$ of $F$, let $F_v$ be the completion of $F$ along $v$.  
In this section we recall basic properties of cuspidal automorphic 
representations of ${\rm GSp}_4(\A_F)$ whose infinite components are either 
discrete series or limit of discrete series. We basically follow the notation of Mok's article \cite{Mok} and add more 
necessarily ingredients for our purpose.  

For any place $v$ of $F$, we denote by $W_{F_v}$ the Weil group of $F_v$. 
Let $m_1,m_2,w$ be integers such that $m_1>m_2\ge 0$ and $m_1+m_2\equiv w+1$ mod 2. 
For the L-parameter $\phi_{(w;m_1,m_2)}:W_\R\lra {\rm GSp}_4(\C)$ defined by 
$$\phi_{(w;m_1,m_2)}(z)=|z|^{-w}\diag\Big(\Big(\frac{z}{\overline{z}}\Big)^{\frac{m_1+m_2}{2}},
\Big(\frac{z}{\overline{z}}\Big)^{\frac{m_1-m_2}{2}},
\Big(\frac{z}{\overline{z}}\Big)^{-\frac{m_1-m_2}{2}},
\Big(\frac{z}{\overline{z}}\Big)^{-\frac{m_1+m_2}{2}}\Big)$$
and 
$$\phi_{(w;m_1,m_2)}(j)=
\left(\begin{array}{cc}
0_2 & s \\
(-1)^w s & 0_2
\end{array}
\right)
,\ s=\left(\begin{array}{cc}
0 & 1 \\
1 & 0
\end{array}
\right).$$ 
By Local Langlands correspondence the archimedean L-packet $\Pi(\phi_{(w;m_1,m_2)})$ corresponding to $\phi_{(w;m_1,m_2)}$ 
consists of two elements $\{\pi^H_{(w;m_1,m_2)},\ \pi^W_{(w;m_1,m_2)}\}$ whose 
central characters both satisfy $z\mapsto z^{-w}$ for $z\in \R^\times_{>0}$.  
These are essentially tempered unitary representations of ${\rm GSp}_4(\R)$ and tempered exactly when $w=0$. 
When $m_2\ge 1$ (resp. $m_2=0$) the representation $\pi^H_{(w;m_1,m_2)}$ is called a (resp. limit of) discrete series representation of 
minimal $K$-type $\uk=(k_1,k_2):=(m_1+1,m_2+2)$ which corresponds to an algebraic representation 
$V_{\uk}:={\rm Sym}^{k_1-k_2}{\rm St}_2\otimes \det^{k_2}{\rm St}_2$ of $K_\C={\rm GL}_2(\C)$. 
Here $K$ is the maximal compact subgroup of ${\rm Sp}_4(\R)$. 
The representation $\pi^H_{(w;m_1,m_2)}$ is called if $m_2\ge 1$ (resp. $m_2=0$) a (resp. limit of ) discrete series representation of minimal $K$-type $V_{(m_1+1,-m_2)}$. 

Fix an integer $w$.  
Let $\pi=\otimes'_{v}\pi_v$ be an automorphic cuspidal representation of ${\rm GSp}_4(\A_F)$ such that 
for each infinite place $v$, $\pi_v$ has L-parameter $\varphi_{(w;m_{1,v},m_{2,v})}$ with 
the parity condition $m_{1,v}+m_{2,v}\equiv w+1$ mod 2. Let ${\rm Ram}(\pi)$ be the set of all 
finite places of which $\pi_v$ is ramified. 
Thanks to \cite{Mok} with \cite{GeeT} we can attach $\pi$ with Galois representations:
\begin{thm}\label{gal}(cf. Theorem 3.1 and Remark 3.3, Theorem 1.1 of \cite{Mok} with Theorem C of \cite{Weiss}) 
Assume that $\pi$ is neither CAP nor endoscopic. 
For each prime $p$ and $\iota_p:\bQ_p\stackrel{\sim}{\lra} \C$ 
there exists a continuous, semisimple Galois representation 
$\rho_{\pi,\iota_p}:G_F\lra {\rm GSp_4}(\bQ_p)$ such that 
\begin{enumerate}
\item $\nu\circ \rho_{\pi,\iota_p}(c_\infty)=-1$ for any complex conjugation $c_\infty$ in $G_F$. 

\item $\rho_{\pi,\iota_p}$ is unramified for all finite places which do not belong to ${\rm Ram}(\pi)\cup\{v|p\}$;

\item for each finite place $v$ not lying over $p$, the local-global compatibility holds:
$${\rm WD}(\rho_{\pi,\iota_p}|_{G_{F_v}})^{F-{\rm ss}}\simeq {\rm rec}^{{\rm GT}}_v(\pi_v\otimes |\nu|^{-\frac{3}{2}})$$
with respect to $\iota_p$ where ${\rm rec}^{{\rm GT}}_v$ stands for the local Langlands correspondence 
constructed by Gan-Takeda \cite{GT};

\item for each $v|p$ and an embedding $\sigma:F_v\hookrightarrow \bQ_p$, there is a unique embedding 
$v_\sigma:F\hookrightarrow \C$ such that $\iota_p\circ \sigma|_F=v_\sigma$. 
Then the representation $\rho_{\pi,\iota_p}|_{G_{F_v}}$ is Hodge-Tate of weights 
$$HT_\sigma(\rho_{\pi,\iota_p}|_{G_{F_v}})=\{\delta_{v_\sigma},\delta_{v_\sigma}+m_{2,v_\sigma},
\delta_{v_\sigma}+m_{1,v_\sigma},\delta_{v_\sigma}+m_{2,v_\sigma}+m_{1,v_\sigma} \}$$
where $\delta_{v_\sigma}=\ds\frac{1}{2}(w+3-m_{1,v_\sigma}-m_{2,v_\sigma})$. 
If we assume either $\pi_v$ is discrete series for all infinite places $v$ or 
genuine, then $\rho_{\pi,\iota_p}|_{G_{F_v}}$ is crystalline and the local-global compatibility also holds up to 
semi-simplification.  
\end{enumerate}
\end{thm}

\begin{dfn}\label{def-odd}Let $F$ be a totally real field and $K$ be a field. Let $\rho:G_F\lra {\rm GSp}_4(K)$ be a 
representation. The natural inclusion ${\rm GSp}_4(K)\subset {\rm GL}_4(K)$ induces the action of $G_F$ on $K^4$. 
We say $\rho$ totally odd if 
the dimension of the subspace of $K^4$ fixed by $\nu\circ \rho(c_\infty)$ is two 
for each complex conjugation $c_\infty$ in $G_F$. 
It is equivalent to the condition that $\nu\circ \rho(c_\infty)=-1$ unless the characteristic of $K$ is two.  
\end{dfn} 
%\begin{lem}\label{abs-irred}Let $n\ge 1$ be an integer. Let $G$ be a group and $\rho:G\lra {\rm GL}_{n}(K)$ be an irreducible representation over a field $K$. 
%Assume that there exists an element $c$ in $G$ such that the subspace of $K^{2n}$ fixed by $\rho(c)$ is non-trivial and 
%its dimension is strictly less than $n$. 
%Then $\rho$ is absolutely irreducible. 
%\end{lem}
%\begin{proof}Put $V=K^{n}$. If not, there exists a $G$-invariant non-trivial proper subspace $W$ of $V\otimes \overline{K}$. 
%Pick a non-zero element $v\in V^{\rho(c)}$. Assume that $v\in W$. 
%By irreducibility over $K$, we have $V=\langle gv\ |\ g\in G \rangle_K$, but     
%$V\otimes_K\overline{K}\subset W$ means that $V$ is of dimension strictly less than $n$. This give a contradiction. 
% 
%If $v\not\in W$, then it is non-trivial in $V\otimes_K\overline{K}/W$. By taking duality, the same argument above 
%works again.  
%\end{proof}
\begin{prop}\label{odd-irred} Keep the notation in Theorem \ref{gal}.  
Suppose that $p$ is odd. It holds that 
\begin{enumerate}
\item there exists a finite extension $E/\Q_p$ in $\bQ_p$ such that ${\rm Im}(\rho_{\pi,\iota_p})\subset {\rm GSp}_4(E)$;
\item $\rho_{\pi,\iota_p}$ is totally odd; 
\item if $\br_{\pi,\iota_p}:G_{F}\lra {\rm GSp}_4(\F_q)$  is irreducible over 
some finite extension $\F_q/\F_p$ and its projective image is not Dihedral, then it is also totally odd and absolutely irreducible. 
\end{enumerate}
\end{prop}
\begin{proof}The first claim follows from the explanation in p.654 of \cite{Sor}. 
The second claim is well-known (see comments in front of Proposition 1 
of \cite{Taylor-complex}).  
For the second claim, by assumption, the isomorphism classes of $\br_{\pi,\iota_p}$ is independent of any choices of symplectic 
lattices in $\bQ^{\oplus 4}_p$. Then one can find a symplectic $\bQ_p$-lattice $L$ such that $\rho_{\pi,\iota_p}$ is 
given by 
$\left(\begin{array}{cc}
s & 0_2 \\
0_2 & s
\end{array}\right)
$ where $s=\left(\begin{array}{cc}
0 & 1 \\
1 & 0
\end{array}\right)$. The half of the second claim follows from this. 
The absolute irreducibility now follows by using a classification of all semisimple subgroups in ${\rm PGSp}_4(\F_q)$ for any finite extension $\F_q$  of $\F_p$ (cf. Theorem 3.2 of \cite{DZ}). 
It is carried out by using tedious computation and therefore, the detailed are omitted.  
\end{proof}
Let $B$ be the upper Borel subgroup of $GSp_4$, $P$ the Siegel parabolic subgroup, and $Q$ the Klingen parabolic subgroup (see \cite{RS}) 
\begin{dfn}\label{auto-gal} 
\begin{enumerate}
\item
Let $\rho:G_F\lra {\rm GSp}_4(\bQ_p)$ be an irreducible $p$-adic Galois representation. 
We say $\rho$ is automorphic if there exists a cuspidal automorphic representation $\pi$ of 
${\rm GSp}_4(\A_F)$ with $\pi_v$ a discrete series representation for any $v|\infty$  such that 
$\rho\simeq\rho_{\pi,\iota_p}$ as a representation which takes the values in ${\rm GL}_4(\bQ_p)$. 
By definition, if $\rho$ is automorphic, then it is totally odd.  
\item  Let $\br:G_F\lra {\rm GSp}_4(\bF_p)$ be an irreducible mod $p$  Galois representation. 
We say $\br$ is automorphic if there exists a cuspidal automorphic representation $\pi$ of 
${\rm GSp}_4(\A_F)$ with $\pi_v$ a discrete series representation for any $v|\infty$  such that 
$\br\simeq \br_{\pi,\iota_p}$ as a representation which takes the values in ${\rm GL}_4(\bF_p)$. 
\item Let $\rho:G_F\lra {\rm GSp}_4(\bQ_p)$ be a $p$-adic Galois representation. 
For a finite place $v$ of $F$, $\rho$ is said to be 
\begin{enumerate}
\item Borel ordinary at $v$ if ${\rm Im}(\rho|_{G_{F_v}})\subset B(\bQ_p)$; 
\item Siegel ordinary at $v$ if ${\rm Im}(\rho|_{G_{F_v}})\subset Q(\bQ_p)$;  
\item Klingen ordinary at $v$ if ${\rm Im}(\rho|_{G_{F_v}})\subset P(\bQ_p)$,  
\end{enumerate}
after choosing a basis of $\rho$ suitably. 
It is similarly defined for $\br:G_F\lra {\rm GSp}_4(\bF_p)$. 
Notice that the $L$-group of $GSp_4$ is (exceptionally) isomorphic to $GSp_4$ again. 
Under this isomorphism, $P$ and $Q$ are interchanged. 
\item Let $\pi$ and $\rho_{\pi,\iota_p}$ be as in Theorem \ref{gal}.  
For a finite place $v$ of $F$, $\pi$ is said to be 
\begin{enumerate}
\item Borel ordinary at $v$ if $\rho_{\pi,\iota_p}$ is Borel ordinary at $v$; 
\item Siegel ordinary at $v$ if $\rho_{\pi,\iota_p}$ is Klingen ordinary at $v$;  
\item Klingen ordinary at $v$ if $\rho_{\pi,\iota_p}$ is Siegel ordinary at $v$. 
\end{enumerate}
\end{enumerate}
\end{dfn}

\subsection{Base change}\label{BC} 
Let $\pi$ be a regular cuspidal representation of ${\rm GSp}_4(\A_F)$. 
Assume that $\pi$ is neither CAP nor endoscopic. 
By \cite{GeeT} there exists a unique globally generic representation $\pi'$ such that 
the finite part of $\pi'$ is equivalent to one of $\pi$. The representation $\pi'$ can be transferred to a cuspidal representation $\Pi$ of ${\rm GL}_4(\A_F)$. 
If there exists a solvable extension $F'$ of $F$ such that Arthur-Clozel Base change $BC_{F'}(\pi)$ is not cuspidal, then 
there exists a totally real quadratic extension $E/F$ and a cuspidal representation $\tau$ of ${\rm GL_2}(\A_E)$ 
such that $\Pi={\rm AI}^E_F(\tau)$ where ${\rm AI}^E_F$ stands for the automorphic induction. 
Note that in general $E$ could be imaginary quadratic extension of $F$ but the regularity condition on $\pi$ implies $E$ is totally real.
Observing the conductor of $\Pi$ there are only finitely many possibilities of $E$. 
Let $K_\Pi$ be the composite field of all such $E$'s for $\Pi$ and put $K_\Pi=F$ if no such a field exists. 
Note that $K_\Pi$ is a finite extension since the conductor of $E$ divides one of $\Pi$. 
By taking a strong backward lift to $GSp_4$ and switching the infinite types 
(using \cite{GeeT} again)
we have a base change theorem:
\begin{thm}\label{bc}For any solvable extension $F'$ which is linearly disjoint from $K_\Pi$ over $F$ 
there exists a base change $\pi_{F'}$ to ${\rm GSp}_4(\A_{F'})$ of $\pi$. 
\end{thm}
\begin{proof}This follows from Proposition 2.4 of \cite{Mok}. 
\end{proof}

\subsection{Jacquet-Langlands correspondence}\label{JL} 
Throughout this subsection, we assume that $d:=[F:\Q]$ is even. 
Let $B$ be a definite quaternion algebra such that 
$$B\otimes_F\R\simeq \mathbb{H}^d,\ B\otimes F_v\simeq_F M_2(F_v)$$
for all finite places $v$. Here $\mathbb{H}=\Big(\ds\frac{-1,-1}{\R}\Big)$ is Hamilton's quaternions. 
We denote by $\overline{x}$ the usual conjugation of an element $x$ of $B$ and it is naturally extended on ${\rm GL}_2(B)$. 
For $g\in {\rm GL}_2(B)$ we write ${}^\ast g={}^t \overline{g}$ where ``$t$" stands for the transpose. 
Let us consider an algebraic group $G_B$ associating 
any $F$-algebra $R$ with 
$$G_B(R):=\{g\in {\rm GL}_2(B\otimes_FR)\ |\ {}^\ast g g=\mu(g)I_2,\ \mu(g)\in R^\times \}.$$
This is an inner form of $GSp_4/F$ which is compact modulo the center at infinities and by assumption we may fix an isomorphism 
\begin{equation}\label{isom}
\iota_v:G_B(F_v)\stackrel{\sim}{\lra} {\rm GSp}_4(F_v). 
\end{equation}
for any finite place $v$. 
For integers $a,b,c$ satisfying $a\ge b\ge 0$  we denote by 
$\xi_{a,b,c}=\nu^{c}\otimes \rho_{a,b}$ an algebraic representation of ${\rm GSp}_4(\R)$ where 
$\rho_{a,b}$ is a unique irreducible representation of the highest weight $(a,b)$ with respect to the restriction to ${\rm Sp}_4(\R)$. 
Note that the central character of $\xi_{a,b,c}$ is given by $z \mapsto z^{2c+a+b}$ for $z\in \R^\times$. 

For integers $m_1,m_2,w$ satisfying $m_1>m_2>0$ and $m_1+m_2\equiv w+1$ mod 2 the dual of 
$\xi_{\delta,m_1-2,m_2-1}$ has the same central character of $\pi^H_{(w;m_1,m_2)}$ or $\pi^W_{(w;m_1,m_2)}$ 
where $\delta=\frac{1}{2}(w+3-m_1-m_2)\in \Z$. 
For any automorphic representation $\Pi$ of ${\rm GSp}_4(\A_F)$ or $G_B(\A_F)$ we denote by $\omega_\Pi$ the central 
character of $\Pi$ and $\Pi_f$ the finite part of $\Pi$. 
For $\xi_{\delta,m_1-2,m_2-1}$ we can associate a unique irreducible algebraic representation 
$\xi^{{\rm JL}}_{\delta,m_1-2,m_2-1}$ of $G_B(\R)$ whose complexification is isomorphic to those of 
$\xi_{\delta,m_1-2,m_2-1}$. It is well known (cf.  Section 3.2 of \cite{KWY}) that 
\begin{equation}\label{dim-al}
{\rm dim}\hspace{0.5mm}\xi^{{\rm JL}}_{\delta,m_1-2,m_2-1}={\rm dim}\hspace{0.5mm}\xi_{\delta,m_1-2,m_2-1}=\frac{1}{6}(m_1-m_2)(m_1+m_2)m_1m_2. 
\end{equation}

According to \cite{Sor1}, we say a cuspidal representation of ${\rm GSp}_4(\A_F)$ or $G_B(\A_F)$ is stable if it is neither CAP nor endoscopic (see Definition 1.1 of \cite{Sor1} for 
endoscopic representations). Here, according to Definition 3.3, p.102 of \cite{Gan}, a cuspidal representation $\pi$ of 
${\rm GSp}_4(\A_F)$ or $G_B(\A_F)$ is said to be a CAP representation if it is weakly equivalent to the irreducible 
constituent of an induced representation ${\rm Ind}^{{\rm GSp}_4(\A_F)}_{P(\A_F)}\tau$, with 
$\tau$ a cuspidal representation of the Levi factor of a rational parabolic subgroup $P$ of $GSp_4$. 
The following theorem is stated as Theorem B in \cite{Sor1}:
\begin{thm}\label{JL}Fix a continuous character $\omega:F^\times\backslash \A^\times_F\lra \C^\times$. 
Keep the notation in Theorem \ref{gal}. Assume that $m_{2,v}>0$ for all infinite place $v$. 
For each infinite place $v$ of $F$, fix a member $\Pi_v$ in $\Pi(\phi_{(w;m_{1,v},m_{2,v})})$. 
Then there exists one-to-one correspondence between the following sets:
\begin{enumerate}
\item stable, tempered automorphic representations $\pi'$ of $G_B(\A_F)$ such that $\omega_{\pi'}=\omega$ and $\pi'_v=\xi^{{\rm JL}}_{\delta_v,m_{1,v}-1,m_{2,v}-2}$ 
for each infinite place $v$ and 
\item stable, tempered automorphic representations $\pi$ of ${\rm GSp}_4(\A_F)$ such that $\omega_\pi=\omega$ and 
$\pi_v=\Pi_v$ 
for each infinite place $v$. 
\end{enumerate}  
The correspondence takes $\pi'\mapsto \pi'_f\otimes \ds\bigotimes_{v|\infty}\Pi_v$ and $\pi\mapsto \pi_f\otimes \ds\bigotimes_{v|\infty}\xi^{{\rm JL}}_{\delta_v,m_{1,v}-1,m_{2,v}-2}$ under the isomorphism (\ref{isom}). 
\end{thm}

To end this section, we discuss tempered-ness and CAP or endoscopic representations of ${\rm GSp}_4(\A_F)$ and 
$G_B(\A_F)$ respectively. 
As for CAP representations for ${\rm GSp}_4(\A_F)$ or $G_B(\A_F)$,  there are three types of CAP representations associated to Siegel parabolic subgroup, 
Klingen parabolic subgroup, and Borel subgroup respectively. Let us first consider the classification of CAP 
representations. For simplicity, only in this paper, we say a CAP representation associated to 
Siegel parabolic subgroup a Siegel CAP representation and similarly we do the same for other cases.  

Let $\pi=\otimes'_{v}\pi_v$ be a regular cuspidal representation of ${\rm GSp}_4(\A_F)$ such that 
for each infinite place $v$, $\pi_v$ has the L-parameter $\phi_{(w;m_{1,v},m_{2,v})}\ m_{1,v}>m_{2,v}>0$ with 
the parity condition $m_{1,v}+m_{2,v}\equiv w+1$ mod 2 for a fixed integer $w$. 
The regularity means $m_{1,v}+1\ge m_{1,v}+2\ge 3$. 
Suppose that $\pi$ is a CAP representation. 
It is well known by Howe, Piatetski-Shapiro \cite{PS} and Soudry \cite{Soudry} that $\pi$ can be realized as a theta lift. 
Observe that $\phi_{(w;m_{1,v},m_{2,v})}$ is the direct sum of 2-dimensional irreducible representation of 
$W_\R$ which are not isomorphic each other. 
Therefore, $\pi$ can not be any Borel CAP representations by Theorem 7.4.1, p.512 of \cite{GeeT}. 
If $\pi$ is a Klingen CAP representation, then by Theorem 7.4.1 of loc.cit. again, 
$m_{1,v}=m_{2,v}$ for each infinite place $v$ of $F$ 
so that it contradicts the regularity.  We can also get around using Arthur classification as above by 
computing possible minimal $K$-types between the theta correspondence for $(GO(2),GSp_4)$ 
which are done by applying the results in \cite{KW}. 
Therefore, if $\pi$ is CAP, then it has to be a Siegel CAP representation. 
By Theorem 7.4.1 of loc.cit. again, we have $k_v:=m_{1,v}+1=m_{2,v}+2>0$ for each finite place $v$ and 
there exists a regular cuspidal automorphic representation $\tau$ of ${\rm GL}_2(\A_F)$ such that 
$\phi_v(\tau_v)[2]$ is isomorphic to $\phi_{(w;m_{1,v},m_{2,v})}$ for each finite place $v$ of $F$ where 
$\phi_v(\tau_v)$ is the local Langlands parameter of $\tau$ at $v$. 
If $\pi$ is endoscopic, by \cite{R}, it is obtained, as a theta lift, from a pair $(\tau_1,\tau_2)$ of cuspidal 
representations of ${\rm GL}_2(\A_F)$ such that $\phi_v(\tau_{1,v})\oplus \phi_v(\tau_{2,v})$ 
is isomorphic to $\phi_{(w;m_{1,v},m_{2,v})}$ for each finite place $v$ of $F$. 
In either the CAP case or the endoscopic case, one can formally attach a reducible $p$-adic 
Galois representation $\rho_{\pi,\iota_p}:G_F\lra {\rm GSp}_4(\bQ_p)$ to $\pi$ for each $\iota_p:\bQ_p\simeq \C$. 
Note that $\rho_{\pi,\iota_p}$ is a priori constructed as a representation to ${\rm GL}_4(\bQ_p)$ but we 
can easily adjust it having the image to ${\rm GSp}_4(\bQ_p)$. 
Summing up, we have proved the following:
\begin{thm}\label{cla-GSp4}Let $\pi$ be a regular cuspidal representation of ${\rm GSp}_4(\A_F)$ and $\rho_{\pi,\iota_p}$ 
be the corresponding $p$-adic Galois representation. It holds the followings:
\begin{enumerate}
\item $\pi$ is non-tempered if and only if $\pi$ is a Siegel CAP representation. In this case, $\rho_{\pi,\iota_p}$ 
is the direct sum of two characters and a 2-dimensional representation attached to a regular cuspidal representation of 
${\rm GL}_2(\A_F)$;
\item $\pi$ is tempered and $\rho_{\pi,\iota_p}$ is reducible if and only if $\pi$ is endoscopic;
\item $\pi$ is stable and tempered if $\rho_{\pi,\iota_p}$ is irreducible. 
\end{enumerate}
In particular, if $\br_{\pi,\iota_p}$ is irreducible, then $\pi$ is stable and tempered. 
\end{thm}
 
Next we consider when $\Pi$ is a regular algebraic cuspidal automorphic representation of $G_B(\A_F)$ such that 
for each infinite place $v$ of $F$, $\Pi_v$ has the L-parameter $\phi_{(w;m_{1,v},m_{2,v})}\ m_{1,v}>m_{2,v}>0$ with 
the parity condition $m_{1,v}+m_{2,v}\equiv w+1$ mod 2 for a fixed integer $w$. 
In this case, the L-packet at such a $v$ is singleton.  
Recall that $H_B$ is the kernel of the similitude character of $G_B(\A_F)$ so that $H_B$ is an inner form of 
$Sp_4$ which is compact at infinity. By the argument after the third paragraph in page 516 of \cite{GeeT}, 
any component of the restriction $\Pi|_{H_B}$ to $H_B$ is a discrete automorphic representation of $H_B(\A_F)$. Pick any 
component $\Pi'$ of $\Pi|_{H_B}$. 
It is easy to see that the L-parameter of $\Pi'$ at each infinite place $v$ of $F$ is 
isomorphic to 
\begin{equation}\label{so5}
1\oplus {\rm Ind}^{W_\R}_{W_\C}(z/\overline{z})^{m_{1,v}}\oplus {\rm Ind}^{W_\R}_{W_\C}(z/\overline{z})^{m_{2,v}}
\end{equation}
which takes the values in ${\rm SO}(5)(\C)\subset {\rm GL}_5(\C)$ which is given as  
the composition of $\phi_{(w;m_{1,v},m_{2,v})}$ with the projection ${\rm GSp}_4(\C)\lra 
{\rm PGSp}_4(\C)\simeq {\rm SO}(5)(\C)$ (see A.7, p.286 of \cite{RS} for the isomorphism). 
We will seek the possible global 
Arthur parameters for $\Pi'$ in terms of Arthur parameter described in Theorem 4.0.1 of \cite{T}. 
What we have known for $\Pi'$ at hand is the L-parameter (\ref{so5}) but it is different from Arthur parameter 
if $\Pi'$ is non-tempered. However, notice that we can compare the infinitesimal character for both parameters 
(cf. Section 3.6 of \cite{CR}). 
The infinitesimal character of (\ref{so5}) is parametrized by the quintuple 
$$(m_{1,v},m_{2,v},0,-m_{2,v},-m_{1,v}).$$ 
Notice that the integers are distinct to each other since 
$m_{1,v}>m_{2,v}>0$ by regularity. 
Therefore, the Arthur parameter of $\Pi'$ does not contain any of the following parameters 
$$\chi_1[2d],\ \chi_2\boxplus\chi_3$$
where $d\in \Z_{\>0}$ and $\chi_1,\chi_2,\chi_3$ are 
characters of  $F^\times\bs {\rm GL}_1(\A_F)$ with $\chi^2_i=1$ for $i=1,2,3$.  
Then we have the possible global 
Arthur parameters of $\Pi'$ as follows:
\begin{enumerate}
\item $1[5]$; 
\item $1\boxplus \pi^{(1)}_2\boxplus \pi^{(2)}_2$, where  
for each $i=1,2$, $\pi^{(i)}_2$ is a cuspidal automorphic representation of ${\rm GL}_2(\A_F)$ with the trivial central 
character and further $\pi^{(1)}_2$ and $\pi^{(2)}_2$ are not isomorphic to each other;  
\item $1\boxplus \pi_2[2]$, where $\pi_2$ is a cuspidal automorphic representation of ${\rm GL}_2(\A_F)$ 
with the trivial central character; 
\item $\pi_2 \boxplus 1[3]$, where   
$\pi_2$ is a cuspidal automorphic representation of ${\rm GL}_2(\A_F)$ with the trivial central character;
\item  $\pi_2\boxplus \pi_3$, where for each $j=2,3$,   
$\pi_j$ is a cuspidal automorphic representation of ${\rm GL}_j(\A_F)$ with the trivial central character and 
further $\pi_3$ is orthogonal;
\item $\chi\boxplus \pi_4$, where $\chi$ is a character of  $F^\times\bs {\rm GL}_1(\A_F)$ with $\chi^2=1$ and 
$\pi_4$ is a cuspidal automorphic representation of ${\rm GL}_4(\A_F)$ which is symplectic and it has the 
central character $\omega_{\pi_4}=\chi$;
\item $\pi_5$, where $\pi_5$ is a cuspidal, orthogonal automorphic representation of ${\rm GL}_5(\A_F)$ with 
the trivial central character.   
\end{enumerate}
By matching of the infinitesimal characters, we see easily that all cuspidal representations 
listed above are regular algebraic. It follows from this that for each cuspidal representation, 
by Theorem 2.1.1 of \cite{BGGT}, one can attach a semisimple 
continuous $p$-adic Galois representation for $\iota_p$. 
Therefore, one can attach a semisimple $p$-adic Galois representation 
$\rho_{\Pi',\iota_p}:G_F\lra {\rm GO}(5)(\bQ_p)$ with $\Pi'$ for $\iota_p$ 
according to each parameter above. 
Let $c:{\rm GO}(5)(\bQ_p)\lra \bQ^\times_p$ be 
the similitude character. The surjection ${\rm GO}(5)(\bQ_p)\lra {\rm SO}(5)(\bQ_p),\ A\mapsto \mu(A)^{-1}A,\ \mu(A)=\det(A)c(A)^{-2}$ 
induces ${\rm GO}(5)(\bQ_p)\simeq {\rm SO}(5)(\bQ_p)\times \bQ^\times_p$ (we do not have a similar isomorphism for 
even orthogonal groups with the similitude). 
Let $p_1$ be the projection onto ${\rm SO}(5)(\bQ_p)$. 
Then we have a continuous representation $p_1\circ \rho_{\Pi',\iota_p}:G_F\lra {\rm SO}(5)(\bQ_p)\simeq {\rm PGSp}_4(\bQ_p)$ by using 
the exceptional isomorphism. By Proposition 2.1.4 of \cite{P}, there exists a continuous, semisimple lift 
$\rho_{\Pi,\iota_p}:G_F\lra {\rm GSp}_4(\bQ_p)$ of $p_1\circ \rho_{\Pi',\iota_p}$ via the 
surjection ${\rm GSp}_4(\bQ_p)\lra {\rm PGSp}_4(\bQ_p)$. 
There may be another lifts but they differ by twists of characters. It is easy to see that 
the adjoint representation $p_1\circ \rho_{\Pi,\iota_p})$ 
is isomorphic to $p_1\circ \rho_{\Pi',\iota_p}$ via the exceptional isomorphism.  

In what follows, the semi-simplification of the reduction $\rho_{\Pi,\iota_p}$ is 
irreducible and in particular, $\rho_{\Pi,\iota_p}$ is irreducible. 
This condition is independent of any choice of $\Pi'$ and any choice of the lifts of $\rho_{\Pi',\iota_p}$. 
Let us observe each parameter above individually.  

First we consider the parameter $1[5]$. This corresponds to the parameter $\phi_{(w;2,1)}$ at infinity such that 
$\Pi_v$ is trivial at infinite place.  
Then $\rho_{\Pi,\iota_p}$ has to be the product of characters and this case is excluded by the 
irreducibility.

%If $\Pi$ comes from a regular cuspidal representation $\pi$ of ${\rm GSp}_4(\A_F)$. 
%Such a parameter implies that $\pi$ is  the convolution product of regular cuspidal representations $\tau_1,\tau_2$ of 
%${\rm GL}_2(\A_F)$ such that they are dihedral. If we write  $\tau_i={\rm AI}^{K_i}_F(\varphi_i)$ for 
%a CM extension $K_i/F$ and a Hecke character $\varphi_i$ of ${\rm GL}_1(\A_{K_i})$ for each $i=1,2$, then 
%$\pi^{(i)}_2={\rm AI}^{K_i}_F(\varphi^2_i)$. 
Next we consider the parameter $1\boxplus \pi^{(1)}_2\boxplus \pi^{(2)}_2$. 
Let $\rho_{i,\iota_p}:G_F\lra {\rm GL}_2(\bQ_p)$ be the 
associated $p$-adic representation of $\pi^{(i)}_2$ for each $i=1,2$. 
By using the classification of $p$-adic Lie algebra associated to 
${\rm Im}(\rho_{\Pi,\iota_p})$ (cf. Proposition 4.5 of \cite{CG}), the $p$-adic Lie algebra of 
$\rho_{\Pi,\iota_p}$ has to be trivial since $p_1\circ \rho_{\Pi,\iota_p}\simeq \rho_{\Pi',\iota_p}$ as mentioned. 
Then the representation $\rho_{\Pi,\iota_p}$ has to be either of 
\begin{enumerate}
\item the direct sum of two characters  and an irreducible induced 2-dimensional representation of $G_F$, 
\item the direct sum of two irreducible induced 2-dimensional representation of $G_F$, or  
\item an induced representation from an induced 2-dimensional representation of $G_F$.   
\end{enumerate}
In either case, $\rho_{\Pi,\iota_p}$ is reducible and this case is also excluded by irreducibility. 
Similar observation is applied to the parameter $\pi_2 \boxplus 1[3]$ and it turns out that 
$\pi_2$ is orthogonal. Hence $\rho_{\Pi,\iota_p}$ is the direct sum of two characters $\chi_1,\chi_2$ 
and an irreducible induced 2-dimensional representation of $G_F$. 
This case is also excluded by irreducibility.

Next we consider the parameter  $1\boxplus \pi_3[2]$ which corresponds to a Siegel CAP representation $SK(\tau)$ 
(which is so called a Saito-Kurokawa lift) of $H_B(\A_F)$. 
It is classified by 
Theorem 7.1 and Proposition 6.2 of \cite{Gan} where $\tau$ is the cuspidal data on the Levi factor of Siegel parabolic. 
(In this case, by Proposition 6.2 of \cite{Gan},  $\tau$ is regular algebraic and this fact matches with the 
above classification.)   
In this case e see that $\rho_{\Pi,\iota_p}$ is  
the direct sum of two characters and an irreducible 2-dimensional representation of $G_F$ and 
this case is also excluded by irreducibility. 
 
Next we consider the parameter $\pi_2\boxplus \pi_3$. 
In this case, we can also attached a semisimple continuous Galois representation 
$\rho_{\Pi,\iota_p}:G_F\lra {\rm GSp}_4(\bQ_p)$ which is the tensor product of 
a 2-dimensional dihedral representation of $G_F$ and a 2-dimensional non-dihedral representation of $G_F$. 
Here ``dihedral" means ``induced".

Next we consider the parameter $\chi\boxplus \pi_4$. Observing the Langlands parameter of $\pi_4$ at infinity we 
see that $\pi_4$ comes from  the convolution product or the twisted convolution associated to 
a cuspidal representation of ${\rm GO}(2,2)(\A_F)$ or ${\rm GO}(4)(\A_F)$ respectively. 
Note that the case of ${\rm GO}(3,1)(\A_F)$ is excluded by the regularity condition at infinity.  
If $\pi_4=\tau_1\boxtimes \tau_2$ for a regular cuspidal representation $(\tau_1,\tau_2)$ of ${\rm GO}(2,2)(\A_F)$, then 
$\rho_{\Pi,\iota_p}$ is, up to twist, equivalent to $\rho_{\tau_1,\iota_p}\oplus \rho_{\tau_1,\iota_p}$. This happens exactly when 
$\Pi$ is endoscopic. Hence the irreducibility excludes the endoscopic case. 
If $\pi_4$ comes from ${\rm GO}(4)(\A_F)$, then $\rho_{\Pi,\iota_p}$ is an irreducible induced representation of 
2-dimensional representation of $G_K$ for some totally real quadratic extension $K/F$.

Finally, we consider the parameter $\pi_5$. 
In this case, we can also attached a semisimple continuous Galois representation 
$\rho_{\Pi,\iota_p}:G_F\lra {\rm GSp}_4(\bQ_p)$  whose adjoint representation $p_1\circ \rho_{\Pi,\iota_p}$ 
corresponds to $\pi_5$. 

Except for the case of (1),(3), or (4), $\Pi'$ is tempered, hence $\Pi$ is essentially tempered. 
By definition, if $\Pi$ is CAP representation, it is not essentially tempered, hence 
it falls into the case of (1),(3), or (4) in the above list. 
In any case, it is excluded once we assume $\rho_{\Pi,\iota_p}$ is irreducible. 
On the other hand, if $\Pi$ is endoscopic, the parameter of $\Pi'$ is of type 
$1\boxplus \pi^{(1)}_2\boxplus\pi^{(2)}_2$ or $1\boxplus \pi_4$. As observed, in either case, 
$\rho_{\Pi,\iota_p}$ is reducible.  

Summing up the things observed for $G_B$, we have proved the following  
\begin{thm}\label{cla-GB}Let $\Pi$ be a regular cuspidal representation of $G_B(\A_F)$ and $\rho_{\Pi,\iota_p}$ 
be a corresponding $p$-adic Galois representation explained as above. It holds the followings:
\begin{enumerate}
\item $\Pi$ is non-tempered if and only if $\pi$ is of type $(1),(3)$, or $(4)$ in the above list. In this case, 
$\rho_{\Pi,\iota_p}$ is reducible; 
\item $\rho_{\Pi,\iota_p}$ is reducible if $\Pi$ is endoscopic;
\item $\Pi$ is stable and $($essentially$)$ tempered if $\rho_{\Pi,\iota_p}$ is irreducible. 
\end{enumerate}
In particular, if the residual representation $\br_{\Pi,\iota_p}$ is irreducible, then $\Pi$ is stable and 
essentially tempered. 
\end{thm}
Combining previous results, we have the following; 
\begin{thm}\label{non-cap}Let $\pi$ be a regular cuspidal representation of ${\rm GSp}_4(\A_F)$ and $p$ be a prime. 
Assume that $\br_{\pi,\iota_p}:G_F\lra {\rm GSp}_4(\bF_p)$ is irreducible $($this implies that $\pi$ is stable and tempered$)$.  
Let $\Pi$ be the Jacquet-Langlands transfer of $\pi$ to $G_B(\A_F)$ and $\Pi_1$ be any regular cuspidal representation of $G_B(\A_F)$ whose 
eigensystem is congruent to one of $\Pi$ modulo the maximal ideal of $\bQ_p$. 
Then $\Pi_1$ is stable and $($essentially$)$ tempered. 
\end{thm}   
\begin{proof}It follows from Theorem \ref{cla-GSp4} and Theorem \ref{cla-GB}.  
\end{proof}

\section{$p$-adic algebraic modular forms}\label{pAMF} 
We keep the notation and the assumptions in Section \ref{JL} for $F$ and $G_B$. We refer \cite{CD} as a quick reference for the present.    
In this section we study $p$-adic algebraic modular forms on $G_B(\A_F)$. By passing to Jacquet-Langlands correspondence and 
mod $p$ algebraic modular forms 
we will try to switch a local representation which has an Iwahori fixed vector with some admissible representation 
which never has Iwahori fixed vectors. This is a kind of generalization of Lemma 3.1.5 of \cite{Kisin}. 

Let $K$ be a finite extension of $\Q_p$ contained in $\bQ_p$ with residue field $k$ and $\mathcal{O}$ the ring of integers, and 
assume that $K$ contains the images of all embeddings $F\hookrightarrow \bQ_p$. 
Fix an maximal order $\mathcal{O}_B$ of $B$ and for each finite place $v$ of $F$ fix an isomorphism 
$O_{B,v}\simeq M_2(\mathcal{O}_{F_v})$.  We can also define an integral model of $G_B$ by using $\mathcal{O}_B$ in a 
similar manner. 

For each $v|p$ let $\tau_v$ be  a smooth representation of ${\rm GSp}_4(\mathcal{O}_{F_v})$ on a finite free 
$\mathcal{O}$ module $W_{\tau_v}$. Put $\tau:=\otimes_{v|p}\tau_v$ which is a representation of $\prod_{v|p}{\rm GSp}_4(\mathcal{O}_{F_v})$ 
acting on $W_\tau:=\otimes_{v|p}W_{\tau_v}$. Suppose 
$\psi:F^\times\bs (\A^\infty_F)^\times\lra \mathcal{O}^\times$ is a continuous character so that 
for each $v|p$, $Z_{G_B}(\mathcal{O}_{F_v})\simeq \mathcal{O}^\times_{F_v}$ acts on 
$W_{\tau_v}$ by $\psi^{-1}|_{\mathcal{O}^\times_{F_v}}$. Note that we put the discrete topology on $\mathcal{O}^\times$ and 
therefore such a character is necessarily of finite order. Let $U=\prod_v U_v$ be a compact open subgroup 
of $G_B(\A^\infty_{F})\simeq {\rm GSp}_4(\A^\infty_{F})$ such that $U_v\subset G_B(\mathcal{O}_{F_v})$ 
for all finite place $v$. 
Put $U_p:=\prod_{v|p}U_v$ and $U^{(p)}=\prod_{v\nmid p}U_v$. 
For any local $\mathcal{O}$-algebra $A$ put $W_{\tau,A}:=W_\tau\otimes_{\mathcal{O}} A$. Let 
$\Sigma$ be a finite set of finite places of $F$. 
For each $v\in \Sigma$, let $\chi_v:U_v\lra A^\times$ be a quasi character. Define $\chi_{\Sigma}:U\lra A^\times$ whose local component is $\chi_v$ if $v\in \Sigma$, the trivial representation otherwise. 

\begin{dfn}\label{dfn-pAMF}$(p$-adic algebraic $($Siegel$)$ modular forms$)$ 
Let $S_{0,\tau,\psi}(U,A)$ denote the space of the functions $f:G_B(F)\bs G_B(\A^\infty_F)\lra W_{\tau,A}$ such that 
\begin{itemize}
\item $f(gu)=\tau(u_p)^{-1}f(g)$ for $u=(u^{(p)},u_p)\in U=U^{(p)}\times U_p$ and $g\in G_B(\A^\infty_F)$;
\item $f(zg)=\psi(z)f(g)$ for $z\in Z_{G_B}(\A^\infty_F)$ and $g\in G_B(\A^\infty_F)$.
\end{itemize}
Similarly, 
let $S_{0,\tau,\psi,\chi_{\Sigma}}(U,A)$ denote the space of the functions $f:G_B(F)\bs G_B(\A^\infty_F)\lra W_{\tau,A}$ such that 
\begin{itemize}
\item $f(gu)=\chi^{-1}_{\Sigma}(u)\tau(u_p)^{-1}f(g)$ for $u=(u^{(p)},u_p)\in U=U^{(p)}\times U_p$ and $g\in G_B(\A^\infty_F)$;
\item $f(zg)=\psi(z)f(g)$ for $z\in Z_{G_B}(\A^\infty_F)$ and $g\in G_B(\A^\infty_F)$.
\end{itemize}
We say a function belongs these spaces a $p$-adic algebraic modular form. 
 
\end{dfn}

The similitude $\mu$ of $G_B$ yields 
\begin{equation}\label{ex}
1\lra H_B\lra G_B\stackrel{\mu}{\lra} GL_1\lra 1.
\end{equation}
where $H_B:={\rm Ker}\mu$. Put $D_U:=\mu(U)$. We identify $\A_F$ with the center of $D\otimes_F\A_F$. We write 
$(\A^\infty_F)^\times=\ds\coprod_{i\in I}F^\times t_i D_U$ for some 
$t_i\in (\A^\infty_F)^\times$ and some finite index $I$. 
Take a $s_i\in G_B(\A^\infty_F)$ so that $\mu(s_i)=t_j$. Since $H_B$ is semisimple we have a decomposition 
$G_B(\A^\infty_F)=\ds\coprod_{i\in I}G_B(F)s_i U Z_{G_B}(\A^\infty_F)$. 
Then we have 
\begin{equation}\label{isom1}
S_{0,\tau,\psi}(U,A)\stackrel{\sim}{\lra}\bigoplus_{i\in I}W^{(U Z_{G_B}(\A^\infty_F)\cap s^{-1}_iG_B(F)s_i)/Z_{G_B}(F)}_{\tau,A},\ 
f\mapsto  \{f(s_i)\}_{i\in I}.
\end{equation}

As (3.1.2) of \cite{Kisin} we make the following assumption 
\begin{equation}\label{finiteness}
\mbox{For all $s\in G_B(\A^\infty_F)$ the group $(U Z_{G_B}(\A^\infty_F)\cap s^{-1}G_B(F)s)/Z_{G_B}(F)$ 
has prime to $p$-order.}
\end{equation}
The sequence (\ref{ex}) induces 
\begin{equation}\label{exact-finite}
(U Z_{G_B}(\A^\infty_F)\cap s^{-1}H_B(F)sZ_{G_B}(F))/Z_{G_B}(F)\lra (U Z_{G_B}(\A^\infty_F)\cap s^{-1}G_B(F)s)/Z_{G_B}(F)$$
$$\lra (D_U\cdot ((\A^\infty_F)^\times)^2\cap F^\times)/(F^\times)^2.
\end{equation}
The first group is finite since $H_B(\R)$ is compact and clearly the last group is finite 2-group, and 
therefore the middle one is also finite. 
For a finite place $v$ of $F$ we denote by $U_{B,v}$ and $U_{1,v}$  the subgroups of ${\rm GSp_4}(\mathcal{O}_{F_v})$ 
consisting of all elements which are congruent to  
$$\left(\begin{array}{cccc}
\ast & \ast & \ast & \ast \\
0 & \ast & \ast & \ast \\
0 & 0 & \ast & \ast \\
0 & 0 & 0 & \ast 
\end{array}\right)\ {\rm and} 
\left(\begin{array}{cccc}
1 & \ast & \ast & \ast \\
0 & 1 & \ast & \ast \\
0 & 0 & 1 & \ast \\
0 & 0 & 0 & \ast 
\end{array}\right)
$$
modulo $\pi_v$ respectively. The group $U_{B,v}$ is called Iwahori subgroup. 
We often identify these groups with the corresponding subgroups of $G_B(\mathcal{O}_{F_v})$ under (\ref{isom}). 
The condition (\ref{finiteness}) holds for all sufficiently small $U$ and in fact we have more precise statement as 
follows:   
\begin{lem}\label{double-coset}Assume $p>2$. 
Assume the order of $g$ is prime to $2$. 
\begin{enumerate}
\item Let $v_0$ be a finite place of $F$ lying over a rational odd prime $\ell$. 
Let $K$ be the extension field of $F$ obtained by adding all eigenvalues  of 
any element in the group defined in (\ref{finiteness}) which is of order prime to 2. Put $k=[K:F]$ and 
assume  $U_{v_0}$ is contained in $U_{B,v_0}$. 
If $v_0$ is split completely in $K$, $\ell \nmid k$ and  $k\nmid \ell-1$, 
then the condition (\ref{finiteness}) holds. 
\item Let $v_0$ be a finite place of $F$ above some rational odd prime $\ell$ such that for any 
non-trivial roots $\zeta,\zeta'$ of  unity in an extension of $F$ of degree 4, 
$\zeta+\zeta^{-1}+\zeta'+\zeta'^{-1}\not\equiv 4\ {\rm mod}\ v_0$. 
Assume that $U_{v_0}$ is contained in $U_{1,v_0}$. Then the group in (\ref{finiteness}) is trivial.  
\end{enumerate}
\end{lem}
\begin{proof} Let $\zeta$ be a primitive $k$-th root of unity which belongs to $K$. 
Assume that $K\neq F$. Then the group defined in (\ref{finiteness}) has 
a non-trivial element $g$ whose eigenvalues contain $\zeta$. 
It follows $k\ge 3$, since we do not consider 2-torsion elements. 
We may also assume $g\in U_{v_0}$. 
By assumption, $g$ mod $v_0$ is upper triangular and its diagonal part is of order $k$ since $\ell \nmid k$. 
Since $v_0$ is split completely, $\zeta$ mod $v_0$ gives an element in $\F^\times_\ell$ and then 
$k$ has to divide $\ell-1$ which contradicts the assumption. 

For the second claim by (\ref{exact-finite}) we may assume that $g$ belongs to the left hand side of (\ref{finiteness}) since $p\neq 2$. 
Then the trace of $g$ has to be 4 modulo $v_0$ but this is absurd with the assumption.   
\end{proof}

Let $R$ be a finite set of finite places of $F$ containing 
all places $v\nmid p$ such that $U_v\neq G_B(\mathcal{O}_{F_v})$.  
We define the (formal) Hecke algebra 
\begin{equation}\label{Hecke-algebra}
\mathbb{T}^R_A:=A[T_{1,v},T_{2,v},S_v]_{v\not\in R\cup\{v|p\}}
\end{equation} 
where $T_{i,v}=[G_B(\mathcal{O}_{F_v})\iota^{-1}_{v}(t_{i,v})G_B(\mathcal{O}_{F_v})],S_v=[G_B(\mathcal{O}_{F_v})\iota^{-1}_v(
\diag(\pi_v,\pi_v,\pi_v,\pi_v))G_B(\mathcal{O}_{F_v})]$ 
are Hecke operators for $t_1=\diag(1,1,\pi_v,\pi_v)$ and $t_2=\diag(1,\pi_v,\pi_v,\pi^2_v)$. 
As explained in Section 2.3 of \cite{CD} $S_{0,\tau,\psi}(U,A)$ has a natural action of $\mathbb{T}^R_A$. 
Let $\frak m$ be a maximal ideal of $\mathbb{T}^R_A$ with residue field a finite field of characteristic $p$. 
We say $\frak m$ is in the support of $(\tau,\psi)$ if $S_{0,\tau,\psi}(U,A)_{\frak m}\neq 0$. 
Let $f\in S_{0,\tau,\psi}(U,A)$ be an eigenform for $\mathbb{T}^R_A$ satisfying $Tf=a_T(f)f$ for all $T\in \mathbb{T}^R_A$. 
Then the kernel of $\mathbb{T}^R_A\lra A/m_A, T\mapsto a_T(f)\ {\rm mod}\ m_A$ yields an maximal ideal and we call it 
the maximal ideal associated to $f$. 

To switch types of some local representations for a given automorphic representation of $G_B(\A_F)$ with 
certain prescribed local representations  
we make a good use of mod $p$ algebraic modular forms. 
We start with the following on the lifting:
\begin{lem}\label{lift} Keep the assumption (\ref{finiteness}).  
Let $A$ be a finite 
local $\mathcal{O}$-algebra $A$  and $I\subset A$ an ideal. Let $\overline{\psi}$ be the composition of $\psi$ and 
$A^\times \lra (A/I)^\times$. Denote by $\sigma$ a $U_p$-representation of $W_{\sigma,A}$ and write $W_{\overline{\sigma}}$ 
for $W_{\sigma,A/I}$, and for $v|p$ assume that on $U_v\cap Z_{G_B}(F_v)$, $\overline{\sigma}$ is given by $\overline{\psi}^{-1}$. 
Let $W_{\tau}$ be another $U_p$-module and put $W_{\overline{\tau}}=W_\tau\otimes A/I$. Suppose that $W_{\overline{\sigma}}$ occurs 
as a $U_p$-module subquotient of $W_{\overline{\tau}}$. Then if $\frak m$ is in the support of $(\overline{\sigma},\overline{\psi})$, 
then $\frak m$ is in the support of $(\tau,\psi)$.  
\end{lem}
\begin{proof}It follows from a similar argument of Lemma 3.1.4 of \cite{Kisin} and so omitted. 
\end{proof}

Let $P_1$ be the Siegel parabolic group in $GSp_4$ consisting of matrices  
$\left(\begin{array}{cc}
A & \ast \\
0_2 & \nu s{}^tAs
\end{array}\right)
$. Let $P_2$ be the Klingen parabolic group in $GSp_4$ consisting of matrices  
$\left(\begin{array}{cccc}
t & \ast & \ast & \ast \\
0 & a    & b    & \ast \\
0 & c    & d    & \ast \\
0 & 0    & 0    & \Delta t^{-1}  \\
\end{array}\right)
$. 
Let $P_i=M_iN_i$ be the Levi decomposition so that $M_1\simeq GL_2\times GL_1$ and $M_2\simeq GL_1\times GL_2$ (see p.20 
of \cite{RS}). Let $B_{{\rm st}}$ be the standard upper Borel subgroup of $GSp_4$. 
For an irreducible admissible representation $\pi$ of ${\rm GL}_2(F_v)$ and quasi characters $\chi_1,\chi_2,\mu:F^\times_v\lra \C^\times$. 
We denote by $\chi_1\times \chi_2 \rtimes \mu$ the normalized induction  
${\rm Ind}^{{\rm GSp}_4(F_v)}_{B_{{\rm st}}(F_v)}\chi_1\otimes\chi_2\otimes \mu$. 
Similarly  $\pi\rtimes \mu$ and $\mu\rtimes \pi$ the normalized inductions 
${\rm Ind}^{{\rm GSp}_4(F_v)}_{{P_1}(F_v)}\pi\otimes \mu$ and 
${\rm Ind}^{{\rm GSp}_4(F_v)}_{{P_2}(F_v)}\mu\otimes \pi$ respectively. For these we followed the notation of \cite{RS}.  
Note that 
\begin{equation}\label{rel-ind}
\chi_1\times \chi_2 \rtimes \mu=\pi(\chi_1,\chi_2)\rtimes \mu=\chi_1\rtimes \pi(\chi_2\mu,\mu)
\end{equation}
where $\pi(\chi_1,\chi_2)={\rm Ind}^{{\rm GL}_2(F_v)}_{B_{GL_2}(F_v)}\chi_1\otimes \chi_2$.  

Henceforth we denote by $v$ a finite place of our totally real field $F$.  
For any irreducible admissible, infinite-dimensional representation $\pi$ of ${\rm GL}_2(F_v)$ and 
a quasi character $\sigma:F^\times_v\lra \C^\times$   
let $\Pi_1(\pi,\mu):=G(|\cdot|^{\frac{1}{2}}_v\pi,|\cdot|^{-\frac{1}{2}}_v \mu)$ be a unique generic quotient of 
$|\cdot|^{\frac{1}{2}}_v\pi\rtimes|\cdot|^{-\frac{1}{2}}_v \mu$ (cf. \cite{ST},\cite{RS}). 
Similarly we define $\Pi_2(\pi,\mu)$ as a unique generic quotient of 
$\mu\rtimes \pi$ or $\mu |\ast|_v\rtimes |\ast|^{-\frac{1}{2}}_v\pi$ when $\pi$ is supercuspidal.
According to the notation in \cite{RS} such a representation is labeled as one of types 
(IIa), (Va), (VIa), (XIa) for $\Pi_1(\pi,\mu)$ and (VII), (VIIIa), (IXa) for $\Pi_2(\pi,\mu)$. Note that we arrow 
a non-trivial central character for the representations and such a representation have handled completely in \cite{ST}. 
In our purpose we will focus only on the types (XIa),(VII),(VIIIa),(IXa) and $\pi$ is a dihedral supercuspidal representation of depth zero.  
For $P=P_i$ we denote by $K_P$ the Siegel parahoric subgroup of 
${\rm GSp}_4(\mathcal{O}_{F_v})$ which consists of all elements 
whose reduction modulo $\varpi_v$ belong to $P(\F_v)$. We also define the subgroup $K^+_{P}$ of $K_P$ which consists of 
all elements whose reduction modulo $\varpi_v$ belong to $N_i(\F_v)$  (cf. p.151 of \cite{Rosner1}). 
For any irreducible admissible representation $\Pi=(\Pi,V)$ of ${\rm GSp}_4(F_v)$
as in Section 2.1 of \cite{Rosner1} we define the parahoric restriction  for $K_P$ by 
$$r_{K_P}(\Pi):=(\Pi|_{K_P},V^{K^+_{P}}).$$
For $GL_n$ we also define the parahoric restriction $r_n$ for $K_n={\rm GL}_n(\mathcal{O}_{F_v})$ and 
$K^+_n:=I_n+\pi_v {\rm M}_n(\mathcal{O}_{F_v})$ in a similar manner (cf. Section 2 of \cite{Rosner1}).     
By definition we also have a representation of $K_P/K^+_P\simeq {\rm GL}_2(\F_v)\times \F^\times_v$ acting on $V^{K^+_{P}}$ 
and denote it by the same symbol.  
\begin{lem}\label{pr}  Let $\pi$ be a supercuspidal representation of ${\rm GL}_2(F_v)$ of depth zero and $\mu$ a quasi character. Then 
\begin{enumerate} 
\item The parahoric restrictions of  $\Pi_1(\pi,\mu)$  for $K_{P_1}$ is given by $r_2(\pi)\times r_1(\mu)$ and 
\item the parahoric restrictions of  $\Pi_2(\pi,\mu)$ for $K_{P_2}$ is given by $r_1(\mu)\times r_2(\pi)$ when $\Pi_2(\pi,\mu)$ is 
of type (VII) or (IXa) and $r_1(\mu)\times r_2(\pi)+r_1(\mu^{-1})\times r_1(\mu)r_2(\pi)$ when $\Pi_2(\pi,\mu)$ is 
of type (VIIIa).
\end{enumerate}
\end{lem}
\begin{proof}See the third column of Table 5 in p.158 of \cite{Rosner1} for (XIa) and Table 6 in loc.cit. for 
(VII), (VIIIa) and (IXa). 
\end{proof}

Let $\Pi$ be an irreducible admissible representation of ${\rm GSp}_4(F_v)$ and $\rho_{\Pi}:={\rm rec}^{{\rm GT}}_v(\Pi):W_{F_v}\lra {\rm GSp}_4(\C)$ 
be the corresponding representation of the Weil group $W_{F_v}$ under the local Langlands correspondence due to Gan-Takeda \cite{GT}.  
We say $\Pi$ is $potentially\ good$ if there exists a  finite extension $E$ of $F_v$ such that $\rho_{\Pi}|_{I_{E}}$ is trivial 
where $I_E$ is the inertia subgroup in $W_E$.  
We define three monodromy matrices $\mathcal{N}_i,\ i=1,2,3$ as in p.660 of \cite{Sor}. 
We say $\Pi$ is of type $\mathcal{N}_i$ if the monodromy operator of $\rho_{\Pi}$ is conjugate (in ${\rm GSp}_4(\C)$) to $\mathcal{N}_i$ 
for some $i$.  
\begin{lem}$($Corollary 1 of \cite{Sor}$)$\label{monod} Keep the notation as above. Then $\Pi$ is not potentially good if and only if 
$\Pi$ is Iwahori-spherical and ramified. 
 In this case it holds that $\Pi$ is of type 
$$\left\{
\begin{array}{cl}
\mathcal{N}_1 & \mbox{if $\pi$ is of type (IIa)}\\
\mathcal{N}_2 & \mbox{if $\pi$ is of type (IIIa),(Va),or (VIa)}\\
\mathcal{N}_3 & \mbox{if $\pi$ is of type (IVa)}
\end{array}\right..
$$ 
\end{lem}

\begin{dfn}\label{unipotent-types} Let $\Pi$ be an irreducible admissible representation of ${\rm GSp}_4(F_v)$. Then we say $\Pi$ is unipotent of type $\mathcal{N}_i$ 
for some $1\le i\le 3$ if $\Pi$ is of type $\mathcal{N}_i$. 
\end{dfn}

Let us fix a sufficiently large extension $E$ of $F_v$. We denote by $\F_E$ the residue field of $E$. Fix $\iota_E:E\hookrightarrow \C$.   
\begin{prop}\label{switching}Let $\Pi$ be an admissible representation which is not potentially good. 
Let $P$ be $P_1$ or $P_2$. 
It holds that 
\begin{enumerate}
\item there exists a $\mathcal{O}_E$-lattice $W(\Pi)$ for each constituent of $r_{K_P}(\Pi)$;
\item   for any constituent $\sigma$ of $W(\Pi)\otimes \F_E$ and each $i\in \{1,2\}$ there exists a principal series or depth zero supercuspidal 
representation $\pi$ (in fact, either representation can occur) of ${\rm GL}_2(F_v)$,  
a quasi character $\mu:F^\times_v\lra \C^\times$, and  a $\mathcal{O}_E$-lattice $W(\Pi_i(\pi,\mu))$ of a constituent of 
$r_{K_P}(\Pi_i(\pi,\mu))$ 
so that $\sigma $ occurs in $W(\Pi_i(\pi,\mu))\otimes \F_K$ as a $K_P/K^+_P$-module subquotient. 
Furthermore, $\pi$ can be a dihedral supercuspidal representation for an unramified quadratic extension $M_v/F_v$. 
\end{enumerate}
\end{prop}
\begin{proof}
We only consider the case of $P=P_1$ and the other case is almost identical. 
We will find $\pi$ and $\mu$ which satisfy both of (1) and (2). By Corollary \ref{monod} we have only to consider the five cases as shown there. 
In the case of (IIa)  
we write $\Pi=\Pi_1(\mu_1{\rm St}\rtimes \mu_0)$ for some quasi characters $\mu_0,\mu_1$ of $F^\times_v$. 
By Table 5 of \cite{Rosner1} we have, as a representation of $K_P/K^+_P={\rm GL}_2(\F_v)\times \F_v$,  
\begin{equation}\label{rKp}
r_{K_{P}}(\pi)=\tilde{\mu}_1{\rm St}\boxtimes\tilde{\mu}_0+\tilde{\mu}^{-1}_1{\rm St}\boxtimes\tilde{\mu}_0\tilde{\mu}^2_1+
(\tilde{\mu}_1\times \tilde{\mu}^{-1}_1)\boxtimes \tilde{\mu}_0\tilde{\mu}_1
\end{equation}
where $\tilde{\mu}_i:=r_1(\mu_i)$. for $i=1,2,3$ we denote by $W_i$ $i$-th component of the above decomposition of $r_{K_P}(\pi)$. 
Enlarging $E$ if necessary we may assume that $W_i$ has a $\mathcal{O}_{E}$-lattice $L_i$ such that $L_i\otimes_{\mathcal{O}_{E},\iota_E}\C=W_i$. 
For $W_1,W_2$ applying a proof of Lemma 3.1.5 of \cite{Kisin} (take $m=1$ in the construction of $\theta$) 
there exists a dihedral, depth zero supercuspidal representation $\pi$ of ${\rm GL}_2(F_v)$ for 
an unramified quadratic extension $M_v/F_v$  and a quasi character $\mu$ of $F^\times_v$ so that a cuspidal ${\rm GL}_2(\mathcal{O}_{F_v})$-type 
contains $\mathcal{O}_{E}$-free submodule $W$ such that $L_i\otimes_{\mathcal{O}_{E}}\F_E$ occurs $W\otimes_{\mathcal{O}_{E}}\F_E$ and also 
$r_{K_P}(\Pi_1(\pi,\mu))$ has a $\mathcal{O}_{E}$-lattice whose reduction contains $L_i\otimes_{\mathcal{O}_{E}}\F_E$. 

For $W_3$ there are three kinds of constituents $W_3$-(1) a twist of 1-dimensional representation (when $\tilde{\mu}_1= 
\tilde{\mu}^{-1}_1$),  
$W_3$-(2)  a twist of Steinberg representation (when $\tilde{\mu}_1= 
\tilde{\mu}^{-1}_1$), and $W_3$-(3) an irreducible principal representation (when $\tilde{\mu}_1\neq 
\tilde{\mu}^{-1}_1$) (cf. Section 3.1 of \cite{CDT}). Let $L_3$ be a $\mathcal{O}_{E}$-lattice of $W_3$.  

For the first two cases by Lemma 3.1.5 of \cite{Kisin} and its proof we would find 
a supercuspidal representation $\pi$ of ${\rm GL}_2(F_v)$ of depth zero and a quasi character $\mu$ of $F^\times_v$ which satisfy 
a similar property. For the third case we would find a principal representation $\pi=\pi(\mu_1,\mu_2)$ of 
${\rm GL}_2(F_v)$ so that 
$$\Pi:=|\cdot|^{\frac{1}{2}}_v\pi\rtimes|\cdot|^{-\frac{1}{2}}_v\mu_0=
|\cdot|^{\frac{1}{2}}_v\mu_1\times |\cdot|^{\frac{1}{2}}_v\mu_2\rtimes |\cdot|^{-\frac{1}{2}}_v\mu_0$$ 
(see (\ref{rel-ind}) for the second equality) and $r_{K_P}(\Pi)$ has a $\mathcal{O}_{E}$-lattice whose reduction contains a constituent of $L_3\otimes_{\mathcal{O}_{E}}\F_E$. 
Similarly we can find $\pi$ and $\mu$ for other remaining cases by using Table 5 of \cite{Rosner1}.  
\end{proof}
\begin{rmk}\label{pot-ord}In the construction of $\Pi_i(\pi,\mu)$ for $i=1,2$ in Proposition \ref{switching} for which $\pi$ is a principal representation as in the case 
$W_3$-(3) in the proof there, one can easily check that $\Pi_i(\pi,\mu)$ is ordinary.   
\end{rmk}

\section{Some local deformation rings for $GSp_4$ and $GL_4$}\label{LDR} 
In this section, we will study some local deformation rings for $GSp_4$ with prescribed local deformation datum. 
To apply the results in Section 3 of \cite{BGG1}, we also discuss the relation between local deformations for $GSp_4$ and those of 
$GL_4$.  
The contents should be handled in more general settings but to avoid an excessive digression we focus only on $GSp_4$. 
We refer \cite{Levin}, \cite{BG} for the general case and Section 7.2 of \cite{gg} for $GSp_4$. 
We normalize the definition of Hodge-Tate weights so that all the Hodge-Tate of $p$-adic cyclotomic character are 
$1$. This normalization is different from the convention in Section \ref{main-results}. 

Let $p>2$ be a prime. 
Let $K$ and $M$ be finite extensions of $\Q_p$ with residue fields $\F_M,\F_K$ respectively . 
We assume that $K$ contains any embedding of $M$ into $\overline{K}$. We denote by ${\rm CLN}_{\mathcal{O}_K}$ 
the category of complete local Noetherian $\mathcal{O}_K$-algebras with local homomorphisms $\iota_B:B\lra \F_K$.  
Let $$\br:G_M\lra {\rm GSp}_4(\F_K)$$ be a continuous representation. 
As in Section 7.2 of \cite{gg} let us consider the framed deformation functor for $\br$ from ${\rm CLN}_{\mathcal{O}_K}$ to the category of sets 
defined by corresponding $B$ 
with the set of triples $(V_B,\beta_B,\alpha_B)$ such that 
\begin{enumerate}
\item $V_B$ is a free $B$-module of rank 4 with a continuous $B$-linear action of $G_M$, 
\item $\alpha_B$ is a perfect symplectic pairing $V_B\times V_B\lra B$ which satisfies that 
for any $g\in G_M$, $\alpha_B(gx,gy)=\lambda(g)\alpha_B(x,y)$ 
for some $\lambda(g)\in B^\times$ , 
\item $\beta_B=\{\beta_1,\beta_2,\beta_3,\beta_4\}$ is a symplectic basis with respect to $\alpha_B$ which means $(\alpha_B(\beta_i,\beta_j))_{i,j}=\lambda J$ 
for some $\lambda \in B^\times$, and 
\item $(V_B,\beta_B,\alpha_B)\otimes_{B,\iota_B}\F_K=\br$. 
\end{enumerate}
As explained in Section 7.2 of \cite{gg} this functor is representable and we denote it by $R^{\Box,{\rm sympl}}_{\br}$ and 
the associated universal symplectic lifting:
$$\rho^{\Box,{\rm sympl}}:G_M\lra {\rm GSp}_4(R^{\Box,{\rm sympl}}_{\br}).$$
Fix a continuous character $\psi:G_M\lra \mathcal{O}^\times_K$. 
By replacing the above deformation condition (2) with  
``(2)' for any $g\in G_M$, $\alpha_B(gx,gy)=\psi(g)\alpha_B(x,y)$", we also have a universal deformation ring 
$R^{\Box,{\rm sympl},\psi}_{\br}$ which is a quotient of $R^{\Box,{\rm sympl}}_{\br}$ and also 
the corresponding universal symplectic lift $\rho^{\Box,{\rm sympl},\psi}:G_M\lra {\rm GSp}_4(R^{\Box,{\rm sympl},\psi}_{\br})$. 

By the natural inclusion $GSp_4\subset GL_4$, we can also view $\br$ as a representation of $G_M$ which takes the 
values in ${\rm GL}_4(\F_K)$. As explained in Section 3.1.5 of \cite{BGG1}, there exists a universal deformation ring 
$R^{\Box}_{\br}$ which represents the deformation functor from ${\rm CLN}_{\mathcal{O}_K}$ to the category of sets 
which takes $A\in {\rm CLN}_{\mathcal{O}_K}$ to the set of continuous lifting $\rho:G_M\lra {\rm GL}_4(A)$ of $\br$. We write 
$\rho^{\Box}:G_M\lra 
{\rm GL}_4(R^{\Box}_{\br})$ for the universal ring. It is immediate that $R^{\Box,{\rm sympl}}_{\br}$ is a quotient of 
$R^{\Box}_{\br}$. Hence ${\rm Spec}\hspace{0.5mm}R^{\Box,{\rm sympl}}_{\br}$ is a closed subscheme of 
${\rm Spec}\hspace{0.5mm}R^{\Box}_{\br}$. 

From now on we assume that a fixed character $\psi$ is crystalline. 
Let $\lambda$ be an element of $(\Z^4_+)^{{\rm Hom}(M,K)}$ and let ${\rm v}_\lambda$ be the associated $p$-adic Hodge type 
(see (2.6), p.531 of \cite{K-ss} or Definition 3.1.6 of \cite{gg}). As explained in p.282 of \cite{gg} there is a unique $p$-torsion free quotient 
$R^{{\rm sympl},{\rm v}_\lambda,{\rm cr},\psi}_{\br}$ of $R^{\Box,{\rm sympl},\psi}_{\br}$ with  the 
property that for any finite $K$-algebra $B$, a homomorphism of $\mathcal{O}_K$-algebra $B\lra 
R^{\Box,{\rm sympl},\psi}_{\br}$ factors through $R^{{\rm sympl},{\rm v}_\lambda,{\rm cr},\psi}_{\br}$ if and only if 
the corresponding representation is crystalline of $p$-adic Hodge type ${\rm v}_\lambda$. 
We also denote by $\rho^{{\rm sympl},{\rm v}_\lambda,{\rm cr},\psi}:G_M\lra {\rm GSp}_4(R^{{\rm sympl},{\rm v}_\lambda,{\rm cr},\psi}_{\br})$ 
the corresponding universal symplectic lift. Let $R^{{\rm v}_\lambda,{\rm cr}}_{\br}$ be a unique $p$-torsion free 
quotient of $R^{\Box}_{\br}$ (see Definition 3.1.2, p.1536 of \cite{BGG1}) and any connected component of 
${\rm Spec}\hspace{0.5mm}R^{{\rm v}_\lambda,{\rm cr}}_{\br}$ is known to be 
irreducible by Theorem 1.2 of \cite{I} provided if $p>2$.

Let $$\overline{r}:G_M\lra {\rm GL}_2(\F_K)$$ be a continuous Galois representation of dimension 2. 
By the proofs of Theorem 2.7.6 and Corollary 2.7.7 of \cite{K-ss} there is the framed crystalline deformation ring 
\begin{equation}\label{Kisin-defo}
R=R^{[0,1],K}_{{\rm cris}}
\end{equation} of $\overline{r}$. 
We denote by $r:G_M\lra {\rm GL}_2(R)$ the universal lift of $\rr$ and $r_B$ its base change for any object $B$ of 
${\rm CLN}_{\mathcal{O}_K}$.    
As explained in Remark 2.1.2 of \cite{Kim} its the generic fiber is related to 
Kisin's flat (framed) deformation ring $R^{{\rm fl},\Box}_{\overline{r}}$ defined in \cite{Kisin} (see (2.3.7) in loc.cit.). Hence 
$(R^{{\rm fl},\Box}_{r}\hat{\otimes}_{W(\F_K)}\mathcal{O}_K)[\frac{1}{p}]\simeq R[\frac{1}{p}]$. 

Let $P_1$ (resp. $P_2$) be Siegel (resp. Klingen) parabolic subgroup in $GSp_4$. For each $P_i\ (i=1,2)$, let $\mathcal{F}_i$ be a flag variety 
over $\mathcal{O}_K$ whose $A$-valued points for any $\mathcal{O}_K$-algebra $A$ consists of $(V_A,\alpha_A,\{{\rm Fil}^{(i)}_j\}_{j\in \Z})$ 
where $(V_A,\alpha_A)$ is the same as above and $\{{\rm Fil}^{(i)}_j\}_{j\in \Z}$ is a filtration with respect to $P_i$.  
For $P_1$ it is an increasing filtration 
\begin{equation}\label{p1-fil}
0={\rm Fil}^{(1)}_0\subset {\rm Fil}^{(1)}_1 \subset {\rm Fil}^{(1)}_2=V_A
\end{equation}
by locally free submodules which, locally, are direct summands and the rank of ${\rm Fil}^{(1)}_1$ is 2. 
For $P_2$  it is a similar increasing filtration 
\begin{equation}\label{p2-fil}
0={\rm Fil}^{(2)}_0\subset {\rm Fil}^{(2)}_1 \subset {\rm Fil}^{(2)}_2 \subset {\rm Fil}^{(2)}_3=V_A
\end{equation}
but the rank of ${\rm Fil}^{(2)}_1$ is 1 and the rank of ${\rm Fil}^{(2)}_2$ is 3. 

Henceforth we assume that $\lambda=0$. For $j=1,\ldots,4$, let $\chi^{0}_j:I_M\lra \mathcal{O}^\times_K$ be 
the character defined in Definition 3.1.2 of \cite{gg} which can be though as 
the restriction of any crystalline character $G_M\lra \overline{K}^\times$ whose Hodge-Tate weight 
with respect to $\tau:M\hookrightarrow \overline{K}$ is given by $j-1$. 

Let ${\rm v}_0={\rm v}_\lambda$ with $\lambda=\{(3,2,1,0)\}_{{\rm Hom}(M,K)}$. 
For $i=1,2$ let $\mathcal{G}_i$ be a subfunctor of the functor defined by the scheme 
\begin{equation}\label{imp-scheme}
\mathcal{F}_i\times_{{\rm Spec}\hspace{0.5mm}\mathcal{O}_K}{\rm Spec}\hspace{0.5mm}
R^{{\rm sympl},{\rm v}_0,{\rm cr},\psi}_{\br}
\times_{{\rm Spec}\hspace{0.5mm}\mathcal{O}_K}{\rm Spec}\hspace{0.5mm}R
\end{equation}
 whose $B$-valued points for any object $B$ in ${\rm CLN}_{\mathcal{O}_K}$ 
correspond to $B$-valued points $(V_B,\beta_B,\alpha,r_B)$ of ${\rm Spec}\hspace{0.5mm}
R^{{\rm sympl},{\rm v}_0,{\rm cr},\psi}_{\br}\times_{{\rm Spec}\hspace{0.5mm}\mathcal{O}_K}{\rm Spec}\hspace{0.5mm}R$ such that 
\begin{enumerate}
\item $\{{\rm Fil}^{(i)}_j\}_{j\in \Z}$ is preserved by the action of $G_M$;   
\item for $\mathcal{F}_1$, ${\rm gr}^{(1)}_2=r_B$ and $\{\beta_3,\beta_4\}$ modulo ${\rm Fil}^{(1)}_1$ is the frame of $r_B$. Note that  
the perfect pairing $\alpha_B$ induces a canonical isomorphism $({\rm gr}^{(1)}_1)^\vee(3)\simeq {\rm gr}^{(1)}_2$; 
\item for $\mathcal{F}_2$, $I_M$ acts on ${\rm gr}^{(2)}_1$ and ${\rm gr}^{(2)}_3$ via $\chi^0_4$ and $\chi^0_1$ 
respectively, and ${\rm gr}^{(2)}_2(-1)=r_B$. Further $\{\beta_2,\beta_3\}$ modulo ${\rm Fil}^{(2)}_2$ is the frame of $r_B$.   
\end{enumerate}
\begin{lem} Keep the notation above. Then $\mathcal{G}_i$ is representable by a closed subscheme of 
$\mathcal{F}_i\times_{{\rm Spec}\hspace{0.5mm}\mathcal{O}_K}{\rm Spec}\hspace{0.5mm}
R^{{\rm sympl},{\rm v}_0,{\rm cr},\psi}_{\br}
\times_{{\rm Spec}\hspace{0.5mm}\mathcal{O}_K}{\rm Spec}\hspace{0.5mm}R$ for each $i=1,2$. 
\end{lem}
\begin{proof}We follow the proof of  Lemma 2.1.2, p.51 of \cite{Ge1}. We only consider the case of $\mathcal{F}_1$ and 
the other case is similarly handled. 

Put $\widetilde{R}=R^{{\rm sympl},{\rm v}_0,{\rm cr},\psi}_{\br}\otimes_{\mathcal{O}_K}R$. 
Let $0={\rm Fil}^{(1)}_0\subset {\rm Fil}^{(1)}_1 \subset {\rm Fil}^{(1)}_2=V_{\widetilde{R}}$ be a filtration of free modules over $\widetilde{R}$ 
with respect to $\mathcal{F}_1$. Let $\beta_{\widetilde{R}}=\{\beta_{1,\widetilde{R}},
\beta_{2,\widetilde{R}},\beta_{3,\widetilde{R}},\beta_{4,\widetilde{R}}\}$ be the symplectic basis 
with respect to $\alpha_{\widetilde{R}}$ obtained by the base change of $\beta_{R^{{\rm sympl},{\rm v}_0,{\rm cr},\psi}_{\br}}$ and 
$\alpha_{R^{{\rm sympl},{\rm v}_0,{\rm cr},\psi}_{\br}}$ to $\widetilde{R}$. 
For $i=3,4$ put $\widetilde{\beta}_{i,\widetilde{R}}=\beta_{i,\widetilde{R}}$ modulo ${\rm Fil}^{(1)}_1$. 
We view $r_{R}:G_M\lra GL_2(R)$ as a matrix representation and let it naturally act on ${\rm gr}^{(1)}_2$ by 
$(\widetilde{\beta}_{3,\widetilde{R}},\widetilde{\beta}_{4,\widetilde{R}})r_{R}(\sigma)$ for all $\sigma\in G_M$. 
Then if we denote by $I$ the ideal of $\widetilde{R}$ generated by coefficients (in front of the fixed basis above) of  
$$(\widetilde{\beta}_{3,\widetilde{R}},\widetilde{\beta}_{4,\widetilde{R}})r_{R}(\sigma)-
(\widetilde{\beta}_{3,\widetilde{R}},\widetilde{\beta}_{4,\widetilde{R}})\rho^{{\rm sympl},{\rm v}_\lambda,{\rm cr},\psi}(\sigma)
\ {\rm for\ all}\  \sigma\in G_M,$$
the quotient $\widetilde{R}/I$ gives the desired closed subscheme.   
\end{proof}

Let $\mathcal{V}$ be a connected component of ${\rm Spec}\hspace{0.5mm}R$ and we denote by $\widetilde{R}_{\mathcal{V}}$ 
the base change of $\widetilde{R}$ to $\mathcal{V}$. 
Similarly we define $\mathcal{G}_{i,\mathcal{V}}$ as the base change of $\mathcal{G}_{i}$ to $\mathcal{V}$. 
Let $R^{P_i,{\rm cr}}_{\br}$ be the ring whose spectrum ${\rm Spec}\hspace{0.5mm}R^{P_i,{\rm cr}}_{\br}$ is the scheme 
theoretic image of the morphism 
$$\mathcal{G}_{i,\mathcal{V}}[1/p]\lra {\rm Spec}\hspace{0.5mm} \widetilde{R}_{\mathcal{V}}.$$  
The following is an extension of Lemma 2.4.2 of \cite{Ge1}. 
\begin{prop}\label{conn1}Assume that $\br$ is trivial. Then the scheme $\mathcal{G}_{i,\mathcal{V}}[1/p]$ is connected.  
\end{prop}
\begin{proof} Let $x$ be a closed point of $\mathcal{G}_{i,\mathcal{V}}[1/p]$. We denote by ${\rm Fil}^{(i)}_x$ 
the corresponding filtration. Let $E$ be the residue field of $x$ and also by $\rho_x$ the corresponding representation of $G_M$. 
Choose $g\in {\rm GSp}_4(\mathcal{O}_E)$ such that $g{\rm Fil}^{(i)}_{{\rm St}}={\rm Fil}^{(i)}_x$ where ${\rm Fil}^{(i)}_{{\rm St}}$ 
is the standard filtration with respect to the parabolic subgroup $P_i$. Then $\rho'=g^{-1}\rho_xg$ takes the values in 
$P_i(E)$ such that the Levi-part is $r_x$ for $i=1$ and $(\chi_{4,x},,r_x,\chi_{1,x})$ for $i=2$ where 
$\chi_{j,x}:G_M\lra E^\times$ is a crystalline character which acts on $I_M$ as $\chi^0_j$ for $j=1,4$.   

Let $A$ be the $m_{\mathcal{O}_E}$-adic completion of  
$\G({\rm Spec}\hspace{0.5mm}{\rm GSp}_4/{\mathcal{O}_E},\mathcal{O}_{{\rm Spec}\hspace{0.5mm}{\rm GSp}_4/{\mathcal{O}_E}})[t]$. 
Put $d_1=\diag(t,t,1,1)$ and $d_2=\diag(t^2,t,t,1)$ which are elements in ${\rm GSp}_4(A[1/t])$.   
Then by a direct computation one can check that $\rho'_{t}:=d_i\rho'd^{-1}_i$ takes the values in $P_i(A)$ and further 
$\rho'_{0}=(r_x)^\vee(3)\oplus r_x$ for $i=1$ and $\rho'_{0}=\chi_{4,x}\oplus r_x(1)\oplus \chi_{1,x}$ for $i=2$. 

For $a\in {\rm GSp}_4(A)$ we define $\rho'_{a,t}:=a\rho'_t a^{-1}$. Then the family $\{\rho_{a,t}\}$ defines 
${\rm Spec}\hspace{0.5mm}A\lra \mathcal{G}_{i,\mathcal{V}}$ since the reduction of $\rho_{a,t}$ modulo $m_A$ is trivial 
representation $\rho$. 
The image of the closed point of ${\rm Spec}\hspace{0.5mm}A[1/p]$ corresponding to $a=g$ and $t=1$ is $x$ while 
the image of the closed point of ${\rm Spec}\hspace{0.5mm}A[1/p]$ corresponding to $a=1_4$ and $t=0$ is $\rho'_0$ as observed. 
Since $A$ is domain, the point corresponds to $\rho'_0$ and $x$ must lie in a common irreducible component.  
Finally  $r_x$ lies in the connected component $\mathcal{V}$ and it is the same for $\chi^0_j$ by the description of universal character 
(cf. p.49 of \cite{Ge1}). This completes the proof.     
\end{proof}

Recall that $\mathcal{G}_{i,\mathcal{V}}[1/p]$ is also a closed subscheme of 
${\rm Spec}\hspace{0.5mm}R^{{\rm v}_0,{\rm cr}}_{\br}$ since $R^{{\rm v}_0,{\rm cr}}_{\br}$ is $p$-torsion free. 
\begin{prop}\label{conn2}Keep the notation in Proposition \ref{conn1}. 
Then any two geometric points $x_1={\rm Spec}\hspace{0.5mm}E_1,\ x_2={\rm Spec}\hspace{0.5mm}E_2$ on 
$\mathcal{G}_{i,\mathcal{V}}[1/p]$ with $E_i$ a finite extension of $K$ for $i=1,2$ lie on a common irreducible component of 
${\rm Spec}\hspace{0.5mm}R^{{\rm v}_0,{\rm cr}}_{\br}$. 
\end{prop}
\begin{proof}
We may assume that $E_1=E_2$ by enlarging the coefficient fields if necessary. 
By Proposition \ref{conn1}, $x_1,x_2$ is connected by irreducible components $V_1,\ldots,V_r$ for some $r\ge 1$ 
such that $V_i\cap V_j\neq \emptyset$. 
Let $W_i$ be an irreducible component of ${\rm Spec}\hspace{0.5mm}R^{{\rm v}_0,{\rm cr}}_{\br}$ containing $V_i$. 
Obviously, the union $\bigcup_{i=1}^rW_i$ is connected. However any connected component of 
${\rm Spec}\hspace{0.5mm}R^{{\rm v}_0,{\rm cr}}_{\br}$ is irreducible by Theorem 1.2 of \cite{I}. Hence 
there exists $i_0$ with $1\le i \le r$ such that $x_1,x_2$ lie on $W_{i_0}$. 
\end{proof}
\begin{rmk}\label{conn3} For the two points $x_1,x_2$ in Proposition \ref{conn2}, enlarging coefficients if necessary, 
$\rho_{x_1}\sim \rho_{x_2}$ in the sense of Definition 3.1.4 of \cite{BGG1} where $\rho_{x_i}$ is 
the lift of $\br$ corresponding to $x_i$ for $i=1,2$. 
\end{rmk}

\section{Cuspidal representations of weight zero}
Let $\pi$ be a regular cuspidal automorphic representation of ${\rm GSp}_4(\A_F)$ in Section \ref{AutoGal} and 
$\rho_{\pi,\iota_p}$ the associated $p$-adic Galois representation. We denote by $\br_{\pi,\iota_p}$ the reduction of $\rho_{\pi,\iota_p}$. 
We say $\pi$ is of weight zero if any infinite component of $\pi$ has L-parameter $\phi_{(w;2,1)}$ for a fixed $w$.   
In this section we will find another cuspidal automorphic representation which is congruent to $\pi$ after 
a suitable base change. 
An idea is to use algebraic modular forms by passing to Jacquet-Langlands correspondence after a suitable 
base change.   
 
\begin{thm}\label{weight0} Let $\pi$ be a regular cuspidal representation of ${\rm GSp}_4(\A_F)$. 
Assume that $p>2$ and $\br_{\pi,\iota_p}$ is irreducible. 
Then there exists a totally real solvable extension $F_1/F$ and a 
cuspidal automorphic representation  $\pi'$ of ${\rm GSp}_4(\A_{F_1})$ 
such that 
\begin{enumerate}
\item $\br_{\pi',\iota_p}\simeq  \br_{\pi,\iota_p}|_{G_F}$ up to twist;
\item  $\pi'$ is of weight zero.  
\end{enumerate}
\end{thm}
\begin{proof}
Put $\br=\br_{\pi,\iota_p}$. 
By using base change in Section \ref{BC}, there exists a totally real solvable finite extension $F_1/F$ and 
a regular cuspidal automorphic representation of $\pi_1$ of ${\rm GSp}_4(\A_{F_1})$ such that 
\begin{enumerate}
\item $[F_1:F]$ is even;
\item $F_1$ is linearly disjoint from $\overline{F}^{{\rm Ker}(\br)}$ over $F$;
\item for each finite place $v$ of $F_1$ lying over $p$, $\pi_{1,v}$ has a non-zero 
Iwahori fixed vector. 
\end{enumerate}
Let $psi$ be a finite character of $G_{F_1}$ such that $\rho_{}$
By Theorem \ref{JL}, we have the Jacquet-Langlands correspondence $\pi^{{\rm JL}}_1$ of $\pi_1$. 
Then there is a Hecke eigen form $f$ generating $\pi^{{\rm JL}}_1$ such that 
all Hecke eigen values for Hecke operators outside the ramified places 
are defined over a suitable $\O=\O_E$ for some finite extension $E/\Q_p$ with the residue field $\F$. 
For each finite place $v$ of $F_1$ lying over $p$, by enlarging $E$ if necessary, 
fix a $\O$-lattice $W_{\tau_v}$ of $\pi^{{\rm JL}}_{1,v}|_{{\rm GSp}_4(\O_{F_1,v})}$. Put $\tau=\otimes_{v|p}\tau_v$. 
Then via ({Inter-map}), $f$ contributes to $S_{0,\tau,\psi}(U,\O)$ for some character $\psi$   and an open 
compact subgroup $U$ of ${\rm GSp}_4(\A^\infty_{F_1})$ such that $U_v=U_{1,v}$ 
for each $v|p$.  By shrining $U$ if necessary, we may assume that $U$ satisfies the condition (\ref{finiteness}). 
Consider the reduction of $f$ to $\overline{f}\in S_{0,\tau,\psi}(U,\O)\otimes_{\O}\F$. Since $U_v$ acts unipotently on 
$W_{\tau}\otimes_{\O}\F$, it can be understood as a successive extension of the 
trivial representations. Then there exists a Hecke eigen form $\overline{g}$ in 
$S_{0,\tau_1,\psi_1}(U,\F)$ corresponding to $\overline{f}$ 
where $\tau_1$ is the trivial representation of ${\rm GSp}_4(\O_{F_1,p})$ and $\psi$ is the trivial character. 
By Lemma \ref{lift} with (\ref{finiteness}) and Lemma 6.11 of \cite{DS}, $\overline{g}$ is liftable to a non-zero Hecke eigen form $h$ of  
$S_{0,\tau_1,\psi_1}(U,\O)$. By Theorem \ref{non-cap}, the cuspidal representation $\pi_h$ associated to $h$ 
is not a CAP representation. Therefore, we can apply Jacquet-Langlands 
correspondence to $\pi_h$. 
Let $\pi'$ be the cuspidal representation of ${\rm GSp}_4(\A_{F_1})$ corresponding to $h$ via Jacquet-Langlands 
correspondence. Since $\tau_1$ is of dimension one, by the formula (\ref{dim-al}), 
any infinite component of $\pi'$ has L-parameter $\phi_{(w;2,1)}$ for a fixed $w$ as desired.  
\end{proof}

\section{Serre weights and algebraic modular forms over $p$-adic rings}\label{Serre-Weights}

Basically we follow the notation in \cite{til&her}. 
Let $B$ be the upper Borel subgroup of ${\rm GSp}_4/\Z$ and $T=\{t=\diag(t_1,t_2,\nu t^{-1}_2,\nu t^{-1}_1)\ |\ t_1,t_2,\nu\in GL_1\}$ be the diagonal torus in $B$. Let $p$ be an odd prime number which is split completely in $F$. 
Let $X(T)$ be the group of characters of $T$. We identify it with the sub-lattice of $\Z^3$ consisting of $(a,b,;c)\in \Z^3$ with 
$a+b\equiv c$ mod 2 by the formula 
$$\lambda=\lambda_{(a,b,;c)}:t=\diag(t_1,t_2,\nu t^{-1}_2,\nu t^{-1}_1)\mapsto t^a_1t^b_2\nu^{\frac{c-a-b}{2}}.$$
Put $X(T)_+=\{\lambda=(a,b,c)\in X(T)\ |\ a\ge b\ge 0\}$. 
As in D\'efinition 4.4 of \cite{til&her} we define the set of $p$-restricted weights by   
\begin{equation}\label{x1}
X_1(T)=\{\lambda=(a,b,c)\in X(T)\ |\ 0\le a-b <p,\ 0\le b<p\}
\end{equation}
and also the set of $p$-regular weights by  
$$X_1(T)=\{\lambda=(a,b,c)\in X(T)\ |\ 0\le a-b <p-1,\ 0\le b<p-1\}$$
with the subgroup $X^0(T)=\{(0,0;c)\ |\ c\in 2\Z\}$. 
\begin{dfn}For each place $v$ above $p$ Serre weights for ${\rm GSp}_4(\F_v)$ are irreducible representations of ${\rm GSp}_4(\F_v)$ over 
$\bF_p$. 
\end{dfn} 
We view $\gs_4$ and its algebraic subgroups as algebraic groups over $\F_v=\F_p$. 
For $\lambda=(a,b;c)\in X(T)_+$ the dual Weyl $\gs_4$-module is defined by $W(\lambda)_{\bF_p}={\rm Ind}^{\gs_4}_{B^{-}}(\bF_p(\lambda))$ 
where $B^-$ is the opposite of $B$. The maximal semisimple $\gs_4$-submodule of $W(\lambda)_{\bF_p}$ is denoted by 
$F(\lambda):={\rm soc}_{\gs_4}(W(\lambda)_{\bF_p})$ which is isomorphic to a simple $\gs_4$-module of the highest weight $\lambda$.    
\begin{prop}\label{swc}(Proposition 4.5 of \cite{til&her}) Any Serre weight is isomorphic to 
$\widetilde{F}(\lambda):=F(\lambda)|_{\gs_4(\F_p)}$ for some $\lambda\in X_1(T)$ and vice versa. For 
$\lambda,\lambda'\in X_1(T)$, $\widetilde{F}(\lambda)\simeq \widetilde{F}(\lambda')$ if and only if $\lambda-\lambda'\in (p-1)X^0(T)$.  
\end{prop}
The relation between $p$-adic algebraic modular forms, their mod $p$ reduction, and Serre weights will be 
revealed as follows. 

In what follows, we assume that $[F:\Q]$ is even until we make any change for $F$. Let $B$ be a definite quaternion algebra introduced in 
Section \ref{JL}. The symbol $B$ should not be confused with the upper Borel subgroup.  
Let $E$ be a finite extension of $\Q_p$ in the 
setting of Section \ref{pAMF}. Fix an embedding $E\hookrightarrow \C$.  Let $S_p$ be the 
set of all places above $p$. Fix an integer $w$.  
For $\underline{\lambda}=(\lambda_v)_{v|p}\in (X(T)_+)^{S_p}$ with 
\begin{equation}\label{rel-weights}
\lambda_v=(m_{1,v}-2,m_{2,v}-1;w)\ {\rm with}\ \delta_v=\frac{1}{2}(w+3-m_{1,v}-m_{2,v})
\end{equation} we define 
$\tau^{{\rm alg}}:=W(\underline{\lambda})_{\mathcal{O}_E}=\ds\bigotimes_{v|p}{\rm Ind}^{\gs_4}_{B^{-}}(\mathcal{O}_E(\lambda_v))$. By (3) in p.40 of \cite{J} 
$\tau^{{\rm alg}}_\C:=W(\underline{\lambda})_{\mathcal{O}_E}\otimes\C
\simeq \ds\bigotimes_{v|p}\xi_{\delta_v,m_{1,v}-2,m_{2,v}-1}\otimes \C$ as an algebraic representation of $\prod_{v|p}G_B(\C)=\prod_{v|p}\gs_4(\C)$. 
Therefore we may view $W(\underline{\lambda})_{\mathcal{O}_E}\otimes\C$ as a representation of $G_B(F\otimes_\Q\R)$ via $G_B(F\otimes_\Q\R)\hookrightarrow \prod_{v|p}G_B(\C)=\prod_{v|p}\gs_4(\C)$. Indeed we may also write $W(\underline{\lambda})_{\mathcal{O}_E}\otimes\C=
\xi^{{\rm JL}}_{\delta_v,m_{1,v}-2,m_{2,v}-1}\otimes \C$ as a representation of $G_B(F\otimes_\Q\R)$. 
Recall a smooth representation $W_\tau$ over $\mathcal{O}_E$ of $U_p:=\prod_{v|p}{\rm GSp}_4(\mathcal{O}_{F_v})$ and a compact open subgroup $U$ of $G_B(\A^\infty_F)$ in Section \ref{pAMF}. 
Let us consider $\rho=(W(\underline{\lambda})_{\mathcal{O}_E}\otimes E)\otimes_E W_\tau$ which is regarded as a representation of $U_p\times G_B(F\otimes_\Q\R)$. Put $\rho_\C=\rho\otimes_E\C $ and denote by 
$\rho^\ast_\C$ the $\C$-linear dual of $\rho_\C$. Choose a compact open subgroup $U'$ of $G_B(\A^\infty_F)$ so that $U'_v=U_v$ if $v\nmid p$ and $U'_v\subset U_v$ acts trivially on 
$W_\tau$ if $v|p$. We denote by $C^\infty(G_B(F)\bs G_B(\A_F)/U')$ the space of smooth $\C$-valued functions on $G_B(F)\bs G_B(\A_F)$ which are invariant on the left by $U'$. 
Then the map 
\begin{equation}\label{Inter-map}
A:S_{0,\rho,\psi}(U,E)\lra {\rm Hom}_{G_B(F\otimes_\Q\R)}(\rho^\ast_\C,C^\infty(G_B(F)\bs G_B(\A_F)/U'))
\end{equation}
defined by 
$$f\mapsto [w\mapsto (g=g_{\textbf{f}} g_\infty\mapsto w(\tau^{{\rm alg}}_\C(g_\infty)^{-1}\tau^{{\rm alg}}_\C(g_p)f(g_{\textbf{f}})))]$$
where $g_p$ is the $p$-component of the finite part $g_{\textbf{f}}$ of $g\in G_B(\A_F)$. Note that we view $\rho$ as the smooth representation $\rho|_{U_p}$ of $U_p$ to define 
$S_{0,\rho,\psi}(U,E)$. This map is $\mathbb{T}^R_E$-equivariant and injective. Then for a cuspidal representation $\pi'$ of $G_B(\A_F)$, it is generated by $A(f)$ for some 
Hecke eigen form $f\in S_{0,\rho,\psi}(U,E)$ for sufficiently large $U$ and $E$ if and only if $\pi'_\infty$ is isomorphic to $W(\underline{\lambda})_{\mathcal{O}_E}\otimes\C=
\xi^{{\rm JL}}_{\delta_v,m_{1,v}-2,m_{2,v}-1}\otimes \C$ and $\pi'_p:=\otimes_{v|p}\pi'_v$ contains $W_\tau\otimes_E\C$ as a representation of $U_p$. 
Furthermore if $\pi$ is not a CAP representation, one can correspond a regular cuspidal representation $\pi$ of $\gs_4(\A_F)$ such that $\pi_v$  belongs to the L-packet 
associated to $\phi_{(w;m_{1,v},m_{2,v})}$ for each $v|\infty$ and the same condition for the $p$-component $\pi_p=\otimes_{v|p}\pi_v$.

On the other hand by Proposition-(b) in 4.18, p.57 and II.8.7,8.8,p.271-272 of \cite{J} we have $W(\underline{\lambda})_{\mathcal{O}_E}\otimes \bF_p
\stackrel{\sim}{\lra} \ds\bigotimes_{v|p}W(\lambda_v)_{\bF_p}$ and 
the image has $F(\underline{\lambda}):=\ds\bigotimes_{v|p} F(\lambda_v)$ as a sub-representation. Therefore, 
we may view any Serre weight as 
a representation of $\gs_4(\mathcal{O}_{F_v})$ via the reduction map. 
%$\gs_4(\mathcal{O}_{F_v})\lra \gs_4(\F_p)$  
We define the space of mod $p$ algebraic modular forms of weight $F(\underline{\lambda})$ by $S_{0,F(\underline{\lambda}),\psi}(U,\bF_p)$.   
By (\ref{isom1}) and Lemma \ref{double-coset} the reduction map $S_{0,W(\underline{\lambda}),\psi}(U,\mathcal{O}_E)\lra S_{0,W(\underline{\lambda}),\psi}(U,\bF_p)$ is surjective if $U$ is sufficiently small. 
Then by passing to via Jacquet-Langlands correspondence for any $\mathbb{T}^R_{\mathcal{O}_{E}}$-eigen form $h$ in 
$S_{0,F(\underline{\lambda}),\psi}(U,\bF_p)\subset S_{0,W(\underline{\lambda}),\psi}(U,\bF_p)$ which can be liftable to 
characteristic zero (this is, of course, true for sufficiently small $U$) one can associate $h$ with a mod $p$ Galois representation $\br_{h,p}:G_F\lra \gs_4(\bF_p)$ via Jacquet-Langlands correspondence.  

\begin{dfn}\label{Serre-weights}
\begin{enumerate}
\item An irreducible representation over $\bF_p$ of $\prod_{v|p}\gs_4(\F_v)$ is called a Serre weight. It is of form 
$\widetilde{F}(\underline{\lambda}):=\ds\bigotimes_{v|p} \widetilde{F}(\lambda_v)$ such that 
$\widetilde{F}(\lambda_v)$ is a Serre weight for each $v|p$. 
\item A mod $p$ irreducible Galois representation $\br:G_F\lra \gs_4(\bF_p)$ is said to be automorphic of weight 
$\widetilde{F}(\underline{\lambda})$ if 
there exists a mod $p$ algebraic modular forms $h$ in $S_{0,F(\underline{\lambda}),\psi}(U,\bF_p)$ for a 
sufficiently small $U$ with $U_p\simeq \prod_{v|p}{\rm GSp}_4(\mathcal{O}_{F_v})$ satisfying the finiteness 
condition $($\ref{finiteness}$)$  such that $\br$ is equivalent to $\br_{h,p}$.  
\item For a mod $p$ irreducible Galois representation $\br:G_F\lra \gs_4(\bF_p)$, 
we define the set $W(\br)$ consisting of all Serre weights $\widetilde{F}(\underline{\lambda})$ such that 
$\br$ is automorphic of weight $\widetilde{F}(\underline{\lambda})$ in the above sense.    
\end{enumerate} 
\end{dfn}
We now turn to the case when $F$ is any totally real field but our prime $p$ is split completely in $F$. 
When $[F:\Q]$ is odd, we do not have any satisfactory result for Jacquet-Langlands correspondence except for 
$F=\Q$ (see \cite{Hoften} in this case). To avoid it naively, we use the base change which is now available 
gratifyingly.    
\begin{dfn}\label{Pot-Serre-weights1}Let $F$ be a totally real field and $p$ be a prime which is split 
completely in $F$. Let $\br:G_F\lra \gs_4(\bF_p)$ be 
a mod $p$ irreducible automorphic Galois representation. 
We define the set $\mathcal{I}^{{\rm BC}}(\br)$ consisting of all finite solvable extension 
$L/F$ of totally real fields such that $[L:\Q]$ is even, $p$ is split completely in $L$, and $L$ is linear disjoint from $\overline{F}^{{\rm Ker}(\br)}(\zeta_p)$ 
over $F$. Then for each $L\in \mathcal{I}^{{\rm BC}}(\br)$, if $\br$ comes from a regular cuspidal 
automorphic representation $\pi$ of ${\rm GSp}_4(\A_F)$, then the base change ${\rm BC}_{L/F}(\pi)$ to 
${\rm GSp}_4(\A_F)$ does exist by Section \ref{BC}. 
For any non-empty subset $\mathcal{I}\subset \mathcal{I}^{{\rm BC}}(\br)$, we define the collection of the 
Serre weights for a family: 
$$W_\mathcal{I}(\br)=\bigcup_{L\in \mathcal{I}} W(\br|_{G_{L}}).$$
\end{dfn}
\begin{dfn}\label{Pot-Serre-weights2}Let $\br:G_F\lra \gs_4(\bF_p)$ be 
a mod $p$ irreducible automorphic Galois representation. Let $\mathcal{I}$ be a non-empty subset of 
$ \mathcal{I}^{{\rm BC}}(\br)$. 
For each $L\in \mathcal{I}$, we define the set $W_{{\rm cris}}(\br|_{G_L})$ (resp. $W_{{\rm pd-cris}}(\br|_{G_L})$) consisting of 
all Serre weights $\widetilde{F}(\underline{\lambda})$ such that for each place $v$ lying over $p$, 
$\br|_{G_{L,v}}$ has a crystalline (resp. potentially diagonalizable, crystalline) lift of regular Hodge-Tate weights 
$$\{\delta'_v,\delta'_v+m'_{2,v},\delta'_v+m'_{1,v},\delta'_v+m'_{1,v}+m'_{2,v} \},\ 
\delta'_v=\frac{1}{2}(w'+3-m'_{1,v}-m'_{2,v})
$$ such that the reduction of the Weyl module $W((m'_{1,v}-2,m'_{2,v}-1;w'))$ contains 
$\widetilde{F}(\underline{\lambda})$ as a constituent. 
Here the highest weight $(m'_{1,v}-2,m'_{2,v}-1,w')$ is related to the Hodge-Tate weights via the relation  
$($\ref{rel-weights}$)$. 
For $\ast\in\{{\rm cris},{\rm pd-cris}\}$, put $W_{\ast,\mathcal{I}}(\br)=\bigcup_{L\in \mathcal{I}}W_\ast(\br|_{G_L})$. 
Clearly $W_{{\rm pd-cris}}(\br|_{G_L}) \subset W_{{\rm cris}}(\br|_{G_L})$ and accordingly 
$W_{{\rm pd-cris},\mathcal{I}}(\br) \subset W_{{\rm cris},\mathcal{I}}(\br)$. 
\end{dfn}
According to Toby Gee's philosophy in Serre conjecture in terms of crystalline lifts given in 
Section 4 of \cite{gee} we formulate a kind of Serre's weight conjecture as follows:
\begin{conj}\label{conj}Let $\br:G_F\lra \gs_4(\bF_p)$ be 
a mod $p$ irreducible automorphic Galois representation. 
For any non-empty subset $\mathcal{I}\subset \mathcal{I}^{{\rm BC}}(\br)$, it holds that 
$$W_{ \mathcal{I}}(\br)=W_{{\rm pd-cris},\mathcal{I}}(\br)=W_{{\rm cris},\mathcal{I}}(\br).$$
\end{conj}
\begin{rmk}\label{directions} Keep the notation as above. Obviously, 
$$W_{ \mathcal{I}}(\br)\subset W_{{\rm cris},\mathcal{I}}(\br).$$
Hence the relations 
\begin{itemize}
\item $W_{ \mathcal{I}}(\br)\supset W_{{\rm pd-cris},\mathcal{I}}(\br)$;
\item $W_{{\rm pd-cris},\mathcal{I}}(\br)=W_{{\rm cris},\mathcal{I}}(\br)$
\end{itemize}
imply the conjecture here.  
The second claim has been remained to be open except for a few cases though it is believed to be true in general. 
\end{rmk}
Pick a Serre weight $F(\lambda)$ in $W_{{\rm pd-cris},\mathcal{I}}(\br)$. 
By Theorem \ref{main-thm2} which is proved later, there exists a $\lambda'\in X_+(T)^{S_p}$ such that 
$\br$ is automorphic of weight $W(\lambda')\otimes_{\O_E}\C$ and $W(\lambda')\otimes_{\O_E}\F_v$ contains 
$F(\lambda)$ as a constituent. 
We expect $\br$ is automorphic of weight $F(\lambda)$ but in general, $W(\lambda')\otimes_{\O_E}\F_v$ contains many constituents and it would be hard to control the constituents for the general situation. However, if 
$\lambda$ is in some range, then one can control them as the authors in \cite{BGG-unitary}(Section 4 and 5) carried out. 
Let us introduce four alcoves described in p.13 of \cite{til&her}. 
For each $i\ (0\le i\le 3)$, we denote by $C_i$ the subset of $X(T)\otimes \R$ with the coordinate $(x,y;z)$ defined by  
\begin{eqnarray}
C_0&:=&\{(x,y;z)\in X(T)\otimes \R\ |\ x>y>0,\ x+y<p\}, \nonumber \\
C_1&:=&\{(x,y;z)\in X(T)\otimes \R\ |\ x+y>p,\ y<x<p\}, \\
C_2&:=&\{(x,y;z)\in X(T)\otimes \R\ |\ x-y<p<y,\ x+y<2p\}, \nonumber \\
C_3&:=&\{(x,y;z)\in X(T)\otimes \R\ |\ y<p,\ x+y>2p,\ x-y<p\} \nonumber 
\end{eqnarray}
Then we see that $X_1(T)=X(T)_+\cap \ds\bigcup_{i=0}^3 \overline{(C_i-\widetilde{\rho})}$ where 
$\widetilde{\rho}=(2,1;3)$. Since $X_1(T)$ (modulo $X^0(T)$) parametrizes all Serre weights, 
they are divided into four types as above. 
As in (5),(6),(7),(8) in p.15 of \cite{til&her}, we have an explicit decomposition of 
the mod $p$ reduction of the Weyl modules in each case. However to relate the highest weight 
of the Weyl module with the inertia type of $\br$ at $v|p$ is still hard to control. 
Therefore, we further restrict ourself to consider $C_0$ or $C_1$. 
In this vein, we introduce the following:
\begin{dfn}\label{Pot-Serre-weights-FL-range}
Keep the notation in Definition \ref{Pot-Serre-weights2}. 
For each $L\in \mathcal{I}$, we define the subset $W^{C_0\cup C_1}_{{\rm pd-cris}}(\br|_{G_L})$ of 
$W_{{\rm pd-cris}}(\br|_{G_L})$ consisting of 
all Serre weights $\widetilde{F}(\underline{\lambda})$ such that for each place $v$ lying over $p$, 
$\br|_{G_{L,v}}$ has a potentially diagonalizable, crystalline lift of regular Hodge-Tate weights 
$$\{\delta'_v,\delta'_v+m'_{2,v},\delta'_v+m'_{1,v},\delta'_v+m'_{1,v}+m'_{2,v} \},\ 
\delta'_v=\frac{1}{2}(w'+3-m'_{1,v}-m'_{2,v})
$$ such that 
\begin{itemize}
\item 
the reduction of the Weyl module $W((m'_{1,v}-2,m'_{2,v}-1;w'))$ contains 
$\widetilde{F}(\underline{\lambda})$ as a constituent;
\item $(x,y,w):=(m'_{1,v}-2,m'_{2,v}-1;w')\in C_0\cup C_1$;
\item $\br|_{I_{L,v}}$ $($or $\br|_{G_{F,v}}$ for the last case$)$ 
enjoys either of the following cases:
\begin{enumerate}
\item $($Borel ordinary case$)$ $\br|_{I_{F,v}}\simeq (\ve^{x+y}\oplus \ve^{x}\oplus \ve^{y}\oplus \textbf{1})\otimes \ve^{\delta},\ 
\delta:=\frac{1}{2}(w+3-x-y)$;
\item $($Siegel ordinary case$)$ $\br|_{I_{F,v}}\simeq (\ve^{x+y}\oplus (\omega^{x+y p}_2\oplus \omega^{y+x p}_2)\oplus \textbf{1})\otimes \ve^{\delta}$;
\item $($Klingen ordinary case$)$ 
$\br|_{I_{F,v}}\simeq (\omega^{x+y p}_2\oplus \omega^{y+x p}_2)\oplus 
((\omega^{-(x+y p)}_2\oplus \omega^{-(y+x p)}_2)\otimes \ve^{w+3})$;
\item $($Endoscopic case$)$ 
$\br|_{I_{F,v}}\simeq ((\omega^{x+y p}_2\oplus \omega^{y+x p}_2)\oplus 
(\omega^{c+d p}_2\oplus \omega^{d+c p}_2))\otimes \ve^{\delta}$ for some integers 
$0\le d< c\le p-1$ satisfying $c+d\equiv x+y\ {\rm mod}\ p-1$, 
\item  $($Irreducible case$)$ $\br|_{G_{F,v}}$ is irreducible. 
\end{enumerate} 
\end{itemize} 
\end{dfn}

\section{Main results}\label{main-results} In this section we prove main results which show 
the existence of a potentially ordinary lift for a given automorphic mod $p$ Galois representation 
which satisfies the adequacy condition. 
As in Lemma 6.1.1 of \cite{BGG1}, we first construct a non-ordinary lift whose local property is easy to handle in 
conjunction with the study of the local deformations done in Section \ref{LDR}. 
In most recent articles regarding automorphic lifting theorems, the symbol $F$ is denoted to be a CM field while $F^+$ stands for a totally real field. 
We follow this convention.    

\begin{lem}\label{non-ord-lift}$($the existence of a potentially non-ordinary lift$)$ 
Let $F^+$ be a totally real field and $p$ be an odd prime which is split completely in $F^+$. 
Let $\br:G_{F^+}\lra {\rm GSp}_4(\bF_p)$ be an automorphic mod $p$ Galois representation such that 
$\br|_{G_{F^+(\zeta_p)}}$ is irreducible and 
$\br(G_{F^+(\zeta_p)})$ is adequate. Then there exists a finite solvable extension $L^+$ of $F^+$ such that 
\begin{enumerate}
\item $\br|_{G_{L^+(\zeta_p)}}$ is irreducible and $\br(G_{L^+(\zeta_p)})$ is adequate; 
\item $\br|_{G_{L^+_v}}$ is trivial for all places $v|p$ of $L^+$;
\item for each place $v|p$, $[L^+_v:\Q_p]\ge 2$;
\item there exists a cuspidal automorphic representation $\pi$ of ${\rm GSp}_4(\A_{L^+})$ of weight zero such that 
\begin{enumerate}
\item $\br_{\pi,\iota_p}\simeq \br|_{G_{L^+}}$, 
\item $\pi$ is unramified at all finite places, 
\item for all places $v|p$ of $L^+$ and a fixed $i\in \{1,2\}$, $\rho_{\pi,\iota_p}|_{G_{L^+_v}}$ is non-ordinary, 
but Klingen ordinary in the sense of \cite{TU} and further it gives an $\mathcal{O}$-valued point of the scheme 
(\ref{imp-scheme}) for $P_i$ and for the integer ring $\mathcal{O}$ of a suitable finite extension of $\Q_p$.  
\end{enumerate}
\end{enumerate} 
\end{lem}
\begin{proof}By Theorem \ref{weight0}, there exists a totally real solvable extension 
$F^+_0/F^+$ and a cuspidal automorphic representation $\pi$ of 
${\rm GSp}_4(\A_{F^+_0})$ of weight zero such that $\br|_{G_{F_0}}\simeq \br_{\pi,\iota_p}$. 

Let $r=\rho_{\pi,\iota_p}:G_{F^+_0}\lra {\rm GSp}_4(\bQ_p)$. 
Choose a finite solvable totally real field extension $F^+_1$ of $F^+_0$ such that 
\begin{enumerate}
\item $\br|_{G_{F^+_1(\zeta_p)}}$ is irreducible 
and $\overline{r}(G_{F^+_1(\zeta_p)})=\br(G_{F^+_1(\zeta_p)})$ is adequate; 
\item $\overline{r}|_{G_{F^+_{1,v}}}$ is trivial for all places $v|p$ of $F^+_1$; 
\item $\overline{r}|_{G_{F^+_{1,v}}}$ is unramified for all places  of $F^+_1$; 
\item $[F^+_1:\Q]$ is even;
\item $[F^+_{1,v}:\Q_p]\ge 2$ for all places $v|p$ of $F^+_1$;
\item the base change $\pi_{F^+_1}$ of $\pi$ to $F^+_1$ is unramified or unipotent of type $\mathcal{N}_i$ 
 for some $1\le i \le 3$ at each finite places of $F^+_1$;
\item if $\pi_{F^+_1}$ is ramified at a place $v\nmid p$ of $F^+_1$, then $Nv\equiv 1$ mod $p$ where $Nv$ stands 
for the cardinality of $\F_v$.  
\end{enumerate}
In fact, there are only finitely many ramified places of $\pi$ and the local Galois group 
$G_{F^+_{1,v}}$ is solvable, one can choose a finite solvable totally real field extension $F^+_1$ of $F_1$ 
satisfying the condition (2),(3),(6) such that  $F^+_1$ is linearly disjoint from 
$\overline{F^+}^{{\rm Ker}(\overline{r})}(\zeta_p)$ over $F^+$. Note that 
the condition (6) is guaranteed by Theorem \ref{bc} with  a suitable solvable extension regarding 
the Grothendieck's monodromy theorem together with the local global compatibility.  
Making finite places inert by extending 
the base field if necessary, we may assume that our $L^+_1$ also satisfies (5), (7) and also (4) by taking 
a suitable quadratic extension. Everything here is done under keeping the linearly independent-ness over $F^+$ 
and hence (1) is guaranteed.  

Notice that $[F^+_1:\Q]$ is even and as in Section \ref{JL}, we consider a 
definite quaternion algebra $B$ with the center $F^+_1$ which is ramified at 
precisely the infinite places. Let $\widetilde{\pi}$ be the automorphic 
representation of $G_B(\A_{F^+_1})$ of weight zero  
corresponding to $\pi_{F^+_1}$ via Jacquet-Langlands correspondence. 
We identify the local components of $\widetilde{\pi}$ and $\pi_{F^+_1}$ via  the isomorphism (\ref{isom}). 

Choose a finite place $v_0\nmid p$ such that 
\begin{enumerate}
\item $v_0$ does not split completely in $F^+_1(\zeta_p)$;
\item $\widetilde{\pi}_{F^+_1}$ is unramified at $v_0$ and ${\rm ad} \overline{r}({\rm Frob}_{v_0})=1$;
\item $v_0$ satisfies the condition in Lemma \ref{double-coset}-(2). 
\end{enumerate}
Here we consider $\overline{r}$ as a Galois representation which takes the values in ${\rm GL}_4(\bF)$ with 
a finite field $\F$ and also 
its adjoint representation accordingly.  
In fact, the condition in (3) is satisfied for all but finitely many places. 
For (1) and (2), it follows from Chebotarev density theorem. 
Once we have such a finite place $v_0$, (1),(2) imply the non-trivial action of  
${\rm Frob}_{v_0}$ on ${\rm ad}^0\overline{r}(1)$ where ${\rm ad}^0\overline{r}={\rm ad}\overline{r}/\F\cdot 1$. 
By the local Tate duality, $H^2(G_{F^+_{1,v_0}},{\rm ad}^0\overline{r})\simeq 
H^0(G_{F^+_{1,v_0}},{\rm ad}^0\overline{r}(1))=0$. 
It follows from this that the local framed deformation space of $\overline{r}|_{G_{F^+_{1,v_0}}}$ 
with a fixed central character is unobstructed and hence the local framed deformation space 
with a fixed central character is irreducible. Since $\widetilde{\pi}_{F^+_1}$ is unramified at $v_0$, any (global) lift of $\overline{r}$ 
with a fixed central character is unramified at $v_0$. This follows from the constancy of the inertia group 
$I_{F^+_{1,v_0}}$ on any geometric connected component of the framed deformation space. We will use this fact later. 

Recall the parahoric subgroup $K_{P_i,v}$ ($i=1,2$) for a parabolic subgroup $P_i$ in Section \ref{pAMF} and 
the isomorphism (\ref{isom}) for each finite place $v$ of $F^+_1$. 
We denote by $K_{P_i,1,v}$ is the subgroup of the Iwahori subgroup 
$K_{P_i,v}$ consisting of all elements whose reduction modulo $\pi_v$ belong to $N_i(\F_v)$ where 
$N_i$ is the unipotent radical of $P_i$. 
Let us choose the open compact subgroup $U=\prod_{v}U_v\subset G_B(\A^\infty_{F^+_1})$ such that 
\begin{enumerate}
\item $U_v=G_B(\mathcal{O}_{F^+_{1,v}})$ for all $v|p$, and all $v\neq  v_0$ where the local component 
$\widetilde{\pi}_v$ of $\widetilde{\pi}$ at $v$ is unramified;
\item $U_{v_0}=\iota^{-1}_{v_0}U_{B,1,v_0}$, where $U_{B,1,v_0}$ is the subgroup of the Iwahori subgroup 
$U_{B,v_0}$ consisting of all elements whose reduction modulo $\pi_v$ are unipotent;
\item $U_v=
\left\{\begin{array}{cc}
\iota^{-1}_v K_{P_1,v} & {\rm if}\ \widetilde{\pi}_v\ {\rm is \ unipotent\ of\ type\ \mathcal{N}_1 }\\  
\iota^{-1}_v K_{P_2,v} & {\rm if}\ \widetilde{\pi}_v\ {\rm is \ unipotent\ of\ type\ \mathcal{N}_2 }\\  
\iota^{-1}_v U_{B,v_0} & {\rm if}\ \widetilde{\pi}_v\ {\rm is \ unipotent\ of\ type\ \mathcal{N}_3 }
\end{array}\right.
 $ if $v\nmid p\cdot v_0$, and $\widetilde{\pi}_v$ is unipotent of type $\mathcal{N}_i$ for some $1\le i \le 3$. 
\end{enumerate} 
Then thanks to the choice of $v_0$, the group $U$ satisfies the condition of Lemma \ref{double-coset}-(2). 
Define the following set of finite places of $F^+_1$; 
\begin{equation}\label{r}
R:=\{v\nmid p\ |\ U_v\neq G_B(\mathcal{O}_{F^+_{1,v}}) \}=\{v_0\}\cup \{ v\nmid p\ |\ 
\widetilde{\pi}_v  {\rm \ is\  unipotent\  of\  type\  \mathcal{N}_i\  for\  some}\ 1\le i \le 3 \}.
\end{equation}
Clearly $R$ is a finite set. 

According to the local type of $\widetilde{\pi}_v$, we define the weight $\tau=\otimes_{v|p}\tau_v$ of 
$G_B(\mathcal{O}_{F^+_1,p})=\prod_{v|p}G_B(\mathcal{O}_{F^+_1,v})$ as follows.  
Fix a finite extension $K\subset \bQ_p$ of $\Q_p$ which includes the images of all embeddings 
$F^+_1\hookrightarrow \bQ_p$. Put $\mathcal{O}=\mathcal{O}_K$. 
When $v|p$ and $\widetilde{\pi}_v$ is unramified, then we define $\tau_v$ to be the 
trivial representation over $\mathcal{O}$ of $G_B(\mathcal{O}_{F^+_1,v})$. 
When $v|p$ and $\widetilde{\pi}_v$ is ramified, then $\widetilde{\pi}_v$ is unipotent of type $\mathcal{N}_j$ 
for some $1\le j \le 3$. In either unipotent type, for a fixed maximal parabolic $P_i$ ($i=1,2$), applying Proposition \ref{switching}-(2), there exists a 
$\mathcal{O}$-lattice $W(\widetilde{\pi})$ of $r_{P_i}(\widetilde{\pi})$ as a $K_{P_i,v}/K^+_{P_i,v}$-module, a 
dihedral depth zero supercuspidal 
representation $\pi'$ of ${\rm GL}_2(F^+_{1,v})$ for an unramified quadratic extension $M_v/F^+_{1,v}$, and a quasi-character $\mu:(F^+_{1,v})^\times \lra \C^\times$ such that 
$W(\widetilde{\pi})\otimes \F_K$ and $W(\Pi_i(\pi',\mu))\otimes \F_K$ share a common constituent. 
We view $W(\widetilde{\pi})$ as a $K_P$-module via the natural surjection $K_{P_i,v}\lra K_{P_i,v}/K^+_{P_i,v}$ and 
pick an irreducible subquotient $\tau_v=\tau^{(i)}_v$ of 
${\rm Ind}^{G_B(\mathcal{O}_{F^+_{1,v}})}_{K_{P,v}}W(\widetilde{\pi})/\mathcal{O}$ so that its extension to $\C$ is 
equivalent to some irreducible constituent of $\widetilde{\pi}_v|_{G_B(\mathcal{O}_{F^+_1,v})}$. 
Let $\psi:\A^\times_{F^+_1}\lra \C^\times$ be the central character of $\pi$. It is easy to check, by the parity condition, that it is a finite character and therefore factors through $(\A^\infty_{F^+_1})^\times$. 
By construction, clearly we have $(S_{0,\tau,\psi}(U,\mathcal{O})\otimes_{\mathcal{O},\iota_p}\C)[\widetilde{\pi}]\neq 0$. 

Recall the Hecke algebra $T^R_{\mathcal{O}}$ in (\ref{Hecke-algebra}) for our $R$ defined in (\ref{r}).  
Enlarging $\mathcal{O}$ if necessary, we may assume that there exists a $T^R_{\mathcal{O}}$-Hecke eigenform $f$ in 
$S_{0,\tau,\psi}(U,\mathcal{O})$ generating $\widetilde{\pi}$ such that 
 all $T^R_{\mathcal{O}}$-Hecke eigenvalues of $f$ are 
defined over $\mathcal{O}$. 
The eigenform $f$ yields a surjection $T^R_{\mathcal{O}}\lra \mathcal{O}\lra \F_{\mathcal{O}}:=
\mathcal{O}/m_{\mathcal{O}}$ as an $\mathcal{O}$-algebra. 
Let $\mathfrak{m}$ be the kernel of this surjective homomorphism. It follows that 
$S_{0,\tau,\psi}(U,\mathcal{O})_{\mathfrak m}\neq \{0\}$. 
  
Similarly, we define the weight $\tau'=\tau'^{(i)}=\otimes_{v|p}{\tau'_v}^{(i)}$ of 
$G_B(\mathcal{O}_{F^+_1,p})$ whose local component ${\tau'_v}^{(i)}$ is defined by using 
$W(\Pi_i(\pi',\mu))$ for any $v|p$.  
By Lemma \ref{lift} and Lemma \ref{switching}-(2), we have $S_{0,\tau',\psi}(U,\mathcal{O})_{\mathfrak m}\neq \{0\}$ 
for $\tau'={\tau'}^{(i)}$ with any fixed $i\in \{1,2\}$. 
It shows that there exists a $T^R_{\mathcal{O}}$-Hecke eigen form $g\in S_{0,\tau',\psi}(U,\mathcal{O})$ such that 
$\br_{\pi_g,\iota_p}\simeq \br_{\pi_{F^+_1,\iota_p}}$ where 
$\pi_g$ stands for the image of the representation $\widetilde{\pi}_g$ of $G_B(\A_{F^+_1})$ corresponding to 
$g$ under the Jacquet-Langlands correspondence. 
Since $\pi_{g,v}=W(\Pi_i(\pi',\mu))$ for each $v|p$, after a suitable solvable base change,   
$\pi_{g,v}$ becomes Klingen ordinary for any $i=1,2$ and further it has the desired increasing filtration for 
each $i=1,2$ by choosing $\mu$ suitably in conjunction with Hodge-Tate weight. 
Furthermore, it is not ordinary in the usual sense, namely, it is not Borel ordinary.   

In order to make the representation unramified everywhere after a suitable solvable base change, 
we need to analyze $\widetilde{\pi}_v$ for each $v\in R$.  
First of all, as observed before, $\widetilde{\pi}_{v_0}$ is unramified since any local deformation with a fixed 
determinant of 
$\overline{r}|_{G_{F^+_1,v_0}}$ is unramified. 
Therefore, we may consider the case when $v\neq v_0$ and $v\in R$. 
Then $\widetilde{\pi}_v$ is of unipotent of type $\mathcal{N}_i$ for some $1\le i \le 3$. 
Choose a quasi character $\chi_v:\mathcal{O}^\times_{F^+_1,v}\lra \mathcal{O}^\times$ which factors through 
$\F^\times_v$ and it also satisfies $\chi^2_v\neq \textbf{1}$. Notice that $|\F^\times_v|=N(v)-1\equiv 0$ mod $p$ by the condition (7) for $F^+_1$. It follows from this that $\chi_v\equiv \textbf{1}$ mod $m_{\mathcal{O}}$.  
By Table A.15, p.297 of \cite{RS}, we see that $\pi_{F+_1,v}$ has an Iwahori fixed vector. 
Therefore, it is still hold that $g\in S_{0,\tau,\psi}(U,\mathcal{O})$ even if  
we replace the local component $U_v$ of $U$ with the Iwahori subgroup $U_{B,v}$ for any $v\in R\setminus\{v_0\}$. 
We denote by ${\rm diag}(g)\in (\mathcal{O}^\times_{F^+_1,v})^4$ the diagonal part of $g\in U_v=U_{B,v}$.  
We define the character $\widetilde{\chi}_v:U_{B,v}\lra \mathcal{O}^\times$ by 
sending $g$ to $\chi_v(a)\chi^{-1}_v(b)$ for $g\in U_v$ with ${\rm diag}(g)=(a,b,c,d)$. Clearly, 
it is well-defined and it satisfies that $\widetilde{\chi}_v(z)=1$ for any $z\in Z_{G_B}(F^+_{1,v})\cap U_{B,v}$ 
and $\widetilde{\chi}_v\equiv \textbf{1}$ mod $v$. By (an easy variant of) Lemma \ref{lift}, 
we see that $S_{0,\tau',\psi,\chi_{\Sigma}}(U,\mathcal{O})_{\mathfrak{m}}\neq \{0\}$ for $\Sigma=R\setminus\{v_0\}$ in 
Definition \ref{dfn-pAMF}. 
It shows that there exists a $T^R_{\mathcal{O}}$-Hecke eigenform $h\in S_{0,\tau',\psi,\chi_{\Sigma}}(U,\mathcal{O})_{\mathfrak{m}}$ such that 
$\br_{\pi_h,\iota_p}\simeq \br_{\pi_{F^+_1,\iota_p}}$. 
By definition, $\pi_{h,v}$ has an Iwahori fixed vector for each $v\in R\setminus\{v_0\}$ and it has to be  
a subquotient of a principal series representation $\chi_1\times \chi_2\rtimes \sigma$ for some 
quasi characters $\chi_1,\chi_2,\sigma$ of $(F^+_{1,v})^\times$. Notice that $\psi$ is a class field character since 
Observing the action of diagonal elements in $U_{B,v}$ we see that $\chi_1|_{\mathcal{O}^\times_{F^+_1,v}}=\chi_v \omega_v$ 
and $\chi_2|_{\mathcal{O}^\times_{F^+_1,v}}=\chi^{-1}_v \omega_v$ for some character $\omega_v:
\mathcal{O}^\times_{F^+_1,v}\lra \C^\times$. Recall that $\chi_v$ is chosen to be neither trivial nor of order 2.   
By the classification of \cite{RS} (see Group I in Section 2.2, p.37),  $\chi_1\times \chi_2\rtimes \sigma$ has to be irreducible. 
Therefore, $\pi_{h,v}$ is a principal series representation which is possibly ramified for each $v\in R\setminus\{v_0\}$.   
The result follows easily by replacing $F^+_1$ with an appropriate totally solvable extension $L^+/F^+_1$.    
\end{proof}

The main technical result of this paper is the following proposition:
\begin{prop}\label{tech-main}$($the existence of a potentially ordinary lift$)$  
Let $F^+$ be a totally real extension of $\Q$ and $p$ be an odd prime which is split completely in $F^+$. 
Let $\br:G_{F^+}\lra {\rm GSp}_4(\bF_p)$ be an irreducible automorphic mod $p$ Galois 
representation. Assume further that $\br(G_{F^+(\zeta_p)})$ is adequate when we view $\br$ as a 
representation which takes the values in ${\rm GL}_4(\bF_p)$. 
Then there exists a solvable totally real finite extension $L^+$ of $F^+$ such that 
\begin{enumerate}
\item  $\br(G_{L^+(\zeta_p)})$ is adequate; 
\item  $\br|_{G_{L^+_v}}$ is trivial for all places $v|p$ of $L^+$;
\item  there exists a cuspidal automorphic representation $\pi$ of ${\rm GSp}_4(\A_{L^+})$ of weight zero such that 
\begin{enumerate}
\item $\br|_{G_{L^+}}\simeq \br_{\pi,\iota_p}$,  
\item $\pi$ is unramified at all finite places,
\item for all places $v|p$ of $L^+$, $\rho_{\pi,\iota_p}|_{G_{L^+_v}}$ is ordinary in the sense of 
Section 7.5 of \cite{gg}. 
\end{enumerate}
\end{enumerate}
\end{prop}
\begin{proof} By Lemma \ref{non-ord-lift}, there exists a totally real solvable finite extension $F^+_2$ of $F^+$ 
such that 
\begin{enumerate}
\item $\br|_{G_{F^+_2(\zeta_p)}}$ is irreducible and $\br(G_{F^+_2(\zeta_p)})$ is adequate; 
\item  $\br|_{G_{F^+_{2,v}}}$ is trivial for all places $v|p$ of $F^+_2$;
\item for each place $v|p$, $[F^+_{2,v}:\Q_p]\ge 2$ (this condition will be used when we choose $\tau_0$ soon later); 
\item there exists a cuspidal automorphic representation $\pi_2$ of ${\rm GSp}_4(\A_{F^+_2})$ of weight zero such that 
\begin{enumerate}
\item $\br|_{G_{F^+_2}}\simeq \br_{\pi_2,\iota_p}$,  
\item $\pi_2$ is unramified at all finite places,
\item for all places $v|p$ of $F^+_2$, $\rho_{\pi,\iota_p}|_{G_{L^+_v}}$ is not ordinary, but Klingen ordinary 
and it gives an $\mathcal{O}$-valued point of the scheme 
(\ref{imp-scheme}) for $P_i$ and for the integer ring $\mathcal{O}$ of a suitable finite extension of $\Q_p$.  
\end{enumerate}
\end{enumerate}
We will apply (a mild modification of) Lemma 2.1.2 of \cite{BGG1}. 
However we use the convention that the Hodge-Tate weight of the $p$-adic cyclotomic 
character is $+1$ while it is defined to be $-1$ in loc.cit..
We now employ the Lemma in the following setting:
\begin{enumerate}
\item $F^+=F^+_2$;
\item $p>2$ (this condition is corresponding to the third assumption of Lemma 2.2 of \cite{BGG1}. 
The extra condition ``$l>2n-2$" can be removed by using the notion of adequacy with the condition 
$l\nmid n$.);
\item $n=m=4$;
\item $F^{({\rm avoid})}=F^+_2\overline{F^+}^{{\rm Ker}(\br)}(\zeta_p)$
\item $T=\emptyset$;
\item $\eta=\eta'=1$;
\item $\{h_{1,\tau},h_{2,\tau},h_{3,\tau},h_{4,\tau}\}=\{0,1,2,3\}$ for each $\tau:F^+_{2,v}\hookrightarrow \bQ_p,\ 
v|p$; furthermore, for each place $v|p$, there exists $\tau_0:F^+_{2,v}\hookrightarrow \bQ_p$ such that 
for any $\tau\neq \tau_0$, 
$$h_{1,\tau}=0,h_{2,\tau}=1,h_{3,\tau}=2,h_{4,\tau}=3,$$
and 
$$h_{1,\tau_0}=0,h_{2,\tau_0}=2,h_{3,\tau_0}=1,h_{4,\tau_0}=3;$$
\item  $h'_{1,\tau}=0,h'_{2,\tau}=4,h'_{3,\tau}=5$, and $h'_{4,\tau}=9$ for all $\tau:F^+_{2,v}\hookrightarrow \bQ_p$;
\item $w=h_{1,\tau}+h_{4,\tau}=h_{2,\tau}+h_{3,\tau}=3,w'=h'_{1,\tau}+h'_{4,\tau}=h'_{2,\tau}+h'_{3,\tau}=9$ 
for any $\tau:F^+_{2,v}\hookrightarrow \bQ_p,\ v|p$.
\end{enumerate}
Then we obtain a CM extension $M/F^+_2$ which is linearly disjoint from $F_2\overline{F^+}^{{\rm Ker}(\br)}(\zeta_p)$ 
over $F^+$ together with two continuous characters 
$$\theta,\theta':G_M\lra \overline{\Z}^\times_p$$
such that 
\begin{enumerate}
\item $\theta$ and $\theta'$ are congruent modulo $p$ each other;
\item $(\br|_{G_{F^+_2}}\otimes {\rm Ind}^{G_{F^+_2}}_{G_M}\overline{\theta})(G_{F^+_2(\zeta_p)})$ is adequate.  
Indeed, this can be checked as follows. By construction (cf. the third line and the line -13 in p.548 of \cite{BGGT}), 
$({\rm Ind}^{G_{F^+_2}}_{G_M}\overline{\theta})|_{G_{F^+_2(\zeta_p)}}$ is irreducible and adequate by Lemma A.2-(1), p.908 
of \cite{Thorne}. The claim follows from Lemma A.2-(2) in loc.cit. together with the adequacy and the irreducibility of 
$\br|_{G_{F^+_2(\zeta_p)}}$;
\item $({\rm Ind}^{G_{F^+_2}}_{G_M}\theta)\simeq ({\rm Ind}^{G_{F^+_2}}_{G_M}\theta)^\vee \otimes \varepsilon^3$ 
where $\varepsilon$ stands for the $p$-adic cyclotomic character whose Hodge-Tate weight is normalized to be $+1$;
\item $({\rm Ind}^{G_{F^+_2}}_{G_M}\theta')\simeq ({\rm Ind}^{G_{F^+_2}}_{G_M}\theta')^\vee \otimes \varepsilon^9 
\widetilde{\omega}^{-6}$ where $\widetilde{\omega}$ is the Teichm\"uller lift of mod $p$ cyclotomic character $\omega:=\overline{\varepsilon}$;
\item for each $v$ above $p$, up to conjugacy, the following decomposition into characters holds:
$$({\rm Ind}^{G_{F^+_2}}_{G_M}\theta)|_{G_{F^+_2,v}}\simeq \chi^{(v)}_1\oplus \chi^{(v)}_2
\oplus \chi^{(v)}_3\oplus \chi^{(v)}_4$$      
where, for each embedding $\tau:F^+_{2,v}\hookrightarrow \bQ_p$ we have ${\rm HT}_\tau(\chi^{(i)}_v)=h_{i,\tau}$. 
Similarly, for each $v$ above $p$, up to conjugacy, the following decomposition into characters holds:
$$({\rm Ind}^{G_{F^+_2}}_{G_M}\theta')|_{G_{F^+_2,v}}\simeq \chi'^{(v)}_1\oplus \chi'^{(v)}_2
\oplus \chi'^{(v)}_3\oplus \chi'^{(v)}_4$$      
where, for each embedding $\tau:F^+_{2,v}\hookrightarrow \bQ_p$ we have ${\rm HT}_\tau(\chi'^{(i)}_v)=h'_{i,\tau}$. 
By the choice of $h'_{i,\tau}$,  ${\rm Ind}^{G_{F^+_2}}_{G_M}\theta'$ is ordinary at each $v$ above $p$. 
On the other hand, ${\rm Ind}^{G_{F^+_2}}_{G_M}\theta$ is semi-stable at each $v$ above $p$. Consider the  
filtered module $D_v$ corresponding to $V_v:=({\rm Ind}^{G_{F^+_2}}_{G_M}\theta)|_{G_{F^+_2,v}}$, then the choice of the 
Hodge-Tate weights $\{h_{i,\tau}\}$ yields the filtration $0=D_0\subset D_1\subset D_2\subset D_3=D_v$ such that 
the rank of $D_1$ is 1 and the rank of $D_2$ is 3. It gives rise to a 
$G_{F^+_2,v}$-stable increasing filtration for Klingen parabolic $P_2$;
\begin{equation}\label{fil-K}
0={\rm Fil}^{(2)}_0\subset {\rm Fil}^{(2)}_1\subset {\rm Fil}^{(2)}_2\subset {\rm Fil}^{(2)}_3=V_v.
\end{equation}
By the choice of $\{h_{i,\tau}\}$ again,  we see that ${\rm gr}^{(2)}_1$ has the constant Hodge-Tate weight 1 in 
$\tau$ and  
 ${\rm gr}^{(2)}_3$ has the constant Hodge-Tate weight 3 in 
$\tau$ while ${\rm gr}^{(2)}_2$ is of rank 2 and it has Hodge-Tate weights $\{h_{2,\tau},h_{3\tau}\}={2,1}$ 
which is not constant in $\tau$. It follows from this that ${\rm gr}^{(2)}_2(-1)$ is non-ordinary and 
it is also potentially 
Barsotti-Tate by Theorem 0.3 of \cite{KF}. In addition, the filtration yields a 
$\mathcal{O}$-valued point of the scheme (\ref{imp-scheme}) for the integer ring $\mathcal{O}$ of 
some finite extension of $\Q_p$.   
\end{enumerate}

Let $F^+_3/F^+_2$ be a solvable extension of totally real fields such that 
\begin{enumerate}
\item $F^+_3$ is linearly disjoint from $\overline{F^+}^{{\rm Ker}(\br)}(\zeta_p)$ over $F^+$. In particular, 
$\br|_{G_{F^+_3(\zeta_p)}}$ is irreducible and $\br(G_{F^+_3(\zeta_p)})$ is adequate; 
\item $({\rm Ind}^{G_{F^+_2}}_{G_M}\theta)|_{G_{F^+_3}}$ and $({\rm Ind}^{G_{F^+_2}}_{G_M}\theta')|_{G_{F^+_3}}$ 
are both unramified at all places of $F^+_3$ not lying over $p$ and crystalline at all places $v$ above $p$;
\item if $v$ is a place of $F^+_3$ dividing $p$, then 
$({\rm Ind}^{G_{F^+_2}}_{G_M}\overline{\theta})|_{G_{F^+_{3,v}}}$ is trivial;
\item if $v$ is a place of $F^+_3$ dividing $p$, then $F^+_{3,v}$ contains a primitive $p$-th root of 
unity. 
\end{enumerate}
Let $F_3/F^+_3$ be a quadratic CM extension which is linearly disjoint from $M\overline{F^+}^{{\rm Ker}(\br)}(\zeta_p)$ 
over $F^+$, and in which all places of $F^+_3$ lying over $p$ are split completely in $F_3$. 
Choose a finite extension $K\subset \bQ_p$ of $\Q_p$ such that the images of $\theta,\theta'$, and $\rho_{\pi_2,\iota_p}$ 
are defined over $K$. Enlarging $K$ if necessary, we may assume that $K$ contains all embeddings 
$F^+_3\hookrightarrow \bQ_p$.  
Choosing suitable lattices, we may regard 
${\rm Ind}^{G_{F^+_2}}_{G_M}\theta,\ {\rm Ind}^{G_{F^+_2}}_{G_M}\theta'$, and 
$\rho_{\pi_2,\iota_p}$ as representations to ${\rm GSp}_4(\mathcal{O})$ and also to ${\rm GL}_4(\mathcal{O})$ 
by the natural inclusion ${\rm GSp}_4(\mathcal{O})\subset {\rm GL}_4(\mathcal{O})$. 
Further, we may suppose that ${\rm Ind}^{G_{F^+_2}}_{G_M}\overline{\theta}=
{\rm Ind}^{G_{F^+_2}}_{G_M}\overline{\theta'}$ since $\theta\equiv \theta'$ modulo $p$. 
Let $\chi_{\pi_2}:G_{F^+_2}\lra \mathcal{O}^\times$ be the character corresponding to the 
central character of $\pi_2$ via class field theory. Note that it holds that 
$\rho_{\pi_2,\iota_p}\simeq \rho_{\pi_2,\iota_p}^\vee \chi_{\pi_2}$.  

We are now ready to apply Theorem 3.5.1 of \cite{BGG1} in the following setting:
\begin{enumerate}
\item $m=n=4$ and $l=p>2$;
\item $F=F'=F_3$
\item $\overline{r}:=\br|_{G_{F_3}}=\br_{\pi_2,\iota_p}|_{G_{F_3}}$;
\item $r'=({\rm Ind}^{G_{F^+_2}}_{G_M}\theta)|_{G_{F_3}}$ and $\chi'={\e}^3$;
\item $\chi=\e^6\widetilde{\omega}^{-6}\chi_{\pi_2}|_{G_{F^+_3}}$ so that $\overline{\chi}=
\overline{\chi}_{\pi_2}|_{G_{F^+_3}}$;
\item $r''=\rho_{\pi_2,\iota_p}|_{G_{F_3}}\otimes ({\rm Ind}^{G_{F^+_2}}_{G_M}\theta')|_{G_{F_3}}$ and 
$\chi''={\e}^9\widetilde{\omega}^{-6}\chi_{\pi_2}|_{G_{F^+_3}}$;
\item $\widetilde{S}_p$ is any set of places of $F_3$ consisting of exactly one place above 
each place in $S_p:=\{v\ |\ v|p\ {\rm in}\ F^+_3\}$. 
\item for each place $\widetilde{v}\in \widetilde{S}_p$, $a_{\widetilde{v}}\in 
(\Z^4_+)^{{\rm Hom}(F^+_{3,\widetilde{v}},\bQ_p)}$ is given by 
$$a_{\tau,1}=9,a_{\tau,1}=5,a_{\tau,3}=4,a_{\tau,4}=0$$
for each $\tau:F^+_{3,\widetilde{v}}\hookrightarrow \bQ_p$. 
\item for each place $\widetilde{v}\in \widetilde{S}_p$, $R_{\widetilde{v}}$ is the unique irreducible 
ordinary component of $R^{{\rm v}_{a_{\widetilde{v}}},{\rm cr}}_{\overline{r}|_{G_{F_3,\widetilde{v}}}}$ 
(see Section \ref{LDR}). 
The uniqueness which implies the irreducibility of the ordinary part is due to Lemma 3.1.4, p.1389 of \cite{Ge-pub} since $\overline{r}|_{G_{F_3,\widetilde{v}}}=1$ 
as a representation to ${\rm GL}_4(\bF_p)$ 
by the second condition for $F^+_2$ 
and $\overline{\e}|_{G_{F_3,\widetilde{v}}}=1$ by the fourth condition for $F^+_3$. 
\end{enumerate}
We now check the hypotheses of Theorem 3.5.1 of \cite{BGG1} hold. 
A minor modification is necessary because they 
employed ``2-big-ness" which should be replaced with adequacy. 
However, nothing will be changed if we do use adequacy instead of 2-big-ness. 
Therefore, we just give some comments about this in a course of the followings. 
We follow the numbering in the claim in Theorem 3.5.1 of  \cite{BGG1}. 
\begin{itemize}
\item for (1), since $\overline{\chi}=
\overline{\chi}_{\pi_2}|_{G_{F^+_3}}$, it follows from Section 3.1.3 of \cite{BGG1} that 
$\overline{r}^c\simeq \overline{r}^\vee\overline{\chi}$ and $\overline{r}$ is odd;
\item for (3), it follows from the third property for $\theta,\theta'$ and  Section 3.1.3 again that 
$r'^c\simeq r^\vee\chi'|_{G_{F_3}}$;
\item for (2) and (4), since $\rr=\br_{G_{F_3}}=\br_{\pi_2,\iota_p}|_{G_{F_3}}$, 
it is unramified outside $p$. By the choice of $F_3$, $r'$ is crystalline at all places above $p$. 
In fact, any place of $F^+_3$ is split completely in $F_3$, the decomposition group is nothing changed and 
the local property of $r'|_{G_{F^+_{3,v}}}$ is preserved.   
The representation $r'$ is also unramified outside $p$ since the set $T$ is chosen to be the empty set in the construction of $\theta$. For $\chi,\chi',\chi'',r''$, they are unramified outside $p$ by choice of $F_3$ and properties of $\pi_2$;
\item for (5), $r''$ is automorphic of level prime to $p$ by Proposition 5.1.3, p.451 of \cite{BGG-Sato} since 
the Hodge-Tate weights $\{h'_{i,\tau}\}$ are chosen so that $r''$ has regular Hodge-Tate weights;
\item for (6), $\rr''=
\br_{\pi_2,\iota_p}|_{G_{F_3}}\otimes  ({\rm Ind}^{G_{F^+_2}}_{G_M}\overline{\theta}')|_{G_{F_3}}=
\br_{\pi_2,\iota_p}|_{G_{F_3}}\otimes  ({\rm Ind}^{G_{F^+_2}}_{G_M}\overline{\theta})|_{G_{F_3}}=\rr\otimes \rr'$ and 
obviously $\chi''|_{G_{F_3}}=\chi|_{G_{F_3}} \chi'|_{G_{F_3}}$ since $\chi''=\chi\chi'$;
\item the conditions (7) and (8) concerning Taylor-Wiles conditions are replaced with the adequacy of 
$\rr''(G_{F_3(\zeta_p)})$ which follows from the adequacy 
of $(\br|_{G_{F^+_2}}\otimes {\rm Ind}^{G_{F^+_2}}_{G_M}\overline{\theta})(G_{F^+_2(\zeta_p)})$ as seen before;
\item for (9), this part is a main heart of the proof here. 
For each place $\widetilde{v}\in \widetilde{S}_p$ lying over $v\in S_p$, 
let $\rho_v:G_{F_3,\widetilde{v}}\lra {\rm GL}_4(\mathcal{O})$ be a lift of $\rr|_{G_{F_3,\widetilde{v}}}$ 
which corresponds to a closed point of $R_{\widetilde{v}}$. 
We may assume that $G_{F_3,\widetilde{v}}=G_{F^+_3,v}$. 
Then $\rho_v$ and 
$({\rm Ind}^{G_{F^+_2}}_{G_M}\theta')|_{G_{F_{3,\widetilde{v}}}}$ are both crystalline and 
ordinary of the same Hodge-Tate weights for any embedding $F_3,\widetilde{v}\hookrightarrow \bQ_p$,  
and reduce to the trivial representation modulo the maximal ideal of $\mathcal{O}$. 
Since $R_{\widetilde{v}}$ is irreducible, 
$\rho_v\sim ({\rm Ind}^{G_{F^+_2}}_{G_M}\overline{\theta}')|_{G_{F_{3,\widetilde{v}}}}$ 
in the sense of Definition 3.1.4 of \cite{BGG1}. 
On other hand, $r'|_{G_{F_{3,\widetilde{v}}}}=
({\rm Ind}^{G_{F^+_2}}_{G_M}\theta)|_{G_{F_{3,\widetilde{v}}}}$ and $\rho_{\pi_2,\iota_p}|_{G_{F_{3,\widetilde{v}}}}$ 
are both crystalline of the same Hodge-Tate weight. 
As observed in (\ref{fil-K}) before, $r'|_{G_{F_{3,\widetilde{v}}}}$ has a increasing filtration 
$\{{\rm Fil}^{(2)}_i\}_{i=0}^3$ such that ${\rm gr}^{(2)}_2(-1)$ is non-ordinary Barsotti-Tate. 
By construction of $\pi_2$, we see that $\rho_{\pi_2,\iota_p}|_{G_{F_{3,\widetilde{v}}}}$ has a similar 
filtration $\{{\rm Fil'}^{(2)}_i\}_{i=0}^3$. Then two representations ${\rm gr}^{(2)}_2(-1)$ and 
${\rm gr'}^{(2)}_2(-1)$ define two geometric points in ${\rm Spec}\hspace{0.5mm}R$ defined in (\ref{Kisin-defo}). 
By Proposition 2.3 of \cite{Gee-H}, these two points both lie on the same connected component $\mathcal{V}$. 
This implies that $r'|_{G_{F_{3,\widetilde{v}}}}$ and  $\rho_{\pi_2,\iota_p}|_{G_{F_{3,\widetilde{v}}}}$ 
define two geometric points which both lie on a connected component of $\mathcal{G}_{\mathcal{V},2}[\frac{1}{p}]$. 
By Proposition \ref{conn2} and Remark \ref{conn3}, we have $r'|_{G_{F_{3,\widetilde{v}}}}\sim\rho_{\pi_2,\iota_p}|_{G_{F_{3,\widetilde{v}}}}$. 
By the results for ``$\sim$" (Section 3.3,3.4 of \cite{BGG-Sato}), 
$$\rho_v\otimes r'|_{G_{F_{3,\widetilde{v}}}}\sim 
({\rm Ind}^{G_{F^+_2}}_{G_M}\theta')|_{G_{F_{3,\widetilde{v}}}}\otimes 
\rho_{\pi_2,\iota_p}|_{G_{F_{3,\widetilde{v}}}}
\sim \rho_{\pi_2,\iota_p}|_{G_{F_{3,\widetilde{v}}}}\otimes 
({\rm Ind}^{G_{F^+_2}}_{G_M}\theta')|_{G_{F_{3,\widetilde{v}}}}\sim r''|_{G_{F_{3,\widetilde{v}}}}.$$
Summing up, the hypotheses of Theorem 3.5.1 of \cite{BGG1} are checked completely with a minor modification with 
adequacy. 
\end{itemize}
We conclude that, after possibly extending $\mathcal{O}$, there exits a continuous lifting 
$r:G_{F_3}\lra {\rm GL}_4(\mathcal{O})$ such that 
\begin{enumerate}
\item $r$ is unramified outside $p$;
\item $r|_{G_{F_{3,\widetilde{v}}}}$ is crystalline and ordinary at each place $\widetilde{v}\in \widetilde{S}_p$ 
(in fact, at all places $\widetilde{v}$ of $F_3$ above $p$), 
with Hodge-Tate weights $0,4,5$, and $9$;
\item $r^c\simeq r^\vee\chi|_{G_{F_3}}$;
\item $r\otimes ({\rm Ind}^{G_{F^+_2}}_{G_M}\theta)|_{G_{F_3}}$ is 
automorphic of level prime to $p$.
\end{enumerate} 
Notice that $\rr|_{G_{MF_3}}=\br|_{G_{MF_3}}$ is irreducible since $F_3$ is linearly disjoint from 
$M\overline{F^+}^{{\rm Ker}(\br)}(\zeta_p)$ 
over $F^+$. 
By Proposition 5.1.1 of \cite{BGG1} and Lemma 5.9 of \cite{Ge-pub}, $r$ is $\iota_p$-ordinary automorphic 
of level prime to $p$. 
By Lemma 6.1.2 of \cite{BGG1} there exists a solvable finite extension $F^+_4/F_3$ of totally fields 
together with a RAECSDC automorphic representation $\pi_3$ of ${\rm GL}_4(\A_{F_4})$, 
where $F_4=F^+_4F_3$ such that 
\begin{enumerate}
\item $F^+_4$ is linearly disjoint from $F_3\overline{F^+}^{{\rm Ker}(\br)}(\zeta_p)$ over $F^+_3$;
\item $\pi_4$ is of weight zero;
\item $\br_{\pi_4,\iota_p}\simeq \rr|_{G_{F_4}}=\br|_{G_{F_4}}$;
\item $\rho_{\pi_4,\iota_p}^c\simeq \rho_{\pi_4,\iota_p}$;
\item $\rho_{\pi_4,\iota_p}^c\simeq \rho_{\pi_4,\iota_p}^\vee \widetilde{\chi}|_{G_{F_4}}$ where 
$\chi:G_{F^+_3}\lra \mathcal{O}^\times$ is the Teichm\"uller lift of $\overline{\chi}:G_{F^+_3}\lra \bF^\times_p$.  
\end{enumerate}
Let $S'_p$ be the set of places pf $F^+_4$ above $p$ and $\widetilde{S}'_p$  be 
any set of places of $F_3$ consisting of exactly one place above 
each place in $S'_p:=\{v\ |\ v|p\ {\rm in}\ F^+_4\}$. 
For each place $\widetilde{v}\in \widetilde{S}'_p$, $a'_{\widetilde{v}}\in 
(\Z^4_+)^{{\rm Hom}(F^+_{3,\widetilde{v}},\bQ_p)}$ is given by 
$$a'_{\tau,1}=3,a'_{\tau,1}=2,a'_{\tau,3}=1,a'_{\tau,4}=0$$
for each $\tau:F^+_{4,\widetilde{v}}\hookrightarrow \bQ_p$. 
For each place $\widetilde{v}\in \widetilde{S}'_p$, $R_{\widetilde{v}}$ is the unique irreducible 
ordinary component of $R^{{\rm v}_{a'_{\widetilde{v}}},{\rm cr}}_{\br|_{G_{F_4,\widetilde{v}}}}$ 
(see Section \ref{LDR}). 
The uniqueness which implies the irreducibility of the ordinary part is due to Lemma 3.1.4, p.1389 of \cite{Ge-pub} since $\overline{r}|_{G_{F_4,\widetilde{v}}}=1$ 
as a representation to ${\rm GL}_4(\bF_p)$ 
by the second condition for $F^+_2$ 
and $\overline{\e}|_{G_{F_4,\widetilde{v}}}=1$. 
As explained in Section 3 of \cite{BGG1}, we can extend $\br|_{G_{F_4}}$ to 
$\rr_4:G_{F^+_4}\lra \mathcal{G}_4(\bF_p)$ where the algebraic group $\mathcal{G}_4$ is defined in 
Section 3.1.1 of \cite{BGG1}.  
Applying Proposition 1.5.1 of \cite{BGGT} for 
$$\mathcal{S}=(F_4/F^+_4,S_p,\widetilde{S}'_p,\mathcal{O},\rr_4,
\widetilde{\chi}|_{G_{F^+_4}},\{R_{\widetilde{v}}\}_{\widetilde{v}\in \widetilde{S}'_p}),$$
there is a lifting $r_4:G_{F^+_4}\lra {\rm GL}_4(\bQ_p)$ of $\rr|_{G_{F^+_4}}=\br|_{G_{F^+_4}}$, which is 
crystalline and ordinary of Hodge-Tate weight $a'_{\widetilde{v}}$ at each place 
$\widetilde{v}\in \widetilde{S}'_p$. 
Applying Theorem 4.4.1 of \cite{BGGT} but the condition ``$l>2(n+1)$" should be replaced with the adequacy of $\rr(G_{F(\zeta_l)})$ in the notation there, then we have that $r_4|_{G_{F_4}}$ is 
$\iota_p$-ordinary automorphic. By Lemma 1.5 of \cite{BGHT}, $r_4$ is $\iota_p$ ordinary automorphic, 
hence there exists a RAESD cuspidal automorphic representation $\pi_4$ of ${\rm GL}_4(\A_{F^+_4})$ such that 
it is $\iota_p$-ordinary and of weigh zero, and $r_4\simeq \rho_{\pi_4,\iota_p}$. 

What remains is to descend $\pi_4$ to a cuspidal automorphic representation of ${\rm GSp}_4(\A_{F^+_4})$. 
Recall that $\br$ is absolutely irreducible by Proposition \ref{odd-irred}. By the choice of $F^+_4$, 
$\br$ is also absolutely irreducible and so is $r_4$. Clearly, $r_4$ is self-dual since so is $\pi_4$. 
Hence we have an identification $${\rm End}_{\bQ_p[G_{F^+_4}]}(r_4)={\rm Bil}_{G_{F^+_4}}(r_4\times r_4,\bQ_p)$$ 
up to twists, where the latter space stands for $G_{F^+_4}$-equivariant, $\bQ_p$-linear 
bilinear forms on $r_4\times r_4$. However, we have 
$$\bQ_p={\rm End}_{\bQ_p[G_{F^+_4}]}(r_4)={\rm Bil}_{\bQ_p[G_{F^+_4}]}(r_4\times r_4,\bQ_p)=
{\rm Sym}_{\bQ_p[G_{F^+_4}]}(r_4\times r_4,\bQ_p)\oplus {\rm Alt}_{\bQ_p[G_{F^+_4}]}(r_4\times r_4,\bQ_p)$$
where ${\rm Sym}_{\bQ_p[G_{F^+_4}]}(r_4\times r_4,\bQ_p)$ stands for the set of $G_{F^+_4}$-equivariant symmetric bilinear forms while 
${\rm Alt}_{\bQ_p[G_{F^+_4}]}(r_4\times r_4,\bQ_p)$ stands for the set of $G_{F^+_4}$-equivariant alternative bilinear forms. 
Hence either ${\rm Sym}_{G_{F^+_4}}(r_4\times r_4,\bQ_p)\neq \{0\}$ or  
 ${\rm Alt}_{G_{F^+_4}}(r_4\times r_4,\bQ_p)\neq \{0\}$ holds. Observe the reduction of $r_4$ and it has 
a symplectic form as we have started with $\br$. Hence we have 
${\rm Alt}_{G_{F^+_4}}(r_4\times r_4,\bQ_p)\neq \{0\}$. 
By the last three lines of the claim of Theorem 4.26 of \cite{AS1}, we have a generic 
cuspidal automorphic representation $\pi_5$ of ${\rm GSp}_4(\A_{F^+_4})$, as desired. 
Note that by Theorem 7.4.1 of \cite{GeeT} 
(see also Section 1.3 of \cite{Schmidt} for $F=\Q$), we can freely switch the types of $\pi_5$ among 
L-packets at infinite places.

\end{proof}

As observed in Proposition 6.1.4, p.1569 of \cite{BGG1}, we can study the case when $\br$ is 
an induced representation. We have the following possible cases:
\begin{enumerate}
\item $\br$ is induced from a character $\overline{\theta}:G_M\lra \bF^\times_p$ for some 
finite extension $M/F^+$ of degree four;
\item $\br$ is induced from an irreducible representation 
$\overline{\tau}:G_M\lra {\rm GL}_2(\bF)$ for some 
finite extension $M/F^+$ of degree two.
\end{enumerate}
In the first case, the odd-ness of $\br$ implies that $M$ is a CM field so that 
the maximal totally real subfield $M^+$ of $M$ includes $F^+$, while in the latter case, a quadratic 
extension $M/F^+$ can be either CM or totally real. 

For a quadratic extension $L/K$ of number fields, a character $\chi$ of $G_L$, and a lift $\tau\in G_K$ of 
the generator of ${\rm Gal}(L/K)$, we define the character $\chi\otimes \chi^\tau$ of $G_K$ by 
$$(\chi\otimes \chi^\tau)(g) 
=\left\{\begin{array}{ll}
 \chi(g)\chi(\tau^{-1}g\tau) & (g\in G_L) \\
 \chi(\tau^2) & (g=\tau)
 \end{array}\right.
 $$
 which is independent of the choice of $\tau$. 

\begin{prop}\label{ind1}  Let $F^+$ be a totally real extension of $\Q$ and $p$ be an odd.    
Let $\br:G_{F^+}\lra {\rm GSp}_4(\bF_p)$ is an irreducible automorphic mod $p$ Galois representation. 
Assume that $\br$ is  induced from a character $\overline{\theta}:G_M\lra \bF^\times_p$ for some 
finite extension $M/F^+$ of degree four (as observed, $M$ has to be a CM field). 
Then, there is a finite solvable extension $L^+/F^+$ of totally real fields such that 
\begin{enumerate}
\item $L^+$ is linearly disjoint from $M\overline{F^+}^{{\rm Ker}(\br)}(\zeta_p)$ over $F^+$; 
\item there is an $\iota_p$-ordinary cuspidal representation $\pi$ of ${\rm GSp}_4(\A_{L^+})$ of weight zero and level prime to $p$ 
such that 
\begin{enumerate}
\item $\br_{\pi,\iota_p}\simeq \br|_{G_{L^+}}$;
\item $\pi$ is unramified at all finite places. 
\end{enumerate}
\end{enumerate}
\end{prop} 
\begin{proof} Let $M^+$ be the maximal totally real subfield of $M$. Put $\rr:={\rm Ind}^{G_{M^+}}_{G_M}\overline{\theta}$ 
so that $\br={\rm Ind}^{G_{M^+}}_{G_M}\rr$. Notice that $\rr$ is irreducible, totally odd. 
Let $\psi:G_{F^+}\lra \bF^\times_p$ be the similitude character of $\br$ and 
let $\chi_{\rr}:=\det(\rr):G_{M^+}\lra \bF^\times_p$. 
It is easy to see that 
\begin{equation}\label{cent-char}(\chi_{\rr}\otimes \chi^\tau_{\rr})\omega_{M^+/F^+}=\psi^2
\end{equation} where 
$\tau$ is any lift of the generator of ${\rm Gal}(M^+/F^+)$ and $\omega_{M^+/F^+}:G_{F^+}\lra \bF^\times_p$ 
is the quadratic character for $M^+/F^+$. 
Substituting $\tau$ into (\ref{cent-char}), we see that $\chi_{\rr}$ has to extend to a character of 
$G_{F^+}$.

Applying the proof of Proposition 6.1.4 of \cite{BGG1} with the different choice of the 
Hodge-Tate weights in using Lemma 4.1.6 of \cite{CHT}, we can choose a solvable extension $L^+/F^+$ of totally real fields such that 
\begin{enumerate}
\item  $L^+$ is linearly disjoint from $M\overline{F^+}^{{\rm Ker}(\br)}(\zeta_p)$ over $F^+$; 
\item  any place of $L^+$ lying over $p$ is split completely in $L^+M$;
\item 
there is an $\iota_p$-ordinary cuspidal representation $\pi_1$ of ${\rm GL}_2(\A_{M^+_1}),\ M^+_1:=L^+M^+$  and level prime to $p$ 
such that 
\begin{enumerate}
\item $\br_{\pi_1,\iota_p}\simeq \rr|_{G_{M^+_1}}$;
\item $\pi_1$ is unramified at all finite places;
\item for any place $v$ of $L^+$ lying over $p$ with the decomposition $v=w^{(1)}w^{(2)}$ in $L^+M$, 
the Hodge-Tate weights of $\rho_{\pi_1,\iota_p}|_{G_{L^+M,w_1}}$ (resp. 
$\rho_{\pi_1,\iota_p}|_{G_{L^+M,w_2}}$) are $\{0,3\}$ (resp. $\{1,2\}$)   
\item  the central character $\omega_{\pi_1}$ of $\pi_1$ corresponds to the Teichmu\"uller lift of 
$(\chi_{\rr}|_{G_{M^+_1}})\ve^{-1}$
via the class field theory.  
\end{enumerate} 
\end{enumerate}
Note that $\omega_{\pi_1}$ descends to a quasi-character of $\A^\times_{L^+}$ since so is $\chi_{\rr}$ as seen before. 
(as seen soon later, this is a sufficient condition to descend a  
cuspidal representation from $GL_4$ to $GSp_4$). Furthermore, by the choice of $L^+$, $\br|_{G_{L^+}}={\rm Ind}^{G_{L^+}}_{G_{M^+_1}}
\br_{\pi_1,\iota_p}$ is irreducible and 
it implies that $\pi_1$ is not any base change from ${\rm GL}_2(\A_{L^+})$.   
By the last three lines in the statement of Theorem 4.26 of \cite{AS1}, we see that there exists a cuspidal automorphic representation of ${\rm GL}_4(\A_{L^+})$ 
which descend to a cuspidal a globally generic cuspidal automorphic representation of ${\rm GSp}_4(\A_{L^+})$. 
By using base change argument to $\Pi$, 
we can make the resulting cuspidal representation $\Pi$ of ${\rm GSp}_4(\A_{L^+})$ unramified at everywhere. 

Finally we check $\Pi$ is of weight zero. It suffice to compute the Hodge-Tate weights. 
In fact, for $v_i=w^{(1)}_iw^{(2)}_i\ (1\le i \le d)$, we see that 
$$\rho_{\Pi,\iota_p}|_{G_{L^+,v_i}}\simeq \rho_{\pi_1,\iota_p}|_{G_{L^+_1,w^{(1)}_i}}\oplus 
 \rho_{\pi_1,\iota_p}|_{G_{L^+_1,w^{(2)}_i}}.$$
It yields that the Hodge-Tate weights are given by $\{0,1,2,3\}$. 
By Theorem 7.4.1 of \cite{GeeT} , we can freely switch the types of $\Pi$ among 
L-packets at infinite places  by Arthur's classification due to \cite{GeeT}.   
\end{proof}

Next we consider the case when $\br$ is induced from 
a 2-dimensional representation $\rr:G_{M}\lra {\rm GL}_2(\bF_p)$ for some quadratic extension $M/F^+$ of 
totally real fields. This will be proved by an essentially same way of the proof of the previous proposition. 
However, due to the lack of the automorphy of $\rr$ (except for the case when $F^+=\Q$), we need to rely on the current status 
of the potential automorphy for $\rr$ by Taylor. Hence we lose ``the solvability" for the field in which we do the 
base change. 
\begin{prop}\label{ind2}  Let $F^+$ be a totally real extension of $\Q$ and $p$ be an odd.    
Let $\br:G_{F^+}\lra {\rm GSp}_4(\bF_p)$ is an irreducible automorphic mod $p$ Galois representation. 
Assume that $\br$ is  induced from a 2-dimensional representation $\rr:G_M\lra {\rm GL}_2(\bF_p)$ for some 
quadratic extension $M/F^+$ of totally real fields. 
Then, there is a finite extension $L^+/F^+$ of totally real fields such that 
\begin{enumerate}
\item $L^+$ is linearly disjoint from $M\overline{F^+}^{{\rm Ker}(\br)}(\zeta_p)$ over $F^+$; 
\item there is an $\iota_p$-ordinary cuspidal representation $\pi$ of ${\rm GSp}_4(\A_{L^+})$ of weight zero and level prime to $p$ 
such that 
\begin{enumerate}
\item $\br_{\pi,\iota_p}\simeq \rr|_{G_{L^+}}$;
\item $\pi$ is unramified at all finite places. 
\end{enumerate}
\end{enumerate}
\end{prop} 
\begin{proof}It is easy to see that $\rr$ is an irreducible, totally odd mod $p$ Galois representation. 
By Corollary 1.7, p.138 of \cite{TR}, there exists  a finite extension $L^+/F^+$ of totally real fields such that 
\begin{enumerate}
\item $L^+$ is linearly disjoint from $M\overline{F^+}^{{\rm Ker}(\br)}(\zeta_p)$ over $F^+$; 
\item there is an $\iota_p$-ordinary cuspidal representation $\pi_1$ of ${\rm GL}_2(\A_{L^+M})$ 
such that $\br_{\pi_1,\iota_p}\simeq \br|_{G_{L^+M}}$. 
\end{enumerate}
If $\br|_{G_{L^+M(\zeta_p)}}$ is not absolutely irreducible, then $\br|_{G_{L^+M}}$ is dihedral, hence 
this case is reduced to Proposition \ref{ind1}. 
Therefore, we may assume that $\br|_{G_{L^+M(\zeta_p)}}$ is absolutely irreducible. 
By Proposition 2.1.1 of \cite{BGG2}, we have three possible cases;
\begin{enumerate}
\item $p=3$, the projective image of $\br(G_{L^+M(\zeta_p)})$ is conjugate to ${\rm PSL}_2(\F_3)$; 
\item $p=5$, the projective image of $\br(G_{L^+M(\zeta_p)})$ is conjugate to ${\rm PSL}_2(\F_5)$;  
\item $\br(G_{L^+M(\zeta_p)})$ is adequate.
\end{enumerate}
For the third case, we apply Theorem A of \cite{BGG2} and the proof of Proposition 
6.1.3 of \cite{BGG1}, in particular, at line 16, p.1568 but 
we apply Lemma 6.1.2 to obtain a cuspidal representation of a suitable weight to make 
the automorphic induction to ${\rm GL}_4(\A_{L^+})$ weight zero. 
For the second case, we apply Section 8.2 of \cite{KT} whose contents are 
essentially same to things done in the course of the proof of Proposition 
6.1.3 of \cite{BGG1}. Similarly, we apply Lemma 6.1.2 as above in the same purpose for the Hodge Tate weights.  
Thus, there is a finite extension $L^+_1/L^+$ of totally real fields such that 
\begin{enumerate}
\item $L^+_1$ is linearly disjoint from $M\overline{F^+}^{{\rm Ker}(\br)}(\zeta_p)$ over $F^+$ so that 
$[M^+_1:L^+_1]=[M:F^+]=2$ with $M_1=L^+_1M$; 
\item any place $v$ of $L^+_1$ lying over $p$ is split completely in $M^+_1=L^+_1M$;
\item there is an $\iota_p$-ordinary cuspidal representation $\pi_2$ of ${\rm GL}_2(\A_{M^+_1}),\ 
M^+_1=L^+_1M$ of weight zero and level prime to $p$ 
such that 
\begin{enumerate}
\item $\br_{\pi_2,\iota_p}\simeq \br|_{G_{M^+_1}}$;
\item for any place $v$ of $L^+_1$ lying over $p$ with the decomposition $v=w^{(1)}w^{(2)}$ in $M^+_1=L^+_1M$, 
the Hodge-Tate weights of $\rho_{\pi_2,\iota_p}|_{G_{L^+M,w^{(1)}}}$ (resp. 
$\rho_{\pi_2,\iota_p}|_{G_{L^+M,w^{(2)}}}$) are $\{0,3\}$ (resp. $\{1,2\}$) 
\item $\pi_2$ is unramified at all finite places. 
\end{enumerate}
\end{enumerate}
We remark that the third property for $\pi_2$ follows the argument in the course of the proof of Lemma 6.1.1 of \cite{BGG1} 
(see the treatment of the finite set $R$ there in p.1563) using $p$-adic algebraic forms via Jacquet-Langlands correspondence 
with a suitable base change. Note that $\pi_2$ corresponds to a Hilbert modular cusp forms of weight 
$(k_{w^{(1)}},k_{w^{(2)}})_{v=w^{(1)}w^{(2)}|p}=(2,4)_{{\rm Hom}(L^+_1,\bQ)}=(2,4,2,4,\ldots,2,4)$ 
(see Section 2.1 of \cite{Dimi}). 

As seen in the previous proposition, by using ramified twists if necessary, we may assume that 
the central character of $\pi_2$ descends to a quasi-character of $\A_{L^+_1}$. Notice that 
$[M^+_1:L^+_1]=[M:F^+]=2$. The remaining part we need to check is completely same as in the latter part of 
the previous proposition. Therefore, we may omit the details. 

In the case when $p=3$ and the projective image of $\br(G_{L^+M(\zeta_p)})$ is conjugate to ${\rm PSL}_2(\F_3)$, 
applying the proof of Theorem 3.1.1 of \cite{BGG2}, we have a finite extension $L^+_1/L_1$ and an ordinary 
cuspidal representation $\pi_2$ of ${\rm GL}_2(\A_{M^+_1})$ of weight zero as above. 
The representation $\pi_2$ corresponds to an ordinary Hilbert modular cusp form $g$ of parallel weight 2.  
Since $2=p-1$, we can apply Hida theory \cite{Hida} to $g$ to obtain a Hilbert modular cusp forms of weight 
$(k_{w^{(1)}},k_{w^{(2)}})_{v=w^{(1)}w^{(2)}|p}=(2,2+(p-1))_{{\rm Hom}(L^+_1,\bQ)}=(2,4,2,4,\ldots,2,4)$ with $M^+_1=L^+_1M$. The remaining part is checked similarly and the details are 
omitted.
\end{proof}

\section{proof of main theorems}\label{pmt}
We are now ready to prove first two main theorems. In this section, the symbol $F$ stands for a totally real field. 
\begin{proof}(A proof of Theorem \ref{main-thm1}.) 
By Proposition \ref{tech-main}, there exists a finite solvable extension $L/F$ of totally real fields 
and a regular cuspidal automorphic representation $\pi$ of ${\rm GSp}_4(\A_L)$ such that 
\begin{enumerate}
\item $\br\simeq \br_{\pi,\iota_p}$;
\item for each place $v$ lying over $p$, $\rho_{\pi,\iota_p}|_{G_{L,v}}$ is $\iota_p$-ordinary and crystalline. 
\end{enumerate}
By Lemma 1.4.3-(1) of \cite{BGGT}, $\rho_{\pi,\iota_p}|_{G_{L,v}}$ is crystalline, potentially diagonalizable. 
Hence $\rho_{\pi,\iota_p}$ is a lift as desired. 
\end{proof}

\begin{proof} (A proof of Theorem \ref{main-thm2}.)  
Let $\mathcal{O}$ be the integer ring of a sufficiently large finite extension of $\Q_p$. 
 Let $S$ be be a finite set of finite places, including all places at which $\br$ is ramified, and all places 
lying over $p$. Choose a totally imaginary CM field $L$ with maximal totally real field $F$, with 
the property that all primes in $S$ is split in $L$. 
We may assume that $L$ is linearly disjoint from $\overline{F}^{{\rm Ker}(\br)}(\zeta_p)$ over $F$. 
Let $\widetilde{S}$ be a set of places of $F$ 
consisting of one place dividing each place in $S$. By Section 7.4, for any object $R$ of ${\rm CLN}_{\mathcal{O}}$ 
and a lift $\rho_R:G_F\lra {\rm GSp}_4(R)$ of $\br$ to $R$, there exists 
a continuous homomorphism  $r_R:G_F\lra \mathcal{G}_4(R)$ corresponding to $\rho_R$. 
The similitude character of $\rr$ is the same as one of $\br$. Let us fix its lift 
$\psi:G_F\lra \mathcal{O}^\times$ to $\mathcal{O}$ which is crystalline at all places lying over $p$. 
For each place $v\in S$ not dividing $p$, we fix an inertia type $\tau_v$ such that $\br|_{G_F,v}$ has a symmetric 
lifting of type $\tau_v$ and we also fix an $p$-torsion free quotient $R_v$ of $R^{{\rm symp},\tau_v,\psi}_{\br|_{G_{F,v}}}$ 
(see Section 7.2, p.284 of \cite{gg} for this deformation ring). 
Let ${\rm v}_\lambda$ be the $p$-adic Hodge type of $\rho_v$. 
For each place $v\in S$ dividing $p$, we also fix a quotient $R_v$ of 
$R^{{\rm symp},{\rm v}_\lambda,{\rm cr},\psi}_{\br|_{G_{F,v}}}$ which corresponds to an irreducible component 
of ${\rm Spec}\hspace{0.5mm}R^{{\rm symp},{\rm v}_\lambda,{\rm cr},\psi}_{\br|_{G_{F,v}}}$ such that 
$\rho_v$ lies on.  For any $v\in \widetilde{S}$, we regard $R_v$ with a deformation for $\rr|_{G_F,v}$ 
(cf. Lemma 2.2.1 of \cite{CHT}). We consider the deformation problem 
$$\mathcal{S}=(L/F,S,\widetilde{S},\mathcal{O},\rr,\psi,\{R_v\}_{v\in S})$$
with corresponding universal deformation 
$r_{\mathcal{S}}:G_{F}\lra \mathcal{G}_4(R^{{\rm univ}}_\mathcal{S})$ which factors through 
$G_{F,S}={\rm Gal}(F(S)/F)$. Here $F(S)$ stands for the maximal subfield of $\overline{F}$ unramified outside $S$.  
By Proposition 1.5 of \cite{BGGT} (note that in the notation there, $\breve{\rr}=\rr|_{G_L}=\br_{G_L}$ 
and it is absolutely irreducible by the choice of $L$ and Proposition \ref{odd-irred}), $R^{{\rm univ}}_\mathcal{S}$ has Krull dimension at least one. 
By the choice of $L$, $\br(G_{M(\zeta_p)})$ is adequate. Hence we can apply Theorem 2.3.2 of \cite{BGGT} and it follows 
from this and ${\rm dim}_{{\rm Krull}}R^{{\rm univ}}_\mathcal{S}\ge 1$ that $R^{{\rm univ}}_\mathcal{S}$ is a non-zero finite $\mathcal{O}$-module. 
Combining it with Lemma 7.4.1 of \cite{gg}, there exists a lift $\rho:G_F\lra {\rm GSp}_4(\bQ_p)$ of $\br$ 
such that for each finite place $v$ of $F$ lying over $p$,  $\rho|_{G_{F,v}}$ is crystalline with the 
Hodge-Tate weights equal to one of $\rho_v$. Furthermore, the choice of $R_v$ shows that $\rho|_{G_{F,v}}$ 
connects to $\rho_v$. By Theorem B of \cite{BGGT}, $\rho_{G_L}$ is automorphic. Hence, there exists a 
regular algebraic, cuspidal, polarized automorphic representation $\Pi$ of ${\rm GL}_4(\A_L)$. 
Clearly, $\rho_{G_L}$ descends to $\rho$ and so is $\Pi$ to a 
regular algebraic, cuspidal, polarized automorphic representation $\pi$ of ${\rm GL}_4(\A_F)$ so that 
$\rho\simeq \rho_{\pi,\iota_p}$. It follows from this that $\pi$ descends to a cuspidal 
automorphic representation  ${\rm GSp}_4(\A_F)$ as desired. 
If $\br|_{G_{F,v}}$ is ordinary for all $v$ above $p$,  then the usual Khare-Wintenberger's method 
to compare the deformation rings for $\br$ and $\br_{G_L}$ (we may choose $R_v$ to be an ordinary irreducible component) guarantees an ordinary automorphic representation $\pi$ of 
${\rm GL}_4(\A_F)$ which descends to ${\rm GSp}_4(\A_F)$ 
as desired. 
\end{proof}

\section{Potentially diagonalizable, crystalline lifts with prescribed types}\label{prescribed}
In this section, we study how we can lift given mod $p$ local Galois representations. 
We refer \cite{Rubin},\cite{BK},\cite{Ne} as basic references. We note that $p=2$ is allowed in this section.  

\subsection{Galois cohomology}\label{GC}
Let $K$ and $E$ be finite extensions of $\Q_p$. Assume that $E$ is sufficient large so that it includes 
all embeddings from $K$ to $\bQ_p$. 
Let $\O=\O_E$ be the ring of integers of $E$ and $\F=\F_E$ be its residue field. Let $\pi=\pi_E$ be 
a uniformizer of $E$. Let $I=I_K$ be the inertia subgroup of $G_K={\rm Gal}(\overline{K}/K)$. 
We fix a lift ${\rm Fr}$ to $G_K$ of the arithmetic Frobenius in $G_{\F}$.  

For any finite, free $G_K$-module $T$ over $\O$, put 
$$V=T\otimes E,\ W=V/T,\ W_n=\frac{1}{\pi^n}T/T\simeq T/\pi^n T,\ (n\ge 1).$$
For each of the above $G_K$-modules, denote it by $M$, we write its continuous Galois cohomology by 
$$H^i(K,M):=H^i(G_K,M)\ (i\ge 0).$$  
For any finite $\O$-module $N$,   we denote by $N_{{\rm tor}}$ the torsion part of $N$.  

Let us recall some facts. 
\begin{lem}\label{useful1}Keep the notation being as above. It holds that 
\begin{enumerate}
\item there exists an exact sequence 
$$0\lra T^{G_K}\lra V^{G_K}\lra W^{G_K}\lra H^1(K,T)_{{\rm tor}}\lra 0;$$
\item $W^{G_K}_1=0$ if and only if $W^{G_K}=0$. In this case,  
$H^1(K,T)$ is $\O$-torsion free;
\item $H^i(K,T)\otimes_{\O}E\simeq H^i(K,V)$. 
\end{enumerate}
\end{lem}
\begin{proof}The first claim is Lemma 1.2.2-(ii). 
The second claim is standard (cf. p.502, line -5 from the bottom of the proof of Theorem 7.4 
of \cite{BKlosin} and also we apply the first claim). See Proposition 2.7.11,p.144 of \cite{NSW} for the third claim. 
\end{proof}
For each $G_K$-module $M$ we define 
the unramified cohomology $H^i_{{\rm ur}}(K,M)$ by 
$$H^i_{{\rm ur}}(K,M):={\rm Ker}(H^i(G_K,M)\stackrel{{\rm res}}{\lra} H^i(I_K,M))$$
where res stands for the restriction map induced by the natural inclusion $I_K\subset G_K$. 
The following Lemma follows from Lemma 1.3.2-(i), p.13 of \cite{Rubin}:
\begin{lem}\label{useful2}Suppose $M$ is finite $G_K$-module over $E$ or $\O$. It holds that 
$$H^i_{{\rm ur}}(K,M)\simeq M^{I_K}/({\rm Fr}-1)M^{I_K}.$$
\end{lem}
Let $T$ be a finite free $G_K$-module over $\O$. Put $V=T\otimes_{\O}E$ and $\T=T\otimes_{\O}\F$. 
The natural inclusion $i:T\hookrightarrow V$ and the natural projection map $p:T\lra \T$ induce 
$H^1(K,T)\stackrel{i_\ast}{\lra }H^1(K,V)$ and $H^1(K,T)\stackrel{p_\ast}{\lra }H^1(K,\T)$ respectively. 
The Bloch-Kato Selmer group of $V$ (see Section 3 of \cite{BK}) is defined by 
$$H^1_f(K,V):={\rm Ker}(H^1(G_K,V)\lra H^1(G_K,V\otimes_{\Q_p}B_{{\rm cris}}))$$
where $B_{{\rm cris}}$ is one of  Fontaine's period rings (see (1.10) of \cite{BK}). If $V$ is crystalline, 
any element of $H^1_f(K,V)$ can be regarded as an extension 
$0\lra V\lra \ast\lra E\lra 0$ such that $\ast$ is a crystalline representation of $G_K$ defined over $E$, and vice versa.  
Then we define $$H^1_f(K,T):=i^{-1}_\ast H^1_f(K,V),\ H^1_f(K,\T):=p_\ast(H^1_f(K,T)).$$ 
Note that it looks  $H^1_f(K,\T)$ depends on $\T$ but remember that it is defined by using 
$H^1_f(K,V)$. For example, for the mod $p$ cyclotomic character $\ve$ of $G_K$ and the $p$-adic 
cyclotomic character $\e$ of $G_K$, even if $\overline{(\e^{1+k(p-1)})}=\ve$ for any integer $k$, 
$H^1_f(K,\ve)$ does not necessarily coincide with $H^1_f(K,\overline{(\e^{1+k(p-1)})})$.  
According to Definition 1.1.3 of \cite{Rubin}, the dual representations are defined by 
$$T^\ast:={\rm Hom}_{\O}(T,\O(1)),\ V^\ast:=T^\ast\otimes_{\O}E,\ W^\ast:=T^\ast\otimes_{\O}E/{\O},\ 
W^\ast_n:=T^\ast\otimes_{\O}\frac{1}{\pi^n}\O/{\O}.$$
We sometimes use $M^\vee$ to be the usual dual of $M$ in the coefficient of $M$ so that 
$T^\ast=T^\vee(1)$. 
\begin{lem}\label{useful3}Keep the notation as above. Then, 
$(W^\ast_1)^{G_K}=0$ if and only if $(W_1(-1))^{G_K}=0$.
\end{lem}
\begin{proof}It follows from $W^\ast_1={\rm Hom}_{\F}(W_1,\F(1))$. 
\end{proof}

\begin{lem}\label{useful4}Keep the notation as above. If $(W^\ast_1)^{G_K}=0$, then 
\begin{enumerate}
\item $H^1_f(K,\T)\simeq H^1_f(K,T)\otimes_{\O}\F$;
\item if further $W^{G_K}_1=0$, then ${\rm dim}_{E}H^1_f(K,V)={\rm dim}_{\O}H^1_f(K,T)={\rm dim}_{\F}H^1_f(K,\T)$.
\end{enumerate} 
\end{lem}
\begin{proof}By the long exact sequence $0\lra T\stackrel{\times \pi}\lra T\stackrel{{\rm mod}\ \pi}{\lra} \T\lra 0$, we have 
$$0\lra H^1(K,T)\otimes_{\O}\F\lra H^1(K,\T)\lra H^2(K,T).$$
By Theorem 1.4.1 of \cite{Rubin}, we have $H^2(K,T)\simeq {\rm Hom}_{\O}((W^\ast)^{G_K},E/\O)$. 
By Lemma \ref{useful1}-(ii) and assumption, we have the vanishing $H^2(K,T)=0$. Hence the first claim follows from the 
definition of $H^1_f(K,\ast)$. 
The second claim is deduced from Lemma \ref{useful1}-(2),(3) with the first claim. 
\end{proof}

Let $\T$ be a finite dimensional $G_K$-module over $\F$. 
\begin{lem}\label{pi2}
Let $T$ be a lift of $\T$ to $\O$ such that $T^\ast=T^\vee(1)$ does not have trivial sub-representations modulo 
$\pi^2$. Hence, $H^0(K,T^\ast/\pi^2)=H^0(K,W^\ast_2)=0$. Then the reduction modulo $\pi^2$ induces a surjection 
$H^1(K,T)\lra H^1(K,T/\pi^2)$. In particular, the reduction map modulo $\pi$, that is $
H^1(K,T)\lra H^1(K,\T)$ factors through this map.  
\end{lem}
\begin{proof}Consider the exact sequence $0\lra T\stackrel{\times \pi^2}{\lra}T\lra T/\pi^2\lra 0$. 
We may check the vanishing of $H^2(K,T)[\pi^2]$. It is easy to see that 
$$H^2(K,T)/\pi^2 H^2(K,T)\simeq H^2(K,T/\pi^2)=0$$
by assumption and $H^3(K,T)=0$ (note that $G_K$ is of cohomological dimension 2). 
By Nakayama's lemma, $H^2(K,T)=0$.  
\end{proof}

Let $\psi:G_K\lra \O^\times$ be a continuous character with the trivial reduction. 
We write $\psi=1+\pi u_\psi$ with a map $u:G_{\Q_p}\lra \O$ which becomes a character after the reduction. 
Let $\T$ and $T$ be as right before the previous lemma. 
The exact sequence 
$$0\lra \T\simeq \pi (\psi\otimes T)/\pi^2 \lra (\psi\otimes T)/\pi^2\lra \T\lra 0$$
yields the connection homomorphism $\delta_\psi:H^1(K,\T)\lra H^2(K,\T)$. 
\begin{lem}\label{conn} 
Keep the notation being as above. For any $\alpha\in H^1(K,\T)$, it holds that 
$\delta_\psi(\alpha)=\overline{u}_\psi\cup \alpha$ where 
$\overline{u}_\psi$ is the reduction of $u_\psi$ modulo $\pi$. 
\end{lem}
\begin{proof}This follows from Lemma 2.2.7 of \cite{Muller}. 
\end{proof}

\subsection{Torsion classes}\label{TC}
Let $\T$ be a finite dimensional $G_K$-module over $\F$. 
Let $T$ be a lift of $\T$ to $\O$ such that $H^0(K,\T)=0$. Put $W=T\otimes E/\O$. 
\begin{lem} It holds that $W^{G_{K}}\simeq H^1(K,T)_{{\rm tor}}$. 
In particular, if $\T$ has a trivial sub-representation, $H^1(K,T)_{{\rm tor}}$ is non-trivial. 
\end{lem} 
\begin{proof}It follows from Proposition \ref{useful1}-(1). The latter claim follows from 
$\T\simeq W_1\subset W$. 
\end{proof}
Let $\psi:G_K\lra \O^\times$ be a continuous character with the trivial reduction and put 
$c(\psi)=\max\{{\rm ord}_\pi(\psi(g)-1)\ |\ g\in G_K\}$. The reader should not confuse it with the conductor of $\psi$. 
We define the following map  
$$f_\psi:G_K\lra \O,\ \sigma \mapsto \frac{\psi(\sigma)-1}{\pi^{c(\psi)}}$$
which comes from the connection map for the exact sequence 
$$0\lra \O(\psi)\lra E(\psi)\lra E/\O(\psi)\lra 0.$$  
\begin{prop}\label{torsion} Keep the notation as above. It holds that $H^1(K,\O(\psi))_{{\rm tor}}$ 
is generated by $f_\psi$ and $H^1(K,\O(\psi))_{{\rm tor}}\simeq \pi^{-c(\psi)}\O/\O\simeq 
\O/\pi^{c(\psi)}$ which sends $f_\psi$ to $\pi^{-c(\psi)}$ and to $1$. 
\end{prop}
\begin{proof}
Put $T=\O(\psi)$. 
By Proposition \ref{useful1}-(1), $W^{G_K}\stackrel{\sim}{\lra} H^1(K,T)_{{\rm tor}}$. 
The isomorphism is given by the connection map $f_\psi$. We observe that 
$x\in W^{G_K}$ if and only if $(\psi(\sigma)-1)x\in \O$ for any $\sigma\in G_K$. 
The claim follows from this. 
\end{proof}

We say a class of $H^1(K,\F)={\rm Hom}(G_K,\F)$ 
an unramified class if it belongs to $H^1_{{\rm ur}}(K,\F)={\rm Hom}(G_K/I_K,\F)$, a ramified class otherwise. 
\begin{prop}\label{torsion-d1} Let $\psi:G_K\lra \O^\times$ be a continuous character with $c(\psi)=1$. 
\begin{enumerate}
\item 
The reduction $H^1(K,\O(\psi))\lra H^1(K,\F)$ is surjective if $\ve|_{G_K}$ is non-trivial;
\item if $\ve|_{G_K}$ is non-trivial, for each class $\alpha\in H^1(K,\F)$, there is $\psi$ as above 
such that $\alpha$ is liftable to $H^1(K,\O(\psi))$ via the reduction map;
\item the generator $f_\psi$ of $H^1(K,\O(\psi))_{{\rm tor}}$ goes to $\overline{u}_\psi$. 
In particular, if $\psi$ is unramified, $\overline{u}_\psi\in H^1_{{\rm ur}}(K,\F)$. 
\end{enumerate}
\end{prop}
\begin{proof}
As in the proof of Lemma \ref{pi2}, the vanishing $H^2(K,\F)\simeq H^0(K,\ve)=0$ yields 
$H^2(K,\O(\psi))=0$. 
The first claim follows from $H^2(K,\O(\psi))[\pi]=0$. For the second claim, let $\psi$ be a character with $c(\psi)=1$. 
Notice that $H^2(K,\O(\psi)/\pi^2)=0$. By Proposition \ref{pi2}, the reduction map 
$H^1(K,\O(\psi))\lra H^1(K,\F)$ factors through the surjection $H^1(K,\O(\psi))\lra H^1(K,\O(\psi)/\pi^2)$. 
Therefore, we may lift a class to $H^1(K,\O(\psi)/\pi^2)$ for a suitable $\psi$.  
Applying Lemma \ref{conn}, we have the exact sequence  
$$H^1(K,\O(\psi)/\pi^2)\lra H^1(K,\F)\stackrel{\delta_\psi}{\lra}  H^2(K,\F)$$
such that $\delta_\psi(\alpha)=\overline{u}_\psi\cup \alpha$ for each $\alpha\in H^1(K,\F)$. 
Since $\ve|_{G_K}$ is trivial, the local Tate duality induces a perfect pairing on $H^1(K,\F)\times H^1(K,\F)$.  
Since ${\rm dim}_\F H^1(K,\F)=1+1+[K:\Q_p]\ge 2$, we may take a $\psi$ such that $\overline{u}_\psi\cup \alpha=0$. 

The third claim is obvious by definition. 
\end{proof}

\subsection{Exhaustion of $H^1(\Q_p,\T)$ by crystalline objects}\label{Ex} 
In this subsection, we suppose $K=\Q_p$ for simplicity. 

Let us fix some notation. For each finite $\O$-module $M$, we denote by 
$M_{{\rm tf}}$ a torsion free part of $M$ such that $M=M_{{\rm tf}}\oplus M_{{\rm tor}}$. 
A choice of $M_{{\rm tf}}$ is not obviously canonical, but $\O$-rank is stable. 
For each Hodge-Tate representation $\rho$ of $G_{\Q_p}$, we denote by ${\rm HT}(\rho)$ 
the multi set of Hodge-Tate weights of $\rho$.  
For multi-subsets $A,B$ of $\Z$, we write $A\ge B$ if ${\rm min}(A)\ge {\rm max}(B)$. 

Let $\T$ be a finite dimensional $G_{\Q_p}$-module over a finite field $\F$. 
Let $W(\F)$ be the ring of Witt vectors of $\F$. 
Let $T$ be any lift $\T$ which is free module over the integer ring $\O=\O_E$ for some finite extension $E/\Q_p$ 
including $W(\F)$.  
We write 
$$h^i_{\ast}(\T)={\rm dim}_{\F}H^i_\ast(\Q_p,\T),\ 
h^i_{\ast}(T)={\rm dim}_{\O}H^i_\ast(\Q_p,T)_{{\rm tf}}, \
h^i_{\ast}(T\otimes_{\O}E )={\rm dim}_{E}H^i_\ast(\Q_p,T\otimes_{\O}E )$$
where $\ast$ is empty or $f$. 

Henceforth, whenever we consider lifts of $\T$, the finite field $\F$ defining $\T$,  
the base ring $\O$ of each lift and its fraction field $E$ are suitably chosen and it should be naturally understood 
from the context without any confusion. Therefore, we do not mention about what $\F$, $\O$ and $E$ are in each case. 

What we will carry out is that firstly, we take a crystalline lift $T$ of $\T$ and secondary, study if 
the reduction map $H^1(K,T)\lra H^1(K,\T)$ is surjective. 
As in Proposition \ref{torsion}, \ref{torsion-d1}, if $\T$ has a trivial sub-representation, then 
$H^1(K,T)$ always has torsion elements. Hence a careful analysis is necessary when we consider 
the liftability of each class of $H^1(K,\T)$.  

\subsubsection{One dimensional case}\label{d1} 
Let $\e=\chi_p:G_{\Q_p}\lra \Z^\times_p$ (resp. $\ve:G_{\Q_p}\lra \F^\times_p$) be the $p$-adic (resp. mod $p$) cyclotomic character of $G_{\Q_p}$ and $\psi:G_{\Q_p}\lra \O^\times$ 
be an unramified character. Let $\textbf{1}:G_{\Q_p}\lra \F^\times_p$ be the trivial character. 

We denote by $\op:G_{\Q_p}\lra \F^\times$ the reduction of $\psi$ modulo $\pi$ where $\pi$ is a uniformizer of $\O$.  
For a non-negative integer $a$, put 
$$V_{a,\psi}:=E(\e^a \psi),\ T_{a,\psi}:=\O(\e^a \psi),\ \T_{\oa,\op}:=\F(\ve^a \op),\ 
W_{a,\psi}:=E/\O(\e^a \psi),\ W_{a,\psi,n}:=\pi^{-n}\O/\O(\e^a \psi).$$
For each integer $a$ above, put $\oa:=a\ {\rm mod}\ p-1$. 
\begin{dfn}\label{classes} 
The natural inclusion $H^1_{{\rm ur}}(\Q_p,\T_{0,\textbf{1}})\subset H^1(\Q_p,\T_{0,\textbf{1}})$ can be identified 
with $${\rm Hom}_{\F_p}(\Q^\times_p/\Z^\times_p,\F)\subset {\rm Hom}_{\F_p}(\Q^\times_p,\F).$$ 
Recall the local Tate pairing 
$H^1(\Q_p,\T_{1,\textbf{1}})\times H^1(\Q_p,\T_{0,\textbf{1}})\lra \F_p$.  
A class of $H^1(\Q_p,\T_{1,\textbf{1}})$ is called peu ramifi\'ee if it is annihilated by 
${\rm Hom}_{\F_p}(\Q^\times_p/\Z^\times_p,\F_p)\subset {\rm Hom}_{\F_p}(\Q^\times_p,\F_p)$ under the pairing. 
Otherwise we say a class tr\`es ramifi\'ee. By Kummer theory, 
$H^1(\Q_p,\T_{1,\textbf{1}})\simeq \Q^\times_p/(\Q^\times_p)^p$ and the subgroup $\Z^\times_p/(\Z^\times_p)^p$ 
exhausts all peu ramifi\'ee classes.  
If $p>2$, then $\Z^\times_p/(\Z^\times_p)^p\simeq \F_p$ while $\Z^\times_2/(\Z^\times_2)^2\simeq \mu_2\times \F_2$. 
Therefore we have two classes, say the plus peu ramifi\'ee class and the minus peu ramifi\'ee class according to the sign in 
$\mu_2$. The definition here is taken from Definition 2.1.2 of \cite{GHLS} and in this special case, 
it is equivalent to what Serre defined in p.186 of \cite{serre}.  

We say a class of $H^1(\Q_p,\T_{0,\textbf{1}})$ an unramified class if it belongs to 
$H^1_{{\rm ur}}(\Q_p,\T_{0,\textbf{1}})$, a ramified class otherwise. 
The  unramified classes are orthogonal complements of peu ramifi\'ee classes 
with respect to the local Tate duality.  
\end{dfn}
\begin{dfn}\label{pt}
For two characters $\chi_1,\chi_2:G_{\Q_p}\lra \bF^\times_p$, an extension class 
$$0\lra \chi_1\lra \ast \lra \chi_2 \lra 0$$ is called tr\`es ramifi\'ee $($resp. ramified$)$ if 
$\chi_1\chi^{-1}_2=\ve$ $($resp. $\textbf{1})$ and $\ast$ is tr\`es ramifi\'ee 
$($ resp. ramified$)$. Otherwise, it is called peu ramifi\'ee. 
Hence, peu ramifi\'ee classes include all unramified class in this notation.  
\end{dfn}

\begin{prop}\label{prop-d1} 
Let $a$ be a non-negative integer with $a\equiv b\ {\rm mod}\ p-1,\ (0\le b\le p-2)$. 
\begin{enumerate}
\item 
Assume that $a$ is positive if $\op\neq \textbf{1}$.  
If $(\ob,\op)$ is neither $(\overline{1},\textbf{1})$ nor $(\overline{0},\textbf{1})$, 
then $H^1(\Q_p,T_{a,\psi})$ is free over $\O$ and
$${\rm dim}_EH^1_f(\Q_p,V_{a,\psi})={\rm dim}_{\O}H^1_f(\Q_p,T_{a,\psi})
={\rm dim}_{\F}H^1(\Q_p,\T_{\ob,\op})=1.$$
Hence any class in $H^1(\Q_p,\T_{\ob,\op})$ has a crystalline lift, 
\item Assume  $(\ob,\op)=(\overline{1},\textbf{1})$ and $p>2$. 
\begin{itemize}
\item For any integer $k\ge 0$, each peu ramifi\'ee class $\alpha$ of 
$H^1(\Q_p,\T_{\overline{1},\textbf{1}})=H^1(\Q_p,\ve)$ is liftable to a class of 
$H^1_f(\Q_p,T_{1+k(p-1),\psi_\alpha})$ for some unramified lift $\psi_\alpha$ of $\op$ with $c(\psi_\alpha)=1$.  
\item For any integer $k\ge 1$, each tr\'es ramifi\'ee class of $H^1(\Q_p,\ve)$ is liftable to a class of 
$H^1_f(\Q_p,T_{1+k(p-1),\psi_\alpha})$ for some unramified lift $\psi_\alpha$ of $\op$ with 
$c(\psi_\alpha \e^{p-1})=1$. 
Further, in either case, $H^1_f(\Q_p,T_{1+k(p-1),\psi_\alpha})$ is torsion free. 
\end{itemize}
\item Assume  $(\ob,\op)=(\overline{0},\textbf{1})$. 
\begin{itemize}
\item
Any unramified class of $H^1(\Q_p,\T_{\overline{0},\textbf{1}})=H^1(\Q_p,\F)$ is liftable to a class of 
$H^1_f(\Q_p,\O)\simeq \O$.  
\item For any integer $k\ge 0$, each unramified class $\alpha$ of 
$H^1(\Q_p,\F)$ is liftable to a free $\O$ module of rank one in  
$H^1_f(\Q_p,T_{k(p-1),\psi_\alpha})$ for some unramified lift $\psi_\alpha$ of $\op$ with $c(\psi_\alpha)=1$.  
\item For any integer $k\ge 1$, each ramified class $\alpha$ of 
$H^1(\Q_p,\F)$ is liftable to a free $\O$ module of rank one in  
$H^1_f(\Q_p,T_{k(p-1),\psi_\alpha})$ for some unramified lift $\psi_\alpha$ of $\op$ with $c(\psi_\alpha \e^{p-1})=1$. 
\end{itemize}
\end{enumerate} 

\begin{proof}The first claim and the latter part of the second claim follow from Proposition 1.24 of \cite{Ne} and Lemma \ref{useful4}-(2). 
Note that $a$ should be positive when $\op\neq \textbf{1}$ otherwise 
${\rm dim}_EH^1_f(\Q_p,V_{0,\psi})=0$. 

The second claim follows from the proof of Proposition 3.5 of \cite{KW} but
we give a detailed proof.  
Let $\psi$ be a continuous character with $c(\psi)=1$ and put $\psi=1+u_\psi \pi$. It follows from the condition that  
$\overline{u}:G_{\Q_p}\lra \F$ is a non-trivial additive homomorphism.  
By Lemma \ref{pi2} and Lemma \ref{conn}, we have 
the surjection $H^1(\Q_p,\O(\psi \e))\lra H^1(\Q_p,\O(\psi \e)/\pi^2)$ and 
the exact sequence 
$$H^1(\Q_p,\O(\psi \e)/\pi^2)\lra H^1(\Q_p,\ve)\stackrel{\delta_{\psi}}{\lra} H^2(\Q_p,\ve)\simeq \F$$
where $\delta_{\psi}(\alpha)=\overline{u}_\psi\cup \alpha$ for $\alpha\in H^1(\Q_p,\ve)$. 
By Kummer theory, we identify $H^1(\Q_p,\ve)=\Q^\times_p/(\Q^\times_p)^p$. Then 
$\delta_{\psi}(\alpha)=\overline{u}_\psi(\alpha)$.  
Then one can find $\psi_\alpha$ such that $\overline{u}_{\psi_\alpha}(\alpha)=0$. 
Hence $\alpha$ is liftable to $H^1(\Q_p,\O(\psi \e)/\pi^2)$ and then to $H^1(\Q_p,\O(\psi \e))$. 
If $\alpha$ is peu ramifi\'ee (it comes from the unit group), then $\overline{u}_{\psi_\alpha}$ is unramified. 
Let $\psi_\alpha$ be an unramified lift of $\overline{u}_{\psi_\alpha}$. 
For any wildly ramified continuous character $\chi:G_{\Q_p}\lra \O^\times$ with the trivial reduction, we have 
$c(\psi_\alpha \chi)=1$ 
and $\overline{u}_{\psi_\alpha \chi}=\overline{u}_{\psi_\alpha}$. 
In particular, we can choose $\chi$ to be $\e^{k(p-1)}$ for any $k\ge 0$. 
Hence we have a lift to $H^1(\Q_p,\O(\psi_\alpha \e^{1+k(p-1)}))$ for 
any $k\ge 0$.   

If $\alpha$ is tr\`es ramifi\'ee, then $\overline{u}_{\psi_\alpha}$ has to ramified, since otherwise, 
the pairing $\overline{u}_{\psi_\alpha}(\alpha)$ can not be zero. 
Then there is an unramified character $\psi_\alpha$ with $c(\psi_\alpha \e^{p-1})=1$ such that 
$\overline{u}_{\psi_\alpha \e^{p-1}}(\alpha)=0$. Similarly, twisting by any wild character with the 
trivial reduction makes no change. We have a lift to $H^1(\Q_p,\O(\psi_\alpha \e^{1+k(p-1)}))$ for 
any $k\ge 1$.   
It follows from the proof of Proposition 3.5 of \cite{KW} that first we can lift any 
tr\`es ramifi\'ee class to a crystalline extension of Hodge-Tate weight $\{0,p\}$ over $\O$ and then we have 
the desired lift by congruence.  
The freeness of the cohomologies follows from Lemma \ref{useful1}-(2). 

For the third claim, pick a non-trivial class $\alpha \in H^1(\Q_p,\F)$. 
If $\alpha$ is unramified, then the claim is clear since $H^1_f(\Q_p,\O)=H^1_{{\rm ur}}(\Q_p,\O)$. 
If $p>2$, then $\ve\neq \textbf{1}$. Therefore, for any character $\psi$ with the trivial reduction, the reduction map 
$H^1(\Q_p,\O(\psi))\lra H^1(\Q_p,\F)$ is always surjective and it factors through the surjection  
$$H^1(\Q_p,\O(\psi))\lra H^1(\Q_p,\O(\psi)/\pi^2)$$ (cf. the proof of Lemma \ref{pi2}). 
We may lift $\alpha$ to $H^1(\Q_p,\O(\psi)/\pi^2)$ for a suitable $\psi$. 
Let $\beta:=\alpha^\perp\in H^1(\Q_p,\ve)$ be a non-trivial element which is orthogonal to $\alpha$ 
in the local Tate duality. 
If $\alpha$ is unramified, then $\beta$ is peu ramifi\'ee. 
by using the previous argument, there is an unramified character 
$\psi_{\beta}$ with $c(\psi_{\beta})=1$ such that for any $k\ge 0$, $\beta$ is liftable to an element 
$\widetilde{\beta}\in H^1(\Q_p,\O(\psi^{-1}_{\beta}\ve^{1-k(p-1)})/\pi^2)$ is not killed by $\pi$. 
By the local Tate duality again, we can find a desired lift $\widetilde{\alpha}\in H^1(\Q_p,\O(\psi_{\beta}\ve^{k(p-1)})/\pi^2)\cap 
\langle \widetilde{\beta} \rangle^{\perp}$ which is not killed by $\pi$. Hence it is not a torsion class 
by Proposition \ref{torsion}. Therefore we have a free $\O$-module of rank one in 
$H^1_f(\Q_p,\O(\psi_{\beta}\ve^{k(p-1)}))$ which yields a crystalline lift of $\alpha$ as desired.      
Similarly, if $\alpha$ is ramified, then $\beta$ is peu ramifi\'ee and 
there is an unramified character 
$\psi_{\beta}$ with $c(\psi_{\beta}\e^{p-1})=1$ such that for any $k\ge 1$, 
$\alpha$ is liftable to a free $\O$-module of rank one in 
$H^1_f(\Q_p,\O(\psi_{\beta}\ve^{k(p-1)}))$. 

Finally, we consider the case when $p=2$. 
In this case we see $\ve=\textbf{1}$ and $H^2(\Q_p,\F)\simeq H^0(\Q_p,\ve)=H^0(\Q_p,\F)=\F$. 
Hence, we can apply the argument for the second claim.  
\end{proof}

\begin{rmk}When $p>2$, the  space $H^1(\Q_p,\T_{\overline{1},\textbf{1}})$ is $2$-dimensional 
and it is divided into a unique peu ramifi\'ee 
line and 
tr\'es ramifi\'ee lines. Each tr\'es ramifi\'ee line of which $\tau$ belongs is liftable to the free $\O$-module  
$H^1_f(\Q_p,T_{1+k(p-1),\psi_\tau})$ of rank one. If we vary $\psi_\tau$, then 
ramified classes are exhausted but we can not do the same for a single $\psi_\tau$. 
It is the same for $H^1(\Q_p,\T_{\overline{0},\textbf{1}})$. 
\end{rmk}

\begin{dfn}\label{psi-and-r}
For each character $\oc:G_{\Q_p}\lra \bF^\times_p$ and each class $\alpha\in H^1(\Q_p,\oc)$, we 
define the character $\psi_{\oc,\alpha}$ to be trivial if $\alpha$ is peu ramifi\'ee and unramified, 
a non-trivial unramified character with the trivial reduction appear in Proposition \ref{prop-d1}-(2),(3) 
according to $\alpha$ is tr\`es ramifi\'ee or ramified. Put  
$$r_{\oc,\alpha}:=
\left\{\begin{array}{cl}
p-1 & \mbox{$\alpha$ is tr\`es ramifi\'ee or ramified} \\
0   & \mbox{otherwise}
\end{array}
\right..
$$
\end{dfn}

\end{prop}
\subsubsection{Two dimensional, irreducible case}\label{d2-irr}
Let $\br:G_{\Q_p}\lra {\rm GL}_2(\bF_p)$ be a continuous irreducible representation. 
It is well-known (cf. Section 2.4 of \cite{Edix} or Section 2.5 of \cite{Muller}) 
that there exist integers $0\le b<a<p$ such that 
$\br\simeq \op\otimes {\rm Ind}^{G_{\Q_p}}_{G_{\Q_{p^2}}}\omega^{b+ap}_2$ where $\Q_{p^2}$ stands for the 
unramified quadratic extension of $\Q_p$, $\op:G_{\Q_p}\lra\bF^\times_p$ is an unramified character and $\omega_2:G_{\Q_{p^2}}\lra \bF^\times_p$ is a unique extension of 
the fundamental character of the inertia subgroup $I_p$ of level two. 
As a result, $\br$ is absolute irreducible since $\omega^p_2\neq \omega_2$. 
By Th\'eor\`eme 2.5.2 of \cite{Muller}, there exists a 
crystalline lift $\rho:G_{\Q_p}\lra {\rm GL}_2(\O)$ of $\br$ with the Hodge-Tate weight $\{a,b\}$ 
for any $\O=\Q_E$ of a finite extension $E/\Q_p$ including 
$\Q_{p^2}$.  It is also potentially diagonalizable 
by Theorem 3.0.3 of \cite{GLiu} since $0\le a-b\le p-1$. 

Let $T$ be the representation space of $\rho$ and put $V=T\otimes_{\O}E$. 
Let $\psi$ and $\op$ be as in Section \ref{d1}. 
\begin{prop}\label{cla-irr}For any non-negative integer $c$ with $c\equiv d\ {\rm mod}\ p,\ 0\le d \le p-2$, it holds that 
$H^1_f(\Q_p,T(\e^c\psi))$ is torsion free and 
${\rm dim}_{\F}H^1(\Q_p,\T(\ve^{d}\op))=2$. Further,  
${\rm dim}_EH^1_f(\Q_p,V(\e^{c} \psi))=2$ if all Hodge-Tate weights $\{c+a,c+b\}$ of $V(\e^{c} \psi)$ are positive.  
\end{prop} 
\begin{proof}
Notice that $\br\simeq W_1$ is irreducible.  
Since $(W_1)^{G_{\Q_p}}=0$ and $(W^\ast_1)^{G_{\Q_p}}=0$, by Lemma \ref{useful1}-(2),(3) and Lemma 
\ref{useful4}-(1), $H^1_f(\Q_p,T(\e^a\psi))$ is torsion free. The last claim follows from \ref{useful4}-(2) and 
Proposition 1.24-(2) of \cite{Ne} (note that the Hodge-Tate weights of 
$V(\e^{a} \psi)$ are positive, hence the 0-th de Rham filtration ${\rm F}^0={\rm Fil}^0$ vanishes in 
the notation of \cite{Ne}).  
\end{proof}

\subsubsection{Two dimensional, reducible case}\label{d2-red}
Let $\br:G_{\Q_p}\lra {\rm GL}_2(\bF_p)$ be a continuous reducible representation. It is easy to see that 
$\br\simeq 
\begin{pmatrix}
\op_1\ve^a & \ot \\
0 & \op_2\ve^b
\end{pmatrix}$ for some $0\le a,b\le p-2$ and for some unramified characters $\op_i:G_{\Q_p}\lra \bF^\times_p\ (i=1,2)$. 
We will study crystalline lifts of each element in $H^1(\Q_p,\br)$.
If the class of $\ot$ is trivial, it is reduced to the case of dimension one. 
Therefore, we may assume that the extension class of $\overline{\tau}$ is non-trivial. 
Put $\oc_1=\op_1\ve^a$ and $\oc_2=\op_2\ve^b$.  
By the local Euler characteristic formula, we see that 
$${\rm dim}_{\F}H^1(\Q_p,\br)=
\left\{
\begin{array}{cl}
4 &\ {\rm if} (\oc_1,\oc_2)=(\textbf{1},\ve)    \\ 
3 &\ {\rm if} \oc_1=\textbf{1},\ \oc_2\neq \ve\ {\rm or}\ \oc_1\neq 1,\oc_2=\ve      \\
2 &\ {\rm otherwise}
\end{array}\right..
$$
By Proposition \ref{prop-d1}, there exists a crystalline lift $\rho$ of $\br$ such that 
\begin{equation}\label{ht-red-d2}
{\rm HT}(\rho)=
\left\{
\begin{array}{ll}
\{a+p-1,b\}=\{b+p,b\} &\mbox{if $\oc_1\oc^{-1}_2=\ve_p$ and $\tau$ is tr\`es ramifi\'ee}    \\ 
\{a,b\} &\mbox{if $a>b$ and $\tau$ is peu ramifi\'ee}    \\
\{a+p-1,b\} &\mbox{if $a\le b$}
\end{array}\right..
\end{equation}
For such a $\rho$ and any integer $k$, $\rho\otimes\e^{k(p-1)}$ is also a crystalline lift of $\br$. 
Therefore, we can make all Hodge-Tate weights positive by twisting so that $H^1_f(\Q_p,\rho\otimes\e^{k(p-1)})=2$.

First we consider the case when $\oc_1\neq \textbf{1},\ve$. 
Consider the exact sequence 
$$0\lra H^1(\Q_p,\oc_1)\stackrel{\iota}{\lra} H^1(\Q_p,\br)\stackrel{\pi}{\lra} H^1(\Q_p,\oc_2)\lra  0.$$
Consider a class $x\in H^1(\Q_p,\br)$ which comes from a crystalline representation as in (\ref{shape}). 
As observed, $\alpha_2:=\pi(x)\in H^1(\Q_p,\oc_2)$ is ramified if $\oc_2=\textbf{1}$. 
Applying Proposition \ref{prop-d1}, there exists an unramified character 
$\psi_{\oc_2,\alpha_2}:G_{\Q_p}\lra \O^\times$ such that $\alpha_2$ is liftable to a rank one $\O$-module in  
$H^1_f(\Q_p,\chi_{2,\alpha_2})$ where $\chi_{2,\alpha_2}$ is a crystalline lift of $\oc_2$ 
depending also on $\alpha_2$ (recall $\psi_{\oc_2,\alpha_2}$ in Definition \ref{psi-and-r}). To be more precise, $\chi_{2,\alpha_2}$ is of form 
$$\chi_{2,\alpha_2}=\psi_{\oc_2,\alpha_2}\psi_2\e^{b+a_2(p-1)+r_{\oc_2,\alpha_2}}$$
for a non-negative integer $a_2$ and an unramified lift $\psi_2$ of $\op_2$. 

Applying Proposition \ref{prop-d1} to the extension class $\overline{\tau}$ in $H^1(\Q_p,\oc_1\oc^{-1}_2)$, 
we have the 
crystalline lifts of $\br$ defined by 
\begin{equation}\label{lift1}
\rho_{a_1,a_2,\ot,\alpha_2}=
\begin{pmatrix}
\psi_{\oc_2,\alpha_2}\psi_{\oc_1\oc^{-1}_2,\ot}\psi_1
\e^{a+a_1(p-1)+r_{\oc_1\oc^{-1}_2,\overline{\tau}}} & \ast \\
0 &\psi_{\oc_2,\alpha_2}\psi_2\e^{b+a_2(p-1)+r_{\oc_2,\alpha_2}}
\end{pmatrix}
\end{equation}
where $a_1\ge a_2$ is a non-negative integer such that 
\begin{equation}\label{cond1}
a+a_1(p-1)+r_{\oc_1\oc^{-1}_2,\overline{\tau}} > b+a_2(p-1)+r_{\oc_2,\alpha_2}>0
\end{equation} and 
$\psi_1$ is an unramified lift of $\op_1$ satisfying $\psi_1=\psi_2$ if 
$\oc_1\oc^{-1}_2=\ve$.  
Put 
$$\chi_1=\psi_{\oc_2,\alpha_2}\psi_{\oc_1\oc^{-1}_2,\ot}\psi_1
\e^{a+a_1(p-1)+r_{\oc_1\oc^{-1}_2,\overline{\tau}}}.$$ 
By Lemma \ref{useful4} , $H^1_f(\Q_p,\chi_1)$ is torsion free and the reduction map 
$H^1_f(\Q_p,\chi_1)\lra H^1(\Q_p,\oc_1)$ is surjective. 
Further, there is an exact sequence 
$$0\lra H^1_f(\Q_p,\chi_1)\lra H^1_f(\Q_p,\rho_{a_1,a_2,\ot,\alpha_2})\lra H^1_f(\Q_p,\chi)\lra 0$$
as free $\O$-modules.
Recall that $\alpha_2$ comes from an element of $H^1_f(\Q_p,\chi)$ and we lift it to 
an element $\widetilde{x}_{\alpha_2}
\in H^1_f(\Q_p,\rho_{a_1,a_2,\ot,\alpha_2})$ by the above exact sequence.  
The image of $\widetilde{x}_{\alpha_2}$ to $H^1(\Q_p,\br)$ and $x$ differ by an element in 
${\rm Im}(\iota)$. Therefore, we can adjust $\widetilde{x}_{\alpha_2}$ in 
$H^1_f(\Q_p,\rho_{a_1,a_2,\ot,\alpha_2})$  by an element of  
$H^1_f(\Q_p,\wc_1)$ so that $\widetilde{x}_{\alpha_2}$ goes to $x$. 
Hence, we have a crystalline lift of $x$ to $H^1_f(\Q_p,\rho_{a_1,a_2,\ot,\alpha_2})$.  

Next we consider the case when $\oc_1\neq \textbf{1}$ and $\oc_2=\ve$. 
We also assume that $\ve\neq \textbf{1}$ which happens exactly when $p>2$.  
Notice that by Lemma \ref{useful1}-(2), $H^1(G_{\Q_p},\rho)$ is always $\O$-torsion free for any 
lift $\rho$ to $\O$ of $\br$. 
Let $\rho_{\chi_1,\chi_2}$ be a crystalline lift of $\br$ which has the shape 
$$
\rho_{\chi_1,\chi_2}=
\begin{pmatrix}
\chi_1& \ast \\
0 & \chi_2 
\end{pmatrix}
$$
such that $H^2(\Q_p,\rho_{\chi_1,\chi_2}/\pi^2)=0$ 
where $\chi_i$ is a crystalline lift of $\oc_i$ for $i=1,2$.  Such a lift always exists for $\br$. 
The exact sequence 
$$0\lra \br\simeq \pi\rho_{\chi_1,\chi_2}/\pi^2\lra 
\rho_{\chi_1,\chi_2}/\pi^2 \stackrel{{\rm mod}\ \pi}{\lra} \br \lra 0$$
yields 
$$H^1(\Q_p,\rho_{\chi_1,\chi_2}/\pi^2)\stackrel{{\rm mod}\ \pi}{\lra}  
H^1(\Q_p,\br)\stackrel{\delta_{\chi_1,\chi_2}}{\lra} H^2(\Q_p,\br)$$
where $\delta_{\chi_1,\chi_2}$ is the connection map. 
It is easy to see that 
$H^1(\Q_p,\rho_{\chi_1,\chi_2})\stackrel{{\rm mod}\ \pi^2}{\lra}  
H^1(\Q_p,\rho_{\chi_1,\chi_2}/\pi^2)$ is surjective since 
$H^2(\Q_p,\rho_{\chi_1,\chi_2}/\pi^2)=0$. Therefore, we have an exact sequence 
$$H^1(\Q_p,\rho_{\chi_1,\chi_2})\stackrel{{\rm mod}\ \pi}{\lra}  
H^1(\Q_p,\br)\stackrel{\delta_{\chi_1,\chi_2}}{\lra} H^2(\Q_p,\br)\simeq H^2(\Q_p,\oc_2)=
H^2(\Q_p,\ve)\simeq \F.$$
We write $\chi_2\e^{-1}=1+u\pi$ with a map $u:G_{\Q_p}\lra \O$ whose reduction 
$\overline{u}:G_{\Q_p}\lra \F$ is an additive homomorphism which factors through 
$G^{{\rm ab}}_{\Q_p}/(G^{{\rm ab}}_{\Q_p})^p$. 
By Lemma \ref{conn}, for each $\alpha\in H^1(\Q_p,\br)$ with the image $\alpha_2$ 
to $H^1(\Q_p,\oc_2)=H^1(\Q_p,\ve)\simeq  \Q^\times_p/(\Q^\times_p)^p$, we have 
$$\delta_{\chi_1,\chi_2}(\alpha)=\overline{u}(\alpha_2).$$ 
If $\delta_{\chi_1,\chi_2}(\alpha)=0$, then $\alpha$ is liftable to $H^1(\Q_p,\rho_{\chi_1,\chi_2})$. 
By choosing $\rho_{\chi_1,\chi_2}$ suitably, $\alpha$ is also liftable to $H^1_f(\Q_p,\rho_{\chi_1,\chi_2})$. 
If $\delta_{\chi_1,\chi_2}(\alpha)\neq 0$, there exists another crystalline lift $\chi'_2$ of $\oc_2$ such that 
$\delta_{\chi_1,\chi'_2}(\alpha)=0$. 
Therefore, $\alpha$ is liftable to $H^1_f(\Q_p,\rho_{\chi_1,\chi'_2})$. 

When $\oc_1\neq \textbf{1},\ \oc_2\neq \ve$, and $\ve\neq \textbf{1}$, there exists 
a crystalline lift to  $H^1_f(\Q_p,\rho_{\chi_1,\chi_2})$ since the reduction map 
$H^1(\Q_p,\rho)\lra H^1(\Q_p,\br)$ is surjective and $H^1(\Q_p,\rho)$ is $\O$-torsion free for 
any lift $\rho$ to $\O$ of $\br$. 

Finally, we consider the when $\oc_1=\textbf{1}$. 
Applying the argument (in using the local Tate duality) 
of the proof of Proposition \ref{prop-d1}-(2),(3), 
we have a lift to a rank one $\O$-submodule in $H^1_f(\Q_p,\rho_{\chi_1,\chi_2})$ for 
specific $\chi_i\ (1\le i \le 2)$. The non-triviality of $\ot_1$ guarantees any lift of each class of 
$H^1(\Q_p,\br)$ is not torsion. Hence it is liftable to a rank one $\O$-submodule of 
$H^1_f(\Q_p,\rho_{\chi_1,\chi'_2})$.  
Summing up, we have proved the following:
\begin{prop}\label{d2} Let $p\ge 2$ be a prime. Let $\br=\begin{pmatrix}
\oc_1 & \ot \\
0   & \oc_2  
\end{pmatrix} $.  
Then each class of  $H^1(\Q_p,\br)$ is liftable to an ordinary, crystalline extension with 
specific regular Hodge-Tate weights. 
\end{prop}

\subsubsection{Three dimensional, reducible case 1}\label{d3-red1}Assume $p>2$.  
Let $\br:G_{\Q_p}\lra {\rm GL}_3(\bF_p)$ be a continuous reducible representation which is of form   
$\br\simeq 
\begin{pmatrix}
\op_1\ve^a & \overline{\tau}_1 & \overline{\tau}_2 \\
0 & \op_2\ve^b & \overline{\tau}_3 \\
0& & \op_3\ve^c
\end{pmatrix}$ for some $0\le a,b,c\le p-2$ and for some unramified characters $\op_i:G_{\Q_p}\lra \bF^\times_p\ (i=1,2,3)$. 
If $\overline{\tau}_1=0$ (resp. $\overline{\tau}_3=0$), 
then by changing basis if necessary, we may assume that $a\ge b$ (resp. $b\ge c$). 
With the same reason, we may assume $a\ge b\ge c$ if all $\overline{\tau}_i\ (i=1,2,3)$ are trivial 
(hence when $\br$ is tame). 
We may also assume that any $\tau_i$ ($i=1,2,3$) is non-trivial. 
Put $\oc_1=\op_1\ve^a,\ \oc_2=\op_2\ve^b$, and $\oc_3=\op_3\ve^c$. 
Put $\br_1= 
\begin{pmatrix}
\oc_1 & \overline{\tau}_1 \\
0 & \oc_2
\end{pmatrix}$.  
By Proposition \ref{d2}, we have 
a crystalline lift of $\br$ whose shape takes 
$$\rho_{\chi_1,\chi_2,\chi_3}=
\begin{pmatrix}
\chi_1 & \ast & \ast \\
0 & \chi_2 & \ast \\
0 & 0 & \chi_3
\end{pmatrix}
$$
such that ${\rm HT}(\chi_1)>{\rm HT}(\chi_2)>{\rm HT}(\chi_3)$. 
If $\oc_3\neq \textbf{1}$, then we have the surjection $H^1(\Q_p,\rho_{\chi_1,\chi_2,\chi_3})
\lra H^1(\Q_p,\br)$. Since all extensions are supposed to be non-trivial, $x$ is liftable to 
a rank one $\O$-submodule in $H^1_f(\Q_p,\rho_{\chi_1,\chi_2,\chi_3})$. 
If $\oc_3=\textbf{1}$, then we apply the argument  (in using the local Tate duality)  of the proof of Proposition \ref{prop-d1}-(2),(3), 
we have a lift to a rank one $\O$-submodule in $H^1_f(\Q_p,\rho_{\chi_1,\chi_2,\chi_3})$ for 
specific $\chi_i\ (1\le i \le 3)$. 
\begin{prop}\label{d3-nonexcep} Keep the notation as above. Each element $x\in H^1(\Q_p,\br)$ is liftable to a rank one $\O$-submodule in $H^1_f(\rho_{\chi_1,\chi_2,\chi_3})$ for 
$\rho_{\chi_1,\chi_2,\chi_3}$ with specific characters $\chi_i\ (1\le i \le 3)$ with regular Hodge-Tate weights.   
\end{prop}
 
\subsubsection{Three dimensional, reducible case 2}\label{d3-red2} 
Let $\br:G_{\Q_p}\lra {\rm GL}_3(\bF_p)$ be a continuous reducible representation which is of type  
$\br\simeq \begin{pmatrix}
\br_1 & \overline{\tau}_1  \\
0 & \op\ve^c
\end{pmatrix}$ or $\begin{pmatrix}
\op\ve^c & \overline{\tau}_1  \\
0 & \br_1
\end{pmatrix}$ for $0\le c \le p-2$, an unramified character $\op:G_{\Q_p}\lra \bF^\times_p$, and 
an irreducible representation $\br_1:G_{\Q_p}\lra {\rm GL}_2(\bF_p)$. We may assume that $\ot_1$ is non-trivial. 

Let us first consider the former case. 
As seen in Section \ref{d2-irr}, there are integers $0\le b<a<p$ such that 
$\br_1\simeq {\rm Ind}^{G_{\Q_p}}_{G_{\Q_{p^2}}}\omega^{b+ap}_2$. 
Put $\oc=\op\ve^c$.   
We may regard $H^1(\br_1)$ as a subspace of $H^1(\br)$ by the following exact sequence 
$$0\lra H^1(\br_1)\lra H^1(\br)\lra H^1(\op\ve^c)\lra 0.$$
We denote by $W_2$ its complement so that it surjects onto $H^1(\op\ve^c)$. 
Pick $x\in H^1(\br)$ and write $x=\alpha_1+w_2,\ \alpha_1\in H^1(\br_1)$. 
We denote by $\alpha_2$ the image of $w_2$ to $H^1(\op\ve^c)$.   
Applying 
Proposition \ref{cla-irr} to $\ot_1$, 
we can lift $\br$ to 
a crystalline lift 
$$\rho_{a_1,a_2,\alpha_2}=\begin{pmatrix}
\rho_{1,a_1} & \ast  \\
0 & \psi_{\oc,\alpha_2}\psi \e^{c+a_2(p-1)+r_{\oc,\alpha_2}}
\end{pmatrix}$$ where 
$\rho_{1,a_1}$ is a crystalline lift of $\br_1$ with Hodge-Tate weights $\{a+a_1(p-1),b+a_1(p-1)\}$ for $a_1\ge 0$ 
and $\psi:G_{\Q_p}\lra \O^\times$ is an unramified lift of $\op$. 
We assume that 
\begin{equation}\label{3-ht}
 b+a_1(p-1)>c+a_2(p-1)+r_{\oc,\alpha_2}>0. 
\end{equation}

If we choose a lift of $\alpha_2$ 
to an element $\widetilde{x}$ of $H^1_f(\rho_{a_1,a_2,\alpha_2})$ via the natural surjection 
to $H^1_f(\wc)$ where $\wc= \psi_{\oc,\alpha_2}\psi \e^{c+a_2(p-1)+r_{\oc,\alpha_2}}$ and send it to $H^1(\br)$ via the reduction map, 
then its difference from $x$ is an element of $H^1(\br_1)$ which comes from an element of $H^1_f(\rho_{1,a_1})$. 
By adjusting in $H^1_f(\rho_{a_1,a_2,\alpha_2})$ with an element of $H^1_f(\rho_{1,a_1})$, we find 
a lift of $x$ to $H^1_f( \rho_{a_1,a_2,\alpha_2})$ as desired. 

Next we consider when $\br=\begin{pmatrix}
\oc & \overline{\tau}_1  \\
0 & \br_1
\end{pmatrix}$ with $\oc=\op\ve^c$. 
By Proposition \ref{d2-irr}, 
 there exists a crystalline lift $\rho_{\chi_1,a_1}$ of form:
$$\rho_{\chi_1,a_1}=\begin{pmatrix}
\chi_1 & \tau_1  \\
0 & \rho_{a_1}
\end{pmatrix}$$
such that ${\rm HT}(\chi_1)>{\rm HT}(\rho_{a_1})>0$. 
Since $\br_1$ is irreducible, the reduction map $H^1(\Q_p,\rho_{\chi_1,a_1})\lra H^1(\Q_p,\br)$ is surjective. 
As observed before, any lift of each class of $H^1(\Q_p,\br)$ is non-torsion, since $\ot_1$ is non-trivial. 
Summing up, we have proved:
\begin{prop}\label{d3-case2}Let $\br$ be as in the  beginning of this section. 
Each class of $H^1(\Q_p,\br)$  is liftable to a rank one $\O$-submodule in 
$H^1_f(\Q_p,\rho_{\chi_1,a_1})$ for 
$\rho_{\chi_1,a_1}$ of regular Hodge-Tate weights with specific characters $\chi_1$ and $a_1$.   
\end{prop}

\subsection{An explicit construction of potentially diagonalizable, crystalline lifts for 
$GSp_4$}\label{exp}
Let $\br:G_{\Q_p}\lra {\rm GSp}_4(\bF_p)$ be a continuous representation. Assume $p>2$ for simplicity. 
It is easy to see that the similitude character $\nu\circ \br$ is the product of a power of $\ve$ and an unramified character. 
We regard $\br$ as a representation to ${\rm GL}_4(\bF_p)$ and denote by $\br^{{\rm ss}}$ the 
semi-simplification of $\br$.   
By Proposition 7.2 of \cite{yam} we have five types:
\begin{enumerate}
\item (Borel ordinary case) $\br^{{\rm ss}}$ is decomposed into the direct sum of characters; 
\item (Siegel ordinary case) $\br^{{\rm ss}}$ is the direct sum of two 1-dimensional representations and 
one 2-dimensional irreducible representation; 
\item (Klingen ordinary case) $\br^{{\rm ss}}$ is the direct sum of two 2-dimensional irreducible representations; 
\item (Endoscopic case) $\br=\br^{{\rm ss}}$ is the direct sum of two 2-dimensional irreducible 
representations with the same determinant character and these two representations are not isomorphic each other 
even after twisting by a character; 
\item (Irreducible case) $\br$ is irreducible. 
\end{enumerate}  
In view of our purpose, in Borel ordinary case, we may assume that it comes from an 
ordinary, crystalline representation over $\O$.   
\subsubsection{Borel ordinary case} In this case, it follows from Proposition 7.2 of \cite{yam} that 
$$\br\simeq \ve^c\op_0 \otimes
\begin{pmatrix}
\br_1 & B \\
0_2 & \op_1\ve^{a+b}\br^\ast_1 
\end{pmatrix}.$$
Here  
$\br_1=\begin{pmatrix}
 \op_1\ve^{a+b} & \ot_0 \\
0 & \op_2\ve^{a}  
\end{pmatrix}$ and 
for each $0\le i\le 2$, $\psi_i:G_{\Q_p}\lra \bF^\times_p$ is an unramified character and $0\le a,b,c \le p-2$. 
Notice that the class of $B$ belongs to $H^1(\Q_p,(\op_1\ve^{a+b})^{-1}{\rm Sym}^2(\br_1))$ where 
the space $(\op_1\ve^{a+b} )^{-1}{\rm Sym}^2(\br_1)$ is of dimension three and 
${\rm Sym}^2(\br_1)$ stands for the symmetric square representation of $\br_1$. 

Then, by Proposition \ref{d3-nonexcep}, there exists a crystalline lift of $\br$ to ${\rm GSp}_4(\O)$ with regular 
Hodge-Tate weights. Since such a lift is ordinary, it is also potentially diagonalizable.  
 
\subsubsection{Siegel ordinary case} In this case, it follows from Proposition 7.2 of \cite{yam} that 
\begin{equation}\label{shape-S}
\br\simeq \ve^c\op_0 \otimes
\begin{pmatrix}
 \op_2\ve^{a+b} & \ot_1 & \ot_3 \\ 
0 & \br_1& \ot_2 \\
0 & 0& 1 
\end{pmatrix}
\end{equation}
where $\br_1\simeq \op_1\otimes{\rm Ind}^{G_{\Q_p}}_{G_{\Q_{p^2}}}\omega^{b+ap}_2$ with $0\le b<a\le p$,  
$0\le c\le p-2$, and $\op_0,\op_1,\op_2$ are unramified characters which satisfy $\op^2_1=\det(\br_1)\ve^{-(a+b)}=\op_2$. 
If we write $\ot_1:G_{\Q_p}\lra V_1$ and $\ot_2:G_{\Q_p}\lra V_2$ where 
$V_1=\F^2$ consists of row vectors while $V_2=\F^2$ consists of column vectors. 
Since $\br$ is a representation to ${\rm GSp}_4(\F)$, it is easy to see that 
the isomorphism $V_1\lra V_2, (x_1,x_2)\mapsto 
\begin{pmatrix}
-x_2\\
x_1
\end{pmatrix}$ induces an isomorphism $H^1(\Q_p,\op_2\ve^{a+b}\otimes \br^\ast_1)\stackrel{\sim}{\lra}
H^1(\Q_p,\br_1)$ such that $\ot_1$ goes to $\ot_2$. Put $\br_2=\begin{pmatrix}
\oc & \ot_1 \\ 
0 & \br_1
\end{pmatrix}$ with $\oc= \op_2\ve^{a+b}$. 

From Section \ref{d2-irr}, we have a crystalline lift $$\rho_2=\begin{pmatrix}
 \psi_2\e^{a+b+a_2(p-1)} & \tau_1 \\ 
0 & \rho_{1,a_1}
\end{pmatrix}$$
of $\br_2$ such that $\det(\rho_{1,a_1})=\psi_2\e^{a+b+a_2(p-1)}$ for some non-negative integers $a_1,a_2$. Consider the exact sequence 
$0\lra H^1(\oc)\lra H^1(\br_2)\lra H^1(\br_1)\lra 0$. 
The extension class defined by $\br$ gives an element $x$ in $ H^1(\br_2)$ sent to $\ot_2$. 
Applying Section \ref{d3-red2}, there exists a crystalline lift $\rho$ to ${\rm GSp}_4(\O)$ of $\br$ which has the shape:
$$\rho=\psi\otimes 
\begin{pmatrix}
\rho_2 & \ast \\
0 & \textbf{1}
\end{pmatrix}
$$
of regular Hodge Tate weights for some crystalline character $\psi$. 
By the properties (5),(7) in p.530 and the proof of Lemma 1.4.3 of \cite{BGGT}, $\rho$ is potentially diagonalizable 
(note that $\rho_{1,a_1}$ is potentially diagonalizable as observed).

\subsubsection{Klingen ordinary case} 
 In this case, it follows from Proposition 7.2 of \cite{yam} that 
$$\br\simeq 
\begin{pmatrix}
\br_1 & \ast \\ 
0_2 & \op_0\ve^c \br_2
\end{pmatrix}$$ 
where $\br_1\simeq \op_1\otimes{\rm Ind}^{G_{\Q_p}}_{G_{\Q_{p^2}}}\omega^{b+ap}_2$ with $0\le b<a\le p$ and 
$\br_2\simeq \br^\ast_1$. We may assume $\ast$ is non-trivial and it belongs to  
$H^1(\Q_p,\op^{-1}_0\ve^{-c} {\rm ad}^0(\br_1))$ since the unipotent radical of the Siegel parabolic subgroup has 
the structure of ${\rm Sym}^2({\rm St}_2)$.  
If $p>2$,  ${\rm ad}^0(\br_1)\simeq \delta\oplus  
{\rm Ind}^{G_{\Q_p}}_{G_{\Q_{p^2}}}\omega^{(p-1)(a-b)}_2$ 
where $\delta:G_{\Q_p}\lra \bF^\times_p$ is the quadratic character associated to $\Q_{p^2}/\Q_p$. 
The second factor is irreducible if and only if 
$(p+1)\nmid2(a-b)$. If $(p+1)|2(a-b)$, put $m=2(a-b)/(p-1)$.  
Then 
$${\rm Ind}^{G_{\Q_p}}_{G_{\Q_{p^2}}}\omega^{(p-1)(a-b)}_2\simeq  \ve^{\frac{m(p-1)}{2}}
\oplus \delta\ve^{\frac{m(p-1)}{2}}.$$
Assume that by using results in Section \ref{d3-red1}, \ref{d3-red2}, we have 
a potentially diagonalizable crystalline lift $\rho$ to ${\rm GSp}_4(\O)$ of $\br$.  

\subsubsection{Endoscopic type} In this case, we have 
$\br\simeq \br_1\oplus \br_2$ for some irreducible 2-dimensional representations $\br_1,\br_2$ with the same determinant.  
We write $\br_i\simeq \op_i\otimes  {\rm Ind}^{G_{\Q_p}}_{G_{\Q_{p^2}}}\omega^{b_i+pa_i}_2$ with 
$0\le b_i<a_i<p$ for $i=1,2$. For $i=1,2$, as explained in Section \ref{d2-irr}, let $\rho_{i,c_i}$ be a crystalline lift of $\br_i$ of 
Hodge-Tate weights $\{a_i+c_i(p-1),b_i+c_i(p-1)\}$ such that $\det(\rho_{1,c_1})=\det(\rho_{2,c_2})$ and 
Hodge-Tate weights of $\rho:=\rho_{1,c_1}\oplus \rho_{2,c_2}$ are different each other. The lift $\rho$ 
takes the values in ${\rm GSp}_4(\O)$ and it is also potentially diagonalizable.   

\subsubsection{Irreducible case}\label{irr-d4}
Let $\br:G_{\Q_p}\lra {\rm GSp}_4(\bF_p)$ be a continuous irreducible representation. 
By  Proposition 7.2 of \cite{yam} and Section 2.5 of \cite{Muller}), 
there exists an integer $$a=a_0+a_1p+a_2p^2+a_3p^3,\ 0\le a_i\le p-1, 
a\not\equiv 0\ {\rm mod}\ p^2+1,\ a\equiv 0\ {\rm mod}\ p+1$$  such that 
$\br\simeq \op\otimes {\rm Ind}^{G_{\Q_p}}_{G_{\Q_{p^4}}}\omega^{a}_4$ where $\Q_{p^4}/\Q_p$ stands for the 
unramified cyclic extension of degree four, $\op:G_{\Q_p}\lra\bF^\times_p$ is an unramified character and $\omega_4:G_{\Q_{p^2}}\lra \bF^\times_p$ is a unique extension of 
the fundamental character of the inertia subgroup $I_p$ of level four. 
For any $\sigma \in {\rm Gal}(\Q_{p^4}/\Q_p)$, we denote by $\omega_{4,\sigma}$ the twist of $\omega_4$ by $\sigma$. 
Let $\tau$ be the generator of ${\rm Gal}(\Q_{p^4}/\Q_p)$. 
Once we fix $\omega_4$, we have $\omega^{a}_4=\omega^{a_0}_{4,\tau^0}\omega^{a_1}_{4,\tau^1}\omega^{a_2}_{4,\tau^2}
\omega^{a_3}_{4,\tau^3}$. 

Let $\chi_4$ be the composition of the homomorphisms 
$$G_{\Q_{p^4}}\lra G^{{\rm ab}}_{\Q_{p^4}}\stackrel{\sim}{\lra} \widehat{\Z}\times 
\O^\times_{\Q_{p^4}}\stackrel{{\rm proj}}{\lra} \O^\times_{\Q_{p^4}}$$
such that the reduction $\chi_4$ modulo $p$ is $\omega_4$. Here the second isomorphism is 
given by the local class field theory and the third map is the natural projection. 
For any $\sigma \in {\rm Gal}(\Q_{p^4}/\Q_p)$, we denote by $\chi_{4,\sigma}$ the twist of $\chi_4$ by $\sigma$. 
Let $\tau$ be the generator of ${\rm Gal}(\Q_{p^4}/\Q_p)$. 
Fix $\Q_{p^4}\hookrightarrow \bQ_p$. 
Let $\psi:G_{\Q_p}\lra \bZ^\times_p$ be a 
crystalline lift of $\op$. 

The congruent condition implies $a_0+a_2\equiv a_1+a_3$ mod $p+1$. We may assume that $ a_0+a_2\ge a_1+a_3$. 
Then we have $a_0+a_2=a_1+a_3+i(p+1),\ i=0,1$. When $i=0$, 
then $$\rho:=\psi\otimes {\rm Ind}^{G_{\Q_p}}_{\Q_{p^4}}\chi^{a_0}_{4,{\tau^0}}
\chi^{a_1}_{4,{\tau^1}}\chi^{a_2}_{4,{\tau^2}}\chi^{a_3}_{4,{\tau^3}}$$
is crystalline lift of $\br$ of $HT(\rho)=\{c+a_0,c+a_1,c+a_2,c+a_3\}$ where $c=HT(\psi)$. 
Since Hodge-Tate weights have the symmetry so that $(c+a_0)+(c+a_2)=(c+a_1)+(c+a_3)$, 
the lift $\rho$ takes the values in ${\rm GSp}_4(\O)$ for some finite extension $E/\Q_p$ such that $\O=\O_E$ includes  
$\O_{\Q_{p^4}}$ and ${\rm Im}(\psi)$. 
If $i=1$, then $a_0=a_1+a_3-a_2+1+p$ and we have 
$$\omega^a_4=\omega^{a_1+a_3-a_2+1}_4\omega^{(a_1+1)p}_4\omega^{a_2p^2}_4\omega^{a_3p^3}_4.$$
Re-indexing as $b_0=a_1+a_3-a_2+1, b_1=a_1+1,\ b_2=a_2, b_3=a_3$, we have the symmetry so that 
$b_0+b_2=b_1+b_3$. It follows from this that 
$$\rho:=\psi\otimes {\rm Ind}^{G_{\Q_p}}_{\Q_{p^4}}\chi^{b_0}_{4,{\tau^0}}
\chi^{b_1}_{4,{\tau^1}}\chi^{b_2}_{4,{\tau^2}}\chi^{b_3}_{4,{\tau^3}}$$
is a crystalline lift of $\br$ to ${\rm GSp}_4(\O)$ for a suitable integral coefficient $\O$ as above 
such that 
$$HT(\rho)=\{c+b_0,c+b_1,c+b_2,c+b_3\}$$
where $c=HT(\psi)$.

\section{Classical (naive) Serre weights}\label{Serre-weights}
Assume $p>2$. 
Let $\br:G_{\Q}\lra {\rm GSp}_4(\bF_p)$ be an irreducible mod $p$ Galois representation 
which is automorphic of level prime to $p$. 
Let $SW(\br)$ be the set of all $(a,b,c)\in \Z_{\ge 0}\times \Z_{\ge 0}\times \Z_{\ge 0}$ such that 
\begin{itemize}
\item $a>b>0$; 
\item there is a potentially diagonalizable, crystalline lift $\rho$ of $\br$ such that  
${\rm HT}(\rho)=\{a+b+c,a+c,b+c,c\}$. 
\end{itemize}
The results in previous section show  $SW(\br)$ is always non-empty. Put 
$$SW^{{\rm cl}}(\br):=SW(\br)+(1,2,0)
=\{(a+1,b+2,c)\ |\ (a,b,c)\in SW(\br)\}.$$ 
\begin{dfn}The classical Serre weight $(k_1(\br),k_2(\br),w(\br))$ of $\br$ 
is defined by the minimum element of $SW^{{\rm cl}}(\br)$ in lexicographic order.   
\end{dfn}
The following result is an variant of Edixhoven's observation of the minimality of the classical Serre weights 
(see Theorem 4.5 of \cite{Edix}):
\begin{thm}\label{va-edix}Assume $p>2$.  Let $\br:G_{\Q}\lra {\rm GSp}_4(\bF_p)$ be an irreducible mod $p$ Galois representation which is automorphic of level prime to $p$. 
Assume that $\br|_{G_{\Q(\zeta_p)}}$ is irreducible and $\br(G_{\Q(\zeta_p)})$ is adequate. 
Then there exists a Hecke eigen Siegel cusp form $F$ of weight $(k_1(\br),k_2(\br))$ satisfying 
$k_1(\br)\ge k_2(\br)\ge 3$ such that 
\begin{enumerate}
\item $\br\simeq \br_{F,\iota_p}\otimes \ve^{w(\br)}\op$ for some character $\op:G_\Q\lra \bF^\times_p$ of conductor prime to $p$;
\item the level of $F$ is prime to $p$. 
\end{enumerate}
\end{thm}
\begin{proof}The claim follows from Theorem \ref{main-thm2}. 
\end{proof}

\subsection{Examples}\label{examples}
In the ordinary case, we give an explicit form of $(k_1(\br),k_2(\br))$ for 
$\br$ in Theorem \ref{va-edix}. 

Let us suppose $$\br|_{G_{\Q_p}}\simeq \ve^c\op_0 \otimes
\begin{pmatrix}
\br_1 & B \\
0_2 & \op_1\ve^{a+b}\br^\ast_1 
\end{pmatrix}.$$
Here  
$\br_1=\begin{pmatrix}
 \op_1\ve^{a+b} & \tau_0 \\
0 & \op_2\ve^{a}  
\end{pmatrix}$ and 
for each $0\le i\le 2$, $\psi_i:G_{\Q_p}\lra \bF^\times_p$ is an unramified character and 
$a,b,c\in \Z$. 
For simplicity we assume that 
$$p-1\ge a>b>c=0.$$
Recall that the class of $B$ belongs to 
$H^1(\Q_p,(\op_1\ve^{a+b})^{-1}{\rm Sym}^2(\br_1))$.  
By using a suitable basis, $\br_1$ is given explicitly by 
$${\rm Sym}^2(\br_1)=
\begin{pmatrix}
\op^2_1\ve^{2a+2b} & 2\op_1\ve^{a+b}\ot_0 & \ot^2_0 \\
          0             & \op_1\op_2 \ve^{2a+b} & \op_2 \ve^a \ot_0 \\
          0             &          0                 & \op^2_2 \ve^{2a} 
\end{pmatrix}.
$$
The $(2,1)$ and $(1,2)$ entries in the above 3 by 3 matrix yield the same extension class 
$\ot_0$ in 
$H^1(\Q_p,\op_1\op^{-1}_2\ve^b)$. 
If the extension class is peu ramifi\'ee, then 
$$(k_1(\br),k_2(\br))=(a+1,b+2).$$
If the extension class is tr\`es ramifi\'ee, hence when $\op_1\op^{-1}_2\ve^b=\ve$, 
then we need to raise 
$b=1$ to $b+p-1=p$ and accordingly $a$ to $a+p-1$. 
Therefore, we will find 
$$(k_1(\br),k_2(\br))=(a+p,p+2).$$
Note that the condition on $a$ and $b$ shows $a\ge 2$.

\section{A proof of Theorem \ref{main-thm3}}In this section we give a proof of Theorem \ref{main-thm3}. 
Before doing that, we need to introduce some genericity condition for a mod $p$ Galois representation and its 
lift to a $p$-adic Galois representation which is potentially diagonalizable and crystalline at each finite place above $p$. 

\subsection{Genericity condition} 
Let $p$ be an odd prime. 
Let $\br:G_F\lra {\rm GSp}_4(\bF_p)$ be an irreducible automorphic mod $p$ Galois representation. 
Suppose that $p$ is split completely in $F$. Assume that for each finite place $v$ of $F$ lying over $p$, 
$\br|_{G_{F,v}}$ is semisimple. As is discussed in Section \ref{exp}. 
We have fives types for $\br|_{G_{F,v}}$. We handle each case individually and discuss about its desired lift to 
study possible Serre weights. 
\subsubsection{Borel ordinary case}\label{BOC}
In this case we have 
$$\br|_{I_{F,v}}\simeq (\ve^{x+y}\oplus \ve^{x}\oplus \ve^{y}\oplus \textbf{1})\otimes \ve^{\delta},\ \delta=\frac{1}{2}(w+3-x-y)\in \Z$$ 
for some integers $x,y\in \Z_{\ge 0},w\in \Z$. Notice that $x+y\equiv w+1$ mod 2. 
Put $\lambda_{1,v}:=(a_1,b_1;c_1)=(x,y;w+3)-\widetilde{\rho}=(x-2,y-1;w)$ and it satisfies $a_1+b_1\equiv c_2$ mod 2. 
Hence $\lambda_{1,v}\in X(T)$. 
When $x>y$, we see that $\lambda_{1,v}\in X_1(T)$ if and only if $x-y\le p$ and $0<y\le p$. 
When $x\le y$, put $x'=x+p-1$. 
Then $\lambda'_{1,v}=(x'-2,y-1;w)\in X_1(T)$ if and only if $1\le x'-y$ and $0<y\le p$. 
It follows that $\lambda_{1,v}\in \overline{C_0-\widetilde{\rho}}\cup \overline{C_1-\widetilde{\rho}}$ 
and $\lambda'_1\in \overline{C_3-\widetilde{\rho}}\cup \overline{C_4-\widetilde{\rho}}$. Clearly, $\lambda_{1,v},\lambda'_{1,v}$ 
give all possible choices so that it belongs to 
$X_1(T)\cap \overline{C_0-\widetilde{\rho}}\cup \overline{C_1-\widetilde{\rho}}$ for $\br$ because, 
in this range, we can not raise $x,y$ by a multiple of $p-1$ anymore. 
In either case, we exclude the case when $x,y\in (p-1)\Z$ as we can not associate any elements in $X_1(T)$. 

\begin{dfn}\label{gen-borel}
Assume $\lambda_{1,v}:=(m_1-2,m_2-1;w)\in X(T)$ associated to $\br|_{G_{F,v}}$ 
as above lies in $X_1(T)$. Then it is said to be generic 
\begin{enumerate}
\item if $\lambda_{1,v}\in C_1$ and $x+y>p+1$, or;  
\item if $\lambda_{1,v}\in C_0$. 
\end{enumerate}
\end{dfn}
We do not discuss $\lambda'_{1,v}$ because in the course of proof, we will see that 
it is not applicable to Fontaine-Laffaille theory. To handle this case we need a further study of 
integral $p$-adic Hodge theory to describe the mod $p$ reductions of crystalline lifts in the higher 
dimensional case. 

\subsubsection{Siegel ordinary case}\label{SOC}
In this case we have 
$$\br|_{I_{F,v}}\simeq (\ve^{x+y}\oplus (\omega^{x+y p}_2\oplus \omega^{y+x p}_2)\oplus \textbf{1})\otimes \ve^{\delta},\ \delta=\frac{1}{2}(w+3-x-y)\in \Z$$ 
for some integers $0\le y<x\le p-1$ and $w\in \Z$. Notice again that $x+y\equiv w+1$ mod 2. 
Put $\lambda_{1,v}:=(a_1,b_1;c_1)=(x,y;w+3)-\widetilde{\rho}=(x-2,y-1;w)$. 
Then $\lambda_{1,v}\in X_1(T)$ if and only if $y>0$. 
Further, it belongs to $\overline{C_0-\widetilde{\rho}}\cup \overline{C_1-\widetilde{\rho}}$. 
Since $\omega^{p^2-1}_2=\textbf{1}$ we can change $x,y$ so that the corresponding highest weight 
 lies on $\overline{C_2-\widetilde{\rho}}\cup \overline{C_3-\widetilde{\rho}}$ though we do not consider 
this case with the same reason in Fontaine-Laffaille theory.  
Otherwise, there is no way to associate it to an element in $X_1(T)$. 
\begin{dfn}\label{gen-Siegel}
Assume $\lambda_{1,v}:=(m_1-2,m_2-1;w)\in X(T)$ associated to $\br|_{G_{F,v}}$ 
as above lies in $X_1(T)$. Then it is said to be generic 
\begin{enumerate}
\item if $\lambda_{1,v}\in C_1$ and $x+y>p+1$, or;  
\item if $\lambda_{1,v}\in C_0$. 
\end{enumerate}
\end{dfn}

\subsubsection{Klingen ordinary case}\label{KOC}
In this case we have 
$$\br|_{I_{F,v}}\simeq (\omega^{x+y p}_2\oplus \omega^{y+x p}_2)\oplus 
((\omega^{-(x+y p)}_2\oplus \omega^{-(y+x p)}_2)\otimes \ve^{w+3})$$ 
for some integers $0\le y<x\le p-1$ and $0\le w \le p-2$.  
If $\br$ satisfies the condition in Theorem \ref{main-thm2}, then 
it is easy to see that $x+y\equiv w+1$ mod 2 by using automorphy. 
Therefore, we may keep this parity condition in this case. 
To associate an element of $X_1(T)$ we need a bit more observation. 
As in Section \ref{prescribed}, there exists a crystalline lift of $\br|_{I_{F,v}}$ whose 
Hodge-Tate weights are $$\{x,y,w+3-y,w+3-x\}.$$
By duality we may assume that $y\ge w+3-y$ and it yields $x>y\ge w+3-y > w+3-x$. 
Put $\lambda_{1,v}:=(a_1,b_1;c_1)=(m_1,m_2;w+3)-\widetilde{\rho}=(m_1-2,m_2-1;w)$ where 
$$m_1=x+y-(w+3),\ m_2=x-y.$$ 
Then $\lambda_{1,v}\in X_1(T)$ if and only if $0<2y-(w+3)\le p$. 
Further, it belongs to $\overline{C_0-\widetilde{\rho}}\cup \overline{C_1-\widetilde{\rho}}$ if and only if 
$m_1=x+y-(w+3)\le p$. Otherwise it 
lies on $\overline{C_2-\widetilde{\rho}}\cup \overline{C_3-\widetilde{\rho}}$ though we do not consider 
this case with the same reason in Fontaine-Laffaille theory. There is no other way to associate $\br$ with an element in 
$X_1(T)$.  
 
\begin{dfn}\label{gen-Klingen}
Assume $\lambda_{1,v}:=(m_1-2,m_2-1;w)\in X(T)$ associated to $\br|_{G_{F,v}}$ 
as above lies in $X_1(T)$. Then it is said to be generic 
\begin{enumerate}
\item if $\lambda_{1,v}\in C_1$ with $m_1+m_2>p+1$ and $m<p-1$, or;  
\item if $\lambda_{1,v}\in C_0$. 
\end{enumerate}
\end{dfn}

\subsubsection{Endoscopic case}\label{EC}
In this case, $\br|_{G_{F,v}}$ is isomorphic to the product of two 2-dimensional representations. 
If either or Both of two is reducible, it falls down to the case of Borel ordinary case or Siegel ordinary case. 
Therefore, we may assume these two representations are irreducible.  
Then we have 
$$\br|_{I_{F,v}}\simeq ((\omega^{a+b p}_2\oplus \omega^{b+a p}_2)\oplus 
(\omega^{c+d p}_2\oplus \omega^{c+d p}_2))\otimes \ve^{e}$$ 
for some integers $0\le b<a\le p-1,\ 0\le d<c\le p-1$, and 
$0\le e\le p-1$ satisfying $a+b\equiv c+d$ mod $p-1$.  
We may assume that $a+b\ge c+d$. Then $a+b=c+d$ or $a+b=c+d+(p-1)$. 

First we consider the case when $a+b=c+d$. Put $w=a+b-3+2e$. 
If $b<d$, then $a>c$. Hence $a>c>d>b$ and  
put $\lambda_{1,v}:=(a_1,b_1;c_1)=(m_1,m_2;w+3)-\widetilde{\rho}=(m_1-2,m_2-1;w)$ where 
$m_1=c-b$ and $m_2=d-b$. 
If $b\ge d$, then $a\le c$. Hence $c\ge a>b\ge d$ and 
put $\lambda_{1,v}:=(a_1,b_1;c_1)=(m_1,m_2;w+3)-\widetilde{\rho}=(m_1-2,m_2-1;w)$ where 
$m_1=a-d$ and $m_2=b-d$. 
In either case, $\lambda_{1,v}\in X_1(T)$ if and only if $b\neq d$. 
Further in this case, $\lambda_{1,v}$ belongs to $\overline{C_0-\widetilde{\rho}}\cup 
\overline{C_1-\widetilde{\rho}}$.  
 
Next we consider the case when $a+b=c+d+(p-1)$. Put $c_1=c-1\ge 0$ and $d_1=d+p$ so that $a+b=c_1+d_1$. 
Clearly $a-(c-1)>b-(c-1)\ge 0$. Then we put 
$\lambda_{1,v}:=(a_1,b_1;c_1)=(m_1,m_2;w+3)-\widetilde{\rho}=(m_1-2,m_2-1;w)$ where 
$m_1=a-(c-1),\ m_2=b-(c-1)$, and $w=a+b-3+2e$. 
Then $\lambda_{1,v}\in X_1(T)$ if and only if $b\neq c-1$.

\begin{dfn}\label{gen-endoscopic}Assume $\lambda_{1,v}:=(m_1-2,m_2-1;w)\in X(T)$ associated to $\br|_{G_{F,v}}$ 
as above lies in $X_1(T)$. Then it is said to be generic 
\begin{enumerate}
\item if $\lambda_{1,v}\in C_1$ and $m_1+m_2>p+1$, or;  
\item if $\lambda_{1,v}\in C_0$. 
\end{enumerate}
\end{dfn}

%\subsubsection{Irreducible case}\label{IC}
%It follows from Section \ref{irr-d4} that 
%$\br|_{I_{F,v}}\simeq \omega^a_4\oplus \omega^{ap}_4\oplus \omega^{ap^2}_4\oplus \omega^{ap^3}_4$ 
%for $$a=a_0+a_1p+a_2p^2+a_3p^3,\ 0\le a_i\le p-1, 
%a\not\equiv 0\ {\rm mod}\ p^2+1,\ a\equiv 0\ {\rm mod}\ p+1.$$
%Since $a_0+a_2\equiv a_1+a_3$ mod $p-1$. The highest weight $\lambda_{1,v}$ is defined 
%similarly as the endoscopic case. 

\subsubsection{Definition for general case}\label{g-genericity}
For each $\br|_{G_{F,v}}$, we have associated $\lambda_{1,v}=(m_1-2,m_2-1;w)\in X(T)$ in Section \ref{BOC} through 
Section \ref{IC}. 
\begin{dfn}\label{geneneral-case}The weight $\lambda_{1,v}$ above is said to be generic of type $C_0$ or $C_1$ if 
$\lambda_{1,v}\in X_1(T)$ and further either of the followings is satisfied 
\begin{enumerate}
\item  when $\lambda_{1,v}\in C_1$,  $m_1+m_2>p+1$ and in addition, $m_1<p-1$ if $\br|_{G_{F,v}}$ is of 
Klingen ordinary type, or;  
\item  $\lambda_{1,v}\in C_0$. 
\end{enumerate}
\end{dfn}

\subsection{Some constituents of $W(\lambda)\otimes \F_v$}Let us keep the notation in Section \ref{Serre-weights}.  
The following lemma is useful and it follows 
from the contents in p.15 of \cite{til&her}:
\begin{lem}\label{constituent} Let $\lambda=(a,b;c)$ be an element in $X_1(T)$. It holds that 
\begin{enumerate}
\item $\widetilde{F}(\lambda)=W(\lambda)\otimes_{\O_E}\bF_p$ if $\lambda\in C_0$;
\item $0\lra \widetilde{F}(\lambda)\lra W(\lambda)\otimes_{\O_E}\bF_p\lra 
\widetilde{F}(\mu) \lra 0$ where $\mu=(p-b-3,p-a-3,c)$ if $\lambda\in C_1$;
\end{enumerate} 
\end{lem}

\subsection{A proof} We are now ready to prove Theorem \ref{main-thm3}. 
We follow the strategy in Section 5 of \cite{BGG-unitary}. 
Let $\br$ be in the claim. 
\begin{proof}$($Proof of Theorem \ref{main-thm3}$)$ Pick a weight 
$\widetilde{F}(\mu)=\otimes_{v|p}\widetilde{F}(\mu_v)\in W^{C_0\cup C_1}_{{\rm pd-cris},\mathcal{I}}(\br)$. Notice that the role of the notation 
for $\lambda$ and $\mu$ is converted and any confusion should not occur in the proof below. 
By definition, $\widetilde{F}(\mu)$ is a subquotient of $W(\lambda)\otimes\bF_p$ for some 
$\lambda\in C_0\cup C_1$. 
Then for each finite place $v$ of $F$ above $p$, there exists $\lambda_{1,v}\in X_1(T)$ 
such that 
$\lambda_{1,v}+\widetilde{\rho}\in C_0\cup C_1$ and $\otimes_{v|p}W(\lambda_{1,v})
\otimes\bF_p$ contains 
$\widetilde{F}(\mu)$ as a constituent. If $\lambda_{1,v}\in C_0$ for any $v$. By Lemma \ref{constituent}-(1), 
$\widetilde{F}(\mu)=\otimes_{v|p}W(\lambda_{1,v})\otimes\bF_p=\widetilde{F}(\lambda),\ \lambda=(\lambda_{1,v})_v$ and  
the claim follows from Theorem \ref{main-thm2}. 

Suppose that $\lambda_{1,v}\in C_1$ for some $v|p$. 
Further we assume that  $\widetilde{F}(\mu_v)$ appears in the quotient of 
$W(\lambda_{1,v})\otimes\bF_p$ in 
Lemma \ref{constituent}-(2) and try to deduce a contradiction. Then we will conclude 
$\widetilde{F}(\mu)$ is a unique submodule of $\otimes_{v|p}W(\lambda_{1,v})\otimes\bF_p$ with 
$\mu\in C_1$. Hence we have the claim by Theorem \ref{main-thm2} again. 
Let $v_\infty:F\lra \C$ be the embedding corresponding to $v$ via the fixed isomorphism $\iota_p:\bQ_p\simeq \C$. 

First we consider the case when $\br|_{G_{F,v}}$ is of Borel ordinary type. 
Recall from Section \ref{BOC} that 
\begin{equation}\label{BOC1}
\br|_{I_{F,v}}\simeq (\ve^{x+y}\oplus \ve^{x}\oplus \ve^{y}\oplus \textbf{1})\otimes \ve^{\delta},\ 
\delta=\frac{1}{2}(w+3-x-y) 
\end{equation}
for some integers $x,y,w$.  
By Lemma \ref{constituent}-(2) and Theorem \ref{main-thm2},  $\br$ has 
an automorphic lift $\pi$ whose the Harish-Chandra parameter at $v_\infty$ is $(p-y,p-x;w+3)$. 
In particular, 
$$HT(\rho_{\pi,\iota_p}|_{G_{F,v}})=\{\delta',\delta'+p-y,\delta'+p-x,\delta'+2p-(x+y)\},\ \delta'=\frac{1}{2}(w+3+x+y-2p).$$   
By genericity assumption, $x+y>p+1$. It follows from this that $\rho_{\pi,\iota_p}|_{G_{F,v}}$ is in 
the Fontaine-Laffaille range. By Proposition 3.3 of \cite{DZ}, we have 
\begin{equation}\label{BOC2}
\br|_{I_{F,v}}\simeq (\ve^{2p-(x+y)}\oplus \ve^{p-y}\oplus \ve^{p-y}\oplus \textbf{1})\otimes \ve^{\delta'},\ 
\delta'=\frac{1}{2}(w+3+x+y-2p)=\delta+(x+y)-p
\end{equation}
Recall that $p>2$. Hence $\ve^p=\ve\neq \textbf{1}$. It follows from this that 
(\ref{BOC1}) and (\ref{BOC2}) never coincide and it yields a contradiction. 

Next we consider the case when $\br|_{G_{F,v}}$ is of Siegel ordinary type.  
We have 
\begin{equation}\label{SOC1}
\br|_{I_{F,v}}\simeq (\ve^{x+y}\oplus (\omega^{x+y p}_2\oplus \omega^{y+x p}_2)\oplus \textbf{1})\otimes \ve^{\delta},\ \delta=\frac{1}{2}(w+3-x-y)\in \Z
\end{equation}
for some integers $0\le y<x\le p-1$ and $w\in \Z$. 
By Lemma \ref{constituent}-(2) and Theorem \ref{main-thm2},  $\br$ has 
an automorphic lift $\pi$ whose the Harish-Chandra parameter at $v_\infty$ is $(p-y,p-x;w+3)$. 
Similarly, by Theorem \ref{main-thm2} and Proposition 3.3 of \cite{DZ}, we have 
\begin{equation}\label{SOC2}
\br|_{I_{F,v}}\simeq (\ve^{2p-(x+y)}\oplus (\omega^{p-x+(p-y) p}_2\oplus \omega^{p-y+(p-x) p}_2)
\oplus \textbf{1})\otimes \ve^{\delta'} 
\end{equation}
or 
\begin{equation}\label{SOC3}
\br|_{I_{F,v}}\simeq (\ve^{p-y}\oplus (\omega^{p-x+(p-y) p}_2\oplus \omega^{p-y+(p-x) p}_2)
\oplus \ve^{p-y})\otimes \ve^{\delta'} 
\end{equation}
with $\delta'=\frac{1}{2}(w+3+x+y-2p)=\delta+(x+y)-p$. The equations (\ref{SOC1}), (\ref{SOC2}) and (\ref{SOC3}) give 
a contradiction since none of $\ve, \ve^y,\ \ve^x$ is trivial. 

Next we consider the case when $\br|_{G_{F,v}}$ is of Klingen ordinary type.  
In this case we have 
\begin{equation}\label{KOC1}
\br|_{I_{F,v}}\simeq (\omega^{x+y p}_2\oplus \omega^{y+x p}_2)\oplus 
((\omega^{-(x+y p)}_2\oplus \omega^{-(y+x p)}_2)\otimes \ve^{w+3})
\end{equation} 
for some integers $0\le y<x\le p-1$ and $0\le w \le p-2$. 
Recall $m_1=x+y-(w+3),\ m_2=x-y$. By Theorem \ref{main-thm2} and Proposition 3.3 of \cite{DZ}, 
we have 
\begin{equation}\label{KOC2}
\br|_{I_{F,v}}\simeq ((\omega^{p-m_1+(p-m_2) p}_2\oplus \omega^{p-m_2+(p-m_1) p}_2)\oplus 
(\omega^{2p-m_1-m_2}_2\oplus \omega^{p(2p-m_1-m_2)}_2))\otimes \ve^{\delta'} 
\end{equation} 
with $\delta'=\frac{1}{2}(w+3+m_1+m_2-2p)=-p+x$. 
Observe the component 
$$\omega^{2p-m_1-m_2}_2\otimes  \ve^{\delta'}=\omega^{w+2-x+(1+x)p}_2$$
which should also appear in (\ref{KOC1}). 
Notice that $\omega^{w+2-x+(1+x)p}_2\omega^{p(w+2-x)+(1+x)}_2=\ve^{w+3}$.
Compare it with the same thing for each component of (\ref{KOC1}), in either case, 
we have the congruence 
$$m_1=x+y-(w+3)\equiv 0\ {\rm mod}\ p-1$$  
which contradicts the genericity in Definition \ref{gen-Klingen}.

Finally we consider the case when $\br|_{G_{F,v}}$ is of endoscopic type. 
Then we have 
\begin{equation}\label{EC1}
\br|_{I_{F,v}}\simeq ((\omega^{x+y p}_2\oplus \omega^{y+x p}_2)\oplus 
(\omega^{c+d p}_2\oplus \omega^{c+d p}_2))\otimes \ve^{\delta},\ 
\delta=\frac{1}{2}(w+3-x-y) 
\end{equation} 
for some integers $0\le y<x\le p-1$ and $0\le d<c\le p-1$ satisfying $x+y\equiv c+d$ mod $p-1$. 
Recall $m_1=x-(c-1)$ and $m_2=y-(c-1)$ in Section \ref{EC} with 
$(x,y)=(a,b)$.   
We may assume that $x+y\ge c+d$. 
First we consider when $x+y=c+d$. In this case we see that $m_1+m_2<p$ and the corresponding the weight never belongs to $C_1$. 
Therefore, we may assume that $x+y=c+d+p-1$.   
By Theorem \ref{main-thm2} and Proposition 3.3 of \cite{DZ}, 
we have 
\begin{equation}\label{EC2}
\br|_{I_{F,v}}\simeq ((\omega^{p-m_1+(p-m_2) p}_2\oplus \omega^{p-m_2+(p-m_1) p}_2)\oplus 
(\omega^{2p-m_1-m_2}_2\oplus \omega^{p(2p-m_1-m_2)}_2))\otimes \ve^{\delta'} 
\end{equation} 
with $\delta'=\frac{1}{2}(w+3+m_1+m_2-2p)=\delta+x+y-(c-1)-p$. 
By computation, we see easily that 
$$(\omega^{p-m_1+(p-m_2) p}_2\oplus \omega^{p-m_2+(p-m_1) p}_2)\otimes \ve^{\delta'}=
(\omega^{x+yp}_2\oplus \omega^{y+x p}_2)\otimes \ve^\delta.$$ Therefore, we may observe a remaining component 
$$\omega^{2p-m_1-m_2}_2\otimes \ve^{\delta'}=\omega^{c-1+(d+1)p}_2\otimes \ve^\delta$$
which should be either $\omega^{c+d p}_2\otimes \ve^\delta$ or  $\omega^{d+c p}_2\otimes \ve^\delta$.  
This is impossible since $0\le c-1,d+1\le p-1$. 
\end{proof}

\section{Adequacy condition}\label{ade} In this Section, we briefly discuss the second condition in 
Theorem \ref{main-thm1} for 
an absolutely irreducible Galois representation $\br:G_F\lra {\rm GSp}_4(\bF_p)$ about when it holds that  
$\br|_{G_{F(\zeta_p)}}$ is irreducible and $\br(G_{F(\zeta_p)})$ is adequate.  
If $\br|_{G_{F(\zeta_p)}}$ is reducible, then $\br$ is an induced representation of a lower dimensional 
representation of $G_M$ for some $F\subset M\subset F(\zeta_p)$ with $[M:F]=2$ or $4$. 
Assume that $\br|_{G_{F(\zeta_p)}}$ is irreducible and $p>2\cdot 4+2=10$, then 
$\br(G_{F(\zeta_p)})$ is adequate by Appendix A of \cite{Thorne}. For small $p\le 7$, 
if the image of $\br$ contains ${\rm Sp}_4(\F_p)$ 
for $p>2$ or 
$\br\simeq {\rm Sym}^3\tau$ for some $\overline{\tau}:G_F\lra {\rm GL}_2(\bF_p)$ so that $\overline{\tau}$ contains ${\rm SL}_2(\F_p)$ for $p>5$ , then $\br(G_{F(\zeta_p)})$ is adequate. 
This follows immediately from Corollary 1.5 of \cite{GHT} and Proposition 2.1.1 of \cite{BGG2}.  
Since all subgroups in ${\rm GSp}_4(\bF_p)$ which act irreducibly on $\F^{\oplus 4}_p$ via 
the natural inclusion ${\rm GSp}_4(\bF_p)\subset {\rm GL}_4(\bF_p)$ are classified , for example, 
in \cite{Dieu}, we would be able to prove a similar result of Proposition 2.1.1 of \cite{BGG2}. 
We will address this problem somewhere. 
\section{Concluding remarks}
It may be possible to generalize the main results for reductive groups $\mathcal{G}$ in $GL_n$ in this paper once we have the following ingredients:
\begin{enumerate}
\item a complete classification of parahoric restrictions for $\mathcal{G}$; 
\item base change theory for $\mathcal{G}$;
\item a construction of a weight zero form for $\mathcal{G}$.
\end{enumerate}
%This may be possible to archive for which $\mathcal{G}$ is an inner form of $GL_3$ or $GL_4$ whose symmetric space is 
%Hermitian.   

\end{document}